\documentclass[aop,preprint]{imsart}

\usepackage{amsthm,amsmath,amsfonts,amssymb}
\usepackage[numbers]{natbib}

\startlocaldefs
\numberwithin{equation}{section}
\theoremstyle{plain}
\newtheorem{theorem}{Theorem}[section]
\newtheorem{proposition}[theorem]{Proposition}
\newtheorem{lemma}[theorem]{Lemma}
\newtheorem{corollary}[theorem]{Corollary}
\theoremstyle{definition}
\newtheorem{assumption}[theorem]{Assumption}
\newtheorem{definition}[theorem]{Definition}
\newtheorem{example}[theorem]{Example}
\newtheorem{remark}[theorem]{Remark}
\endlocaldefs

\begin{document}
\begin{frontmatter}
\title{Intrinsic Tangential Hadamard Differentiability of Rough-BSDE Solution Maps}
\runtitle{Intrinsic Differentiability of Rough-BSDE Solution Maps}

\begin{aug}
\author{\fnms{Yuhao}~\snm{Wang}}
\end{aug}

\begin{abstract}
We study first-order sensitivity of a scalar backward stochastic
differential equation with a deterministic rough driver.  The driver belongs
to the nonlinear space of step-two weakly geometric $p$-rough paths,
$2<p<3$, so an ordinary Banach-space difference quotient is not available.
At a fixed rough path $\mathbf x$, we use a weakly geometric,
finite-$(p,p/2)$-variation tensor realization of the Qian--Tudor tangent
structure,
represented intrinsically by a first-level direction $h$ and a compatible
second-level direction $\kappa$.  Admissible rough-path secants are required
to converge in a strong levelwise variation topology.  Under bounded smooth
rough vector fields, a bounded terminal condition, and a globally Lipschitz
generator with the stated smoothness assumptions, we construct a continuous
linear map
\[
 A_{\mathbf x}:\mathbb T_{\mathbf x}^p
 \longrightarrow
 \mathbb S^\infty\times\mathbb H^2_{\mathrm{BMO}}.
\]
The proof first establishes a uniform four-jet expansion for reset rough
flows along bounded realizations of full rough tangents.  A
Doss--Sussmann transformation transfers this expansion to a family of
quadratic generators.  Uniform BMO and reverse-H\"older estimates then yield
a local difference-quotient theorem, which is propagated over a fixed
deterministic partition and reconstructed in the original coordinates.  The
resulting derivative is independent of the joint lift, central
decomposition, and radial realization used in the proof.  Consequently,
solution
difference quotients converge to $A_{\mathbf x}(h,\kappa)$ for varying
directions and arbitrary admissible secants with strong levelwise variation
contact.  This is intrinsic tangential Hadamard differentiability on that
tensor-coordinate tangent class.  Locally bounded ray-homogeneous selections
give Fr\'echet-differentiable chart pullbacks at the parameter origin.  The
latter statement is chartwise; it is not Fr\'echet differentiability of the
solution map in the homogeneous rough-path metric.
\end{abstract}


\end{frontmatter}

\begin{center}
\textit{Preprint. Version 2026.07.29 (29 July 2026).}
\end{center}

\section{Introduction}
\label{sec:introduction}

The robust formulation of a stochastic equation asks that its solution vary
continuously with the enhanced driving signal.  For backward stochastic
differential equations with rough drivers, this program was carried out by
Diehl and Friz \cite{DiehlFriz2012}: a Doss--Sussmann transformation converts
the rough equation into a family of classical BSDEs with quadratic growth,
and stability of the transformed equations yields continuity in rough-path
topology.  Continuity is the correct threshold for well-posedness and for
Wong--Zakai type approximation.  It does not, however, identify the
first-order effect of a perturbation of the rough signal.  The purpose of this
paper is to obtain such an effect without replacing the rough-path domain by
a linear surrogate.

The distinction is already visible at step two.  If
$\mathbf x=(1,x,\mathbb x)$ is a weakly geometric $p$-rough path with
$2<p<3$, then a first-order perturbation is not specified by a path $h$
alone.  It also contains a second-level component $\kappa$ satisfying the
linearized Chen and symmetry identities at $\mathbf x$.  A central
second-level perturbation may change an RDE through a Lie-bracket drift even
when its first level is zero; see, for example, the rough-path translation
mechanism in \cite{FrizOberhauser2009}.  Discarding $\kappa$ would therefore
erase legitimate first-order directions.  It would also make the proposed
derivative depend on how a joint lift of $(x,h)$ had been chosen.

The area response is not only a bookkeeping effect.  At the zero base
driver, take two scalar vector fields $H_1(y)=\cos y$ and
$H_2(y)=\sin y$, and perturb only the $(1,2)$ central area at constant
speed.  Their bracket is the constant field one.  For a zero generator and
a constant terminal value $c$, the resulting solution is
$Y_t^\varepsilon=c-\varepsilon(1-t/T)$ and
$Z^\varepsilon=0$.  The response is exactly linear although the first-level
perturbation vanishes.  If the relevant vector fields commute, the same
area derivative is zero.  These two cases separate a genuine second-level
mechanism from a generic nondegeneracy claim.

This observation determines the domain of the derivative.  We work with the
curve-equivalence differential structure introduced by Qian and Tudor
\cite{QianTudor2011}, restricted to weakly geometric variation curves and
completed in the levelwise variation coordinates used here.  The resulting
space $\mathbb T_{\mathbf x}^p$ consists of pairs $\tau=(h,\kappa)$, where
$h$ has finite $p$-variation and $\kappa$ has finite $p/2$-variation,
subject to the two linearized geometric identities.  It is a Banach space
in the levelwise variation norm.  Young first-level perturbations together
with independent central area perturbations form a continuous linear
subspace.  They provide a useful prototype because the two mechanisms can
be seen separately; the theorem itself is stated for every pair satisfying
the two linearized identities.

The nonlinear ambient space creates a second problem.  The expression
$\mathbf x+\varepsilon\tau$ is not generally a rough path, and different
joint lifts can realize the same tangent pair.  We therefore formulate the
limit through genuine weakly geometric rough-path secants.  Their normalized
first and second tensor levels converge separately, in $p$-variation and
$p/2$-variation, respectively.  This strong levelwise contact is more
demanding than uniform tensor contact.  Its role is analytic: it permits a
uniform first-order remainder when both the tangent and the realizing secant
vary.  The geometry is handled by a quotient version of the
Lyons--Victoir extension theorem \cite{LyonsVictoir2007}.  Tangent balls admit
representatives with controlled lift budgets, and every admissible secant can
be recoded exactly as a bounded radial realization of a nearby tangent.
Neither construction requires a canonical or continuous choice of lift.

This normalization should be distinguished from Carnot scaling.  In the
tangent calculation, the first and second tensor increments are divided by
the same parameter, so that $\kappa$ remains first-order data.  This is the
linearization of the two geometric identities at the fixed base path.
Homogeneous rough-path metrics instead assign degree two to area.  The two
conventions answer different questions: the former describes tensor-level
sensitivity, whereas the latter measures ambient rough-path distance.  We
use the levelwise convention because it retains every direction in the
tensor-coordinate tangent class and
supports a representation-independent linear response.  The sequential
Hadamard formulation then tests that response against changing directions
and changing genuine rough paths, rather than only along one selected
coordinate curve.

There is also a stochastic obstruction.  Under the assumptions of this
paper, the original generator is globally Lipschitz.  After the
Doss--Sussmann transformation, the second spatial derivative of the rough
flow produces a term proportional to $|z|^2$.  The differentiability problem
is therefore a parameter problem for quadratic BSDEs, not for uniformly
Lipschitz BSDEs.  Existence of bounded scalar solutions follows from the
quadratic theory of \cite{Kobylanski2000}.  Parameter differentiability for
quadratic BSDEs is available in finite-dimensional settings under suitable
hypotheses \cite{AnkirchnerImkellerDosReis2007}, but it cannot be inserted as
a black box here.  The coefficients arise from rough-flow jets, the parameter
directions live in a rough tangent fibre, and the relevant estimates must be
uniform over changing secants.  The $Z$-coefficient in the linearized
equation has only BMO control.  Reverse-H\"older estimates and a fixed local
partition are thus part of the proof, rather than ancillary integrability
facts.

We now state the conclusion informally.  Let $\mathcal S(\mathbf x)$ denote
the Doss--Sussmann solution pair in
$\mathbb S^\infty\times\mathbb H^2_{\mathrm{BMO}}$.  At every fixed
$\mathbf x$ there is a continuous linear operator
\[
 A_{\mathbf x}:\mathbb T_{\mathbf x}^p
 \longrightarrow
 \mathbb S^\infty\times\mathbb H^2_{\mathrm{BMO}}
\]
with the following property.  Suppose $\varepsilon_n\to0$, suppose
$\tau_n\to\tau$ in the tensor-coordinate tangent class, and let $\mathbf x_n$ be any
weakly geometric secants having strong levelwise variation contact with
$(\varepsilon_n,\tau_n)$ at $\mathbf x$.  Then
\[
 \frac{\mathcal S(\mathbf x_n)-\mathcal S(\mathbf x)}{\varepsilon_n}
 \longrightarrow A_{\mathbf x}\tau
 \quad\text{in }\mathbb S^\infty\times\mathbb H^2_{\mathrm{BMO}}.
\]
The limit depends only on $(h,\kappa)$.  In particular, it is independent of
the joint lift, the split between cross and central second-level terms, and
the self-area inserted to obtain a genuine radial rough path.  The quantifier
over changing directions and changing genuine secants is the reason for the
term \emph{intrinsic tangential Hadamard differentiability}.  A fixed-curve
directional derivative would not imply this statement.

The derivative is most transparent after local transformation.  Fix a
deterministic reset interval and let $\phi$ be the backward rough flow that
absorbs the rough driver.  The first variation of the flow is first
constructed from a representative of $\tau$.  Its local Davie expansion can
then be written only in terms of $(x,\mathbb x)$ and $(h,\kappa)$.  The
linearized Chen identity cancels the potentially non-sewable terms, leaving a
three-point defect of order $3/p>1$.  Sewing gives an intrinsic variation,
and uniqueness shows that all representatives yield the same object.  This
step is also where linearity on the tensor-coordinate tangent class is
proved; it is not
deduced from a family of unrelated one-parameter limits.

The flow variation alone is insufficient for the BSDE.  The transformed
generator contains the flow, its first two spatial derivatives, inverse
Jacobians, and compositions with $f$.  We therefore establish a uniform
four-jet expansion on every interval of one fixed reset partition.  The
estimate is uniform on tangent balls after choosing representatives with a
bounded realization budget.  It yields a quadratic-weighted expansion for
the generator itself and the natural continuity bounds for its $y$- and
$z$-derivatives.  We do not impose an artificial second-order Taylor estimate
on all three coefficients; such an estimate would introduce stochastic
weights that are neither needed nor supplied by the BMO argument.

For the transformed BSDE, difference quotients satisfy linear equations with
a stochastic $Z$-coefficient.  Uniform BMO bounds give a common change of
measure and common reverse-H\"older exponents.  On a sufficiently short
interval, the remaining critical terms are absorbed in a conditional source
norm.  This proves convergence to the linearized transformed BSDE.  The
partition is selected from the base rough control and the uniform BMO budget
before the perturbation is chosen.  Fixed endpoints make the terminal
difference quotient on one interval equal to the value handed back from the
next interval.  Finite backward induction then yields a global derivative,
and differentiation of the reconstruction formulas returns to the original
coordinates.

Three points in this argument deserve emphasis.  First, representation
independence is an analytic conclusion.  Equality of Qian--Tudor tangent
classes identifies the desired first-order data, but it does not by itself
control solution difference quotients.  The intrinsic sewn flow variation,
the weighted generator expansion, and the local BSDE estimate supply that
control.  Second, uniformity is asserted on bounded selections of
representatives.  The derivative associated with any two fixed finite-budget
representatives is the same, but no estimate is claimed uniformly over
representatives whose self-area becomes unbounded.  Third, the second tensor
level is measured linearly in the tangent norm.  A pure central displacement
of size $\varepsilon$ has homogeneous rough-path distance of order
$|\varepsilon|^{1/2}$.  A nonzero linear response in the area chart is thus
not a Fr\'echet derivative in the homogeneous rough-path metric.

The RDE literature contains several distinct notions of sensitivity.  Friz
and Victoir \cite{FrizVictoir2010} study
differentiability along Young translations of a fixed rough path.  Coutin
and Lejay \cite{CoutinLejay2018} obtain H\"older and Lipschitz dependence of
the It\^o map on its parameters through an Omega-lemma argument, while
Bailleul \cite{Bailleul2015} proves Fr\'echet regularity in controlled-path
Banach spaces.  A Friz--Victoir Young translation corresponds here to the
$a=0$ part of the selected Young--central slice in
Corollary~\ref{cor:fixed-slices}.  The independent central coordinate in
that slice, and the general second-level coordinate $\kappa$ used in the
main theorem, are separate data.  These results are relevant to the
fixed-representation flow calculations used below.  The changing-secant
limit requires the additional uniform estimates developed here.

Qian and Tudor \cite{QianTudor2011} introduce the curve-equivalence
differential structure used as the starting point for
Section~\ref{sec:tangent-geometry}.  In this paper, that structure is
realized by weakly geometric, finite-$(p,p/2)$-variation tensors and
completed in the associated norm.  Geller and Lyons
\cite{GellerLyons2026} instead
construct a vector space $H^p(V)$ and an action
$\Omega_p(V)\times H^p(V)\to\Omega_p(V)$, written
$(\mathbf X,H)\mapsto\mathbf X\boxplus H$, for finite displacements of a
rough path.  No identification or inclusion between their displacement
space and the base-dependent tangent fibre is used here.  The Young--central
curves in
Corollary~\ref{cor:fixed-slices} are selected explicit slices; the main
theorem treats all tensor directions in
Definition~\ref{def:tangent-fibre} by joint lifts and strong levelwise
secants, without invoking the displacement action.

On the BSDE side, Diehl and Friz \cite{DiehlFriz2012} establish
well-posedness and rough-driver continuity in the scalar setting.  Eddahbi
and S\`ene \cite{EddahbiSene2014} consider quadratic rough-driver BSDEs with
an $L^2$ terminal condition, while Diehl and Zhang
\cite{DiehlZhang2017} treat continuous Young drifts by a direct fixed-point
argument.  Liang and Tang \cite{LiangTang2025} obtain multidimensional
results under smallness or componentwise structure and also treat linear,
random, time-varying rough drifts with square-integrable terminal data.  Li,
Zhang, and Zhang \cite{LiZhangZhang2026} develop reflected rough BSDEs and
a penalization approximation.  Recent preprints address nonlinear
space--time Young drivers \cite{SongZhangZhang2025I,SongZhangZhang2025II}
and discontinuous Young drivers with forward or Marcus jump conventions
\cite{BechererSun2025}.  The present paper instead studies the
representative-independent first-order response under changing genuine
secants, within the scalar, bounded-terminal, and smooth-coefficient setting
stated above.

The argument has four main components.  A concrete Banach model of the
selected tangent class provides quantitative bounded realizations and an
exact recoding of admissible secants.  Tangent sewing then gives the
intrinsic flow variation and a uniform four-jet remainder along full rough
tangents.  The resulting expansion passes to the Doss--Sussmann generator
with the weights required by quadratic BSDE analysis.  Common BMO and
reverse-H\"older bounds yield a local difference-quotient theorem, which is
propagated over fixed reset intervals to obtain the changing-secant Hadamard
limit.  For a locally bounded ray-homogeneous selection of realizations, the
associated chart pullback is also Fr\'echet differentiable at the parameter
origin.  This is a coordinate corollary, not a replacement for the intrinsic
theorem.

The chart statement has a practical but limited role.  It permits ordinary
linear perturbation calculus after a Banach parameter has been mapped into
the tangent fibre and a ray-homogeneous realization has been selected.  The
selection need not be canonical, and two selections can generate different
nonlinear charts.  Their first derivatives agree after identification with
$A_{\mathbf x}$ because the intrinsic derivative is representation
independent.  This is precisely the information needed for finite-dimensional
coordinate sensitivity or a subsequent delta method, but those applications
require their own state and estimation assumptions and are not asserted
here.

The scope is deliberate.  The solution $Y$ is scalar; the rough driver and
tangent directions are deterministic; the terminal condition is bounded and
does not vary with the scenario; $2<p<3$; and the rough vector fields satisfy
$C_b^9$ regularity.  We do not address a random rough driver, vector-valued
quadratic BSDEs, higher roughness levels, or differentiability in the
homogeneous rough metric.  These restrictions isolate the interaction that
the theorem is designed to resolve: full rough tangent geometry, uniform
flow jets, and quadratic-BSDE linearization.

The paper is organized as follows.  Section~\ref{sec:setting} fixes the
spaces, solution convention, and main statements.  Section~\ref{sec:strategy}
records the proof dependencies.  Section~\ref{sec:tangent-geometry} develops
the tangent and secant geometry.  Sections~\ref{sec:flow-jets} and
\ref{sec:generator} establish the flow-jet and transformed-generator
expansions.  Section~\ref{sec:local-bsde} proves local differentiability of
the transformed quadratic BSDE.  Section~\ref{sec:patching} performs the
fixed-endpoint propagation and reconstruction.  The intrinsic theorem is
proved in Section~\ref{sec:main-proof}, followed by the chart corollary and
structural examples.  Technical probability estimates are collected in the
appendix.

\section{Setting and main results}
\label{sec:setting}

\subsection{Probability space and solution norms}

Let $(\Omega,\mathcal F,\mathbb F,\mathbb P)$ satisfy the usual conditions,
where $\mathbb F=(\mathcal F_t)_{0\leq t\leq T}$ is the usual augmentation of
the natural filtration of an $\mathbb R^m$-valued Brownian motion $W$:
\begin{equation}\label{eq:brownian-filtration}
  \mathcal F_t
  =\sigma(W_s:0\leq s\leq t)\vee\mathcal N_{\mathbb P},
  \qquad 0\leq t\leq T.
\end{equation}
Thus $\mathbb F$ has the martingale representation property with respect to
$W$. In particular, no martingale orthogonal to $W$ is needed in the BSDEs
below.

We use the spaces
\begin{align}
 \mathbb S^\infty
 &:={\left\{Y:Y\text{ is continuous and adapted},\quad
       \left\|\sup_{0\leq t\leq T}|Y_t|\right\|_{L^\infty}<\infty\right\}},
       \label{eq:S-infinity}\\
 \mathbb H^2_{\mathrm{BMO}}
 &:={\left\{Z:Z\text{ is predictable and }Z\mathbin{\cdot}W
       \in\mathrm{BMO}_2\right\}}.\label{eq:H-BMO}
\end{align}
Their norms are
\begin{align}
 \|Y\|_{\mathbb S^\infty}
 &:=\left\|\sup_{0\leq t\leq T}|Y_t|\right\|_{L^\infty},\label{eq:S-norm}\\
 \|Z\|_{\mathbb H^2_{\mathrm{BMO}}}^2
 &:=\sup_{\tau\leq T}
   \left\|
     \mathbb E\left[\left.\int_\tau^T|Z_s|^2\,ds\right|\mathcal F_\tau\right]
   \right\|_{L^\infty},\label{eq:BMO-norm}
\end{align}
where the supremum is over all $[0,T]$-valued stopping times. We write
\begin{equation}\label{eq:solution-target}
 \mathcal X:=\mathbb S^\infty\times\mathbb H^2_{\mathrm{BMO}},
 \qquad
 \|(Y,Z)\|_{\mathcal X}
 :=\|Y\|_{\mathbb S^\infty}+\|Z\|_{\mathbb H^2_{\mathrm{BMO}}}.
\end{equation}

\subsection{Weakly geometric rough paths and tangent coordinates}

Set $E=\mathbb R^d$, fix $2<p<3$, and put
\begin{equation}\label{eq:q-definition}
 q:=\frac p2\in(1,3/2).
\end{equation}
All tensor spaces carry fixed finite-dimensional admissible norms. For a
two-index map $K$ on
$\Delta_T:=\{(s,t):0\leq s\leq t\leq T\}$, let
\begin{equation}\label{eq:delta-convention}
 (\delta K)_{s,u,t}:=K_{s,t}-K_{s,u}-K_{u,t}.
\end{equation}
For a Banach space $F$ and $r\geq1$, define
\begin{align}
 \mathcal V_0^r(F)
 &:={\left\{h\in C([0,T];F):h_0=0,
       \ \|h\|_{r\text{-var}}<\infty\right\}},\label{eq:based-var-space}\\
 \mathcal V_2^r(F)
 &:={\left\{K\in C(\Delta_T;F):K_{t,t}=0,
       \ \|K\|_{r\text{-var};2}<\infty\right\}},\label{eq:two-index-var-space}
\end{align}
where
\begin{equation}\label{eq:two-index-var-norm}
 \|K\|_{r\text{-var};2}
 :=\left(\sup_{D=(t_i)}\sum_i|K_{t_i,t_{i+1}}|^r\right)^{1/r}.
\end{equation}
The subscript $0$ in $\mathcal V_0^r$ means only that paths are based at
zero. It does not impose approximation by smooth paths.

A step-two weakly geometric $p$-rough path is written
\begin{equation}\label{eq:base-rough-path}
 \mathbf x=(1,x,\mathbb x)\in WG\Omega_p(E).
\end{equation}
Thus $x\in\mathcal V_0^p(E)$ and
$\mathbb x\in\mathcal V_2^q(E\otimes E)$ satisfy
\begin{align}
 (\delta\mathbb x)_{s,u,t}
 &=x_{s,u}\otimes x_{u,t},\label{eq:chen-base}\\
 \operatorname{Sym}\mathbb x_{s,t}
 &=\frac12x_{s,t}^{\otimes2},\label{eq:symmetry-base}
\end{align}
where $\operatorname{Sym}A=(A+A^\top)/2$ and $A^\top$ denotes the tensor
flip.

\begin{definition}[Tensor model of the tangent fibre]\label{def:tangent-fibre}
For $\mathbf x\in WG\Omega_p(E)$, let $\mathbb T_{\mathbf x}^p$ be the
space of pairs $\tau=(h,\kappa)$ such that
\begin{equation}\label{eq:tangent-components}
 h\in\mathcal V_0^p(E),
 \qquad
 \kappa\in\mathcal V_2^q(E\otimes E),
\end{equation}
and, for $0\leq s\leq u\leq t\leq T$,
\begin{align}
 (\delta\kappa)_{s,u,t}
 &=x_{s,u}\otimes h_{u,t}+h_{s,u}\otimes x_{u,t},
 \label{eq:linearized-chen}\\
 \operatorname{Sym}\kappa_{s,t}
 &=\frac12\bigl(x_{s,t}\otimes h_{s,t}
                   +h_{s,t}\otimes x_{s,t}\bigr).
 \label{eq:linearized-symmetry}
\end{align}
We equip this space with
\begin{equation}\label{eq:tangent-norm}
 \|\tau\|_{\mathbb T_{\mathbf x}^p}
 :=\|h\|_{p\text{-var}}+\|\kappa\|_{q\text{-var};2}.
\end{equation}
\end{definition}

The second component $\kappa$ is first-order tangent data. It is not
additive in general; its Chen defect is prescribed by
\eqref{eq:linearized-chen}. Section~\ref{sec:tangent-geometry} identifies
Definition~\ref{def:tangent-fibre} with the weakly geometric,
finite-$(p,q)$-variation Qian--Tudor class selected here and proves that it
is a Banach space.

\begin{definition}[Strong levelwise variation contact]
\label{def:strong-contact}
Fix $\mathbf x\in WG\Omega_p(E)$. Let $\varepsilon_n\to0$ with
$\varepsilon_n\neq0$, let
$\tau_n=(h_n,\kappa_n)\to\tau=(h,\kappa)$ in
$\mathbb T_{\mathbf x}^p$, and let
$\mathbf x_n=(1,x_n,\mathbb x_n)\in WG\Omega_p(E)$. We say that
$\mathbf x_n$ has strong levelwise variation contact with
$(\varepsilon_n,\tau_n)$ at $\mathbf x$ if
\begin{equation}\label{eq:strong-contact}
 \|x_n-x-\varepsilon_n h_n\|_{p\text{-var}}
 +\|\mathbb x_n-\mathbb x-\varepsilon_n\kappa_n\|_{q\text{-var};2}
 =o(|\varepsilon_n|).
\end{equation}
Such a sequence is called an admissible secant at $\mathbf x$.
\end{definition}

Condition~\eqref{eq:strong-contact} is imposed separately at the two tensor
levels. It is stronger than uniform tensor contact. This strength is used to
control the first-order remainder when both the tangent and the realizing
rough path vary with $n$.

\subsection{The rough BSDE and its solution convention}

The equation of interest is
\begin{equation}\label{eq:rough-bsde-formal}
 Y_t^{\mathbf x}
 =\xi+\int_t^T f(r,Y_r^{\mathbf x},Z_r^{\mathbf x})\,dr
   +\int_t^T H_i(Y_r^{\mathbf x})\,d\mathbf x_r^i
   -\int_t^T Z_r^{\mathbf x}\,dW_r.
\end{equation}
The rough term in \eqref{eq:rough-bsde-formal} is interpreted through the
following Doss--Sussmann construction. This convention does not require a
joint lift of $W$ and $\mathbf x$.

Let $I=[u,v]\subset[0,T]$ be a deterministic reset interval. Define the
backward rough flow
\begin{equation}\label{eq:backward-flow}
 \phi_t^{\mathbf x,I}(y)
 =y+\int_t^v H_i\bigl(\phi_r^{\mathbf x,I}(y)\bigr)\,d\mathbf x_r^i,
 \qquad \phi_v^{\mathbf x,I}(y)=y.
\end{equation}
Write $J=\partial_y\phi$ and $K=\partial_{yy}\phi$. On an interval where
$J$ is bounded away from zero, set
\begin{equation}\label{eq:transformed-generator}
 F^{\mathbf x,I}(t,y,z)
 :=\frac{1}{J_t(y)}
     f\bigl(t,\phi_t(y),J_t(y)z\bigr)
   +\frac12\frac{K_t(y)}{J_t(y)}|z|^2.
\end{equation}
For a bounded terminal value $\eta\in L^\infty(\mathcal F_v)$, solve
\begin{equation}\label{eq:transformed-bsde}
 \widetilde Y_t
 =\eta+\int_t^vF^{\mathbf x,I}
       (r,\widetilde Y_r,\widetilde Z_r)\,dr
       -\int_t^v\widetilde Z_r\,dW_r,
 \qquad t\in I,
\end{equation}
and reconstruct
\begin{equation}\label{eq:reconstruction}
 Y_t=\phi_t(\widetilde Y_t),
 \qquad
 Z_t=J_t(\widetilde Y_t)\widetilde Z_t.
\end{equation}
A preliminary finite deterministic partition with sufficiently small rough
control on each interval is used to obtain bounded local solutions and an
initial global bound.  At that stage an arbitrary bounded solution supplied
by the scalar existence theorem may be selected on each interval; the
preliminary construction is only an existence-and-bound witness.  In
Section~\ref{sec:patching} a further deterministic partition is chosen on
which the critical local uniqueness estimate holds.  Backward uniqueness
there removes both the preliminary solution selection and the reset
partition.  For $\mathbf x\in WG\Omega_p(E)$, the resulting well-defined
Doss--Sussmann solution is denoted by
\begin{equation}\label{eq:solution-map}
 \mathcal S(\mathbf x)
 :=(Y^{\mathbf x},Z^{\mathbf x})\in\mathcal X.
\end{equation}
The flow in \eqref{eq:backward-flow} is defined directly for weakly
geometric drivers. The definition does not use approximation in the same
$p$-variation topology by smooth lifts. On geometric drivers, this is the
Doss--Sussmann convention of \cite{DiehlFriz2012}; the direct
weakly geometric extension is the convention adopted here.

\begin{assumption}\label{ass:coefficients}
The following conditions hold.
\begin{enumerate}
 \item The terminal condition $\xi$ belongs to $L^\infty(\mathcal F_T)$
 and does not depend on the rough scenario.
 \item The solution $Y$ is scalar, and
 $H=(H_1,\ldots,H_d)$ satisfies $H_i\in C_b^9(\mathbb R)$.
 \item The generator
 \[
 f:\Omega\times[0,T]\times\mathbb R\times\mathbb R^m\longrightarrow\mathbb R
 \]
 is progressively measurable in $(\omega,t)$ and is $C^3$ in $(y,z)$.
 There are deterministic constants $L_0,L_1,L_2,L_3<\infty$ such that,
 almost surely,
 \begin{align}
  |f(t,0,0)|&\leq L_0,\label{eq:f-origin-bound}\\
  |f(t,y,z)-f(t,y',z')|
  &\leq L_1\bigl(|y-y'|+|z-z'|\bigr),\label{eq:f-lipschitz}\\
  \sup_{y,z}|D_{(y,z)}^k f(t,y,z)|&\leq L_k,
  \qquad k=2,3.\label{eq:f-higher-derivatives}
 \end{align}
 The second and third derivatives are uniformly continuous in $(y,z)$,
 uniformly in $(\omega,t)$ outside one null set.
 \item The base rough path, tangent directions, joint lifts, central
 residuals, and admissible secants are deterministic.
\end{enumerate}
Uniform estimates are stated on fixed bounded rough and tangent sets and
under a fixed bound on the realization budget introduced below.
\end{assumption}

The original generator $f$ is globally Lipschitz. The quadratic term in
\eqref{eq:transformed-generator} is produced by the Doss--Sussmann
transformation.

\subsection{Radial realizations and the main theorem}

Let $\tau=(h,\kappa)\in\mathbb T_{\mathbf x}^p$. A representation of
$\tau$ is a pair $(\mathbf Z,a)$ such that
\begin{equation}\label{eq:representative-data}
 \mathbf Z\in WG\Omega_p(E\oplus E),
 \qquad
 \pi_1(\mathbf Z)=\mathbf x,
 \qquad
 \pi_2(\mathbf Z)^1=h,
\end{equation}
and
\begin{equation}\label{eq:representative-split}
 \kappa=C_{\mathbf Z}+a,
 \qquad
 C_{\mathbf Z}:=\pi_{1,2}(\mathbf Z)+\pi_{2,1}(\mathbf Z),
\end{equation}
where $a$ is the increment map of a path in
$\mathcal V_0^q(\mathfrak{so}(E))$. We write
$\operatorname{Rep}_{\mathbf x}(\tau)$ for the collection of such pairs.
Let
\begin{equation}\label{eq:representative-budget}
 \mathfrak b(\mathbf Z,a)
 :=\rho_p(\mathbf Z)+\|a\|_{q\text{-var}},
\end{equation}
where $\rho_p$ is the homogeneous group-metric $p$-variation size. If
$\mathbb h_{\mathbf Z}:=\pi_2(\mathbf Z)^2$, define the radial realization
\begin{equation}\label{eq:radial-realization}
 \mathbf x^{\varepsilon;\mathbf Z,a}
 :=\left(
  1,
  x+\varepsilon h,
  \mathbb x+\varepsilon\kappa
       +\varepsilon^2\mathbb h_{\mathbf Z}
 \right).
\end{equation}
Section~\ref{sec:tangent-geometry} proves that
\eqref{eq:radial-realization} belongs to $WG\Omega_p(E)$ and that tangent
balls admit representations with uniformly bounded budgets.

The next theorem is the main statement of the paper. Its proof is deferred
to Sections 5--9, where the deterministic flow expansion, the transformed
quadratic-BSDE estimates, fixed-endpoint patching, and the Hadamard closure
are developed.

\begin{theorem}[Intrinsic tangential Hadamard differentiability]
\label{thm:intrinsic-hadamard}
Suppose Assumption~\ref{ass:coefficients} holds, and fix
$\mathbf x\in WG\Omega_p(E)$. There is a unique continuous linear map
\begin{equation}\label{eq:intrinsic-derivative}
 A_{\mathbf x}:\mathbb T_{\mathbf x}^p
 \longrightarrow\mathcal X
\end{equation}
with the following properties.

\begin{enumerate}
 \item Let
 \[
  \mathfrak r:\{\tau\in\mathbb T_{\mathbf x}^p:
                    \|\tau\|_{\mathbb T_{\mathbf x}^p}\leq1\}
  \longrightarrow\bigcup_\tau\operatorname{Rep}_{\mathbf x}(\tau)
 \]
 be any set-theoretic assignment such that
 $\mathfrak r(\tau)\in\operatorname{Rep}_{\mathbf x}(\tau)$ and
 \begin{equation}\label{eq:bounded-assignment}
  \sup_{\|\tau\|_{\mathbb T_{\mathbf x}^p}\leq1}
  \mathfrak b\bigl(\mathfrak r(\tau)\bigr)<\infty.
 \end{equation}
 Writing $\mathfrak r(\tau)=(\mathbf Z_\tau,a_\tau)$, one has
 \begin{equation}\label{eq:uniform-radial-expansion}
  \lim_{\varepsilon\to0}
  \sup_{\|\tau\|_{\mathbb T_{\mathbf x}^p}\leq1}
  \left\|
   \frac{\mathcal S(\mathbf x^{\varepsilon;\mathbf Z_\tau,a_\tau})
         -\mathcal S(\mathbf x)}{\varepsilon}
   -A_{\mathbf x}\tau
  \right\|_{\mathcal X}=0.
 \end{equation}

 \item Let $\varepsilon_n\to0$, $\varepsilon_n\neq0$, let
 $\tau_n\to\tau$ in $\mathbb T_{\mathbf x}^p$, and let
 $\mathbf x_n\in WG\Omega_p(E)$ have strong levelwise variation contact
 with $(\varepsilon_n,\tau_n)$ in the sense of
 Definition~\ref{def:strong-contact}. Then
 \begin{equation}\label{eq:hadamard-secant-limit}
  \frac{\mathcal S(\mathbf x_n)-\mathcal S(\mathbf x)}{\varepsilon_n}
  \longrightarrow A_{\mathbf x}\tau
  \qquad\text{in }\mathcal X.
 \end{equation}

 \item The value $A_{\mathbf x}\tau$ depends only on the tensor tangent
 $(h,\kappa)$. It is independent of the joint lift, the decomposition of
 $\kappa$ into cross and central parts, and the self-area used in a radial
 realization.
\end{enumerate}
\end{theorem}

If $A_{\mathbf x}\tau=(U^\tau,V^\tau)$, the pair is constructed by the
intrinsic reset-flow variation, a local transformed linear BSDE, backward
handoff at the fixed reset endpoints, and reconstruction through the inverse
Doss--Sussmann coordinates. This construction does not require a global
original-coordinate linear rough BSDE driven by an unspecified lift of
$(\mathbf x,h)$.

\begin{corollary}[Selected chart pullbacks]\label{cor:chart-pullback}
Let $B$ be a Banach space and let
$J:B\to\mathbb T_{\mathbf x}^p$ be continuous and linear. Suppose a
set-theoretic representation $\mathfrak r(Ju)=(\mathbf Z_{Ju},a_{Ju})$ is
defined for $u\in B$, is bounded in the budget
\eqref{eq:representative-budget} on bounded subsets of $B$, and is
ray-homogeneous:
\begin{equation}\label{eq:ray-homogeneous-selection}
 \mathbf Z_{\lambda\tau}=D_{1,\lambda*}\mathbf Z_\tau,
 \qquad
 a_{\lambda\tau}=\lambda a_\tau,
 \qquad
 D_{1,\lambda}(v,w)=(v,\lambda w).
\end{equation}
Define $\chi(0)=\mathbf x$ and, near zero,
\begin{equation}\label{eq:selected-chart}
 \chi(u):=\mathbf x^{1;\mathbf Z_{Ju},a_{Ju}}.
\end{equation}
Then the pullback $\mathcal S\circ\chi:B\to\mathcal X$ is Fr\'{e}chet
differentiable at zero, with
\begin{equation}\label{eq:chart-derivative}
 D(\mathcal S\circ\chi)(0)=A_{\mathbf x}\circ J.
\end{equation}
\end{corollary}

Corollary~\ref{cor:chart-pullback} concerns a selected local pullback. It
does not assert Fr\'{e}chet differentiability in the homogeneous rough-path
metric or between fixed-base controlled-path Banach models. The chart itself
may depend on the set-theoretic lift selection, although its first derivative
in \eqref{eq:chart-derivative} does not.
\section{Strategy of proof}
\label{sec:strategy}

The proof separates the geometry of the rough perturbation from the
stochastic stability argument.  This separation is needed because
\(WG\Omega_p(\mathbb R^d)\) is not a linear space.  In particular, the
expression \(\mathbf x+\varepsilon h\) does not specify a rough path when
\(h\) is only a first-level path, and it does not record the first-order
variation of the second level.

Fix \(2<p<3\), put \(q=p/2\), and let
\(\mathbf x=(1,x,\mathbb x)\in WG\Omega_p(\mathbb R^d)\).  An admissible
first-order displacement is described by a pair
\(\tau=(h,\kappa)\) in the tangent fibre
\(\mathbb T_{\mathbf x}^p\).  Thus \(h\) has finite \(p\)-variation,
\(\kappa\) has finite \(q\)-variation, and
\begin{align}
(\delta\kappa)_{s,u,t}
  &=x_{s,u}\otimes h_{u,t}
    +h_{s,u}\otimes x_{u,t},                                      \label{eq:strategy-linear-chen}\\
\operatorname{Sym}\kappa_{s,t}
  &=\frac12\bigl(
      x_{s,t}\otimes h_{s,t}
      +h_{s,t}\otimes x_{s,t}\bigr).                              \label{eq:strategy-linear-sym}
\end{align}
The first relation is the linearization of Chen's identity.  The second is
the linearization of weak-geometric symmetry.  Both are needed to identify
the derivative intrinsically.

The argument has four parts.

\subsection{The deterministic flow}
We first study the reset flow.  On each reset interval, it is encoded as a
Banach-space-valued RDE whose state contains the first three spatial
derivatives.  The reset endpoints
are selected from the base rough path and are kept fixed for every
perturbation in a bounded tangent set.  For each
\(\tau\in\mathbb T_{\mathbf x}^p\), Section~\ref{sec:flow-tangent} constructs
a path \(P^\tau\) by an intrinsic Davie increment.  Its source terms involve
\(h\) and \(\kappa\), rather than a chosen joint lift.  The central
calculation is a cancellation between Chen's identity, the linearized Chen
identity \eqref{eq:strategy-linear-chen}, and the second-order terms in the
flow expansion.  The remaining three-point defect has order \(3/p>1\), so
it can be sewn.

This construction gives a continuous linear map
\[
  \mathbb T_{\mathbf x}^p\longrightarrow
  \mathcal V_p(I;\mathcal E_{\mathrm{jet}}),\qquad
  \tau\longmapsto P_I^\tau ,
\]
on every reset interval \(I\).  The path \(P_I^\tau\) is independent of the
joint lift used to construct it.  This first-order statement is not, by
itself, a uniform Taylor estimate for the nonlinear flow.  The latter
requires a separate expansion of the flow, inverse flow, and all spatial
jets used by the transformed generator.

\subsection{The transformed coefficients}
The uniform nonlinear flow-jet expansion is inserted into the
Doss--Sussmann formula for the transformed generator.  The estimates are
made only on the bounded strip reached by the transformed BSDE.  This
produces a continuous linear source
\(\tau\mapsto\dot F[\tau]\), together with the weighted remainders for
\(F\), \(F_y\), and \(F_z\).  The weights retain the linear and quadratic
\(z\)-growth generated by the transformation.  An unweighted expansion on
all of \(\mathbb R\times\mathbb R^m\) is neither used nor asserted.

\subsection{The stochastic estimate}
On a fixed reset interval, the transformed equations are scalar quadratic
BSDEs.  Uniform \(\mathbb S^\infty\) and
\(\mathbb H^2_{\mathrm{BMO}}\) estimates provide a common range for the
solutions and a common BMO bound for their martingale parts.  The
difference quotient is compared with a linear BSDE whose source is
\(\dot F[\tau]\).  Reverse-H\"older estimates are used with exponents
determined by the available BMO bounds.  No \(L^2\) reverse-H\"older exponent
is imposed without the corresponding smallness assumption.

\subsection{Globalization and the intrinsic limit}
The local derivative is propagated backward over the fixed reset
partition.  At each endpoint, the derivative of the preceding terminal
condition is the left-end value obtained on the next interval.  The
original variables are then recovered through the inverse
Doss--Sussmann map.  Finally, exact recoding of admissible rough-path
secants reduces a varying-secant limit to the uniform radial estimates.
This yields a continuous linear operator
\[
 A_{\mathbf x}:\mathbb T_{\mathbf x}^p
 \longrightarrow
 \mathbb S^\infty\times\mathbb H^2_{\mathrm{BMO}} .
\]
The last step uses strong levelwise variation contact.  It does not convert
the result into Fr\'echet differentiability for the homogeneous rough-path
metric.

All uniform constants are recorded on bounded sets.  They may depend on
\[
 p,d,m,T,\ \|\mathbf x\|_{p\text{-var}},\
 \|H\|_{C_b^9},\ \|\xi\|_\infty,
\]
the bounds on the derivatives of the original generator, the tangent-ball
radius, and the quantitative realization budget.  They do not depend on
the sign of the perturbation parameter, the particular tangent in the
fixed ball, or the representative used to display its first variation.
\section{Qian--Tudor tangent fibres and admissible secants}
\label{sec:tangent-geometry}

The rough-path domain is not a linear space, and adding an arbitrary pair of
first- and second-level increments does not preserve either Chen's relation
or weak-geometric symmetry. This section gives a tensor-coordinate tangent
space, constructs bounded radial realizations, and recodes every admissible
secant as an exact radial path. The analytic derivative is not used here.

\subsection{Identification of the tangent fibre}

Let $T_{\mathbf x}^{\mathrm{QT,wg};p,q}WG\Omega_p(E)$ denote the subclass
of Qian--Tudor curve-equivalence classes \cite{QianTudor2011} that admit
weakly geometric variation curves whose first and second tensor
derivatives have finite $p$- and $q$-variation, respectively.  This notation
records the weak-geometric and variation-topology choices made in this
paper; it is not attributed as a separate Banach completion to Qian and
Tudor.  If
$\mathbf Z\in WG\Omega_p(E\oplus E)$ has first marginal $\mathbf x$, set
\begin{equation}\label{eq:cross-sum}
 C_{\mathbf Z}:=\pi_{1,2}(\mathbf Z)+\pi_{2,1}(\mathbf Z).
\end{equation}

\begin{proposition}[Tensor-coordinate realization of the selected class]
\label{prop:tangent-identification}
The map
\begin{equation}\label{eq:qt-coordinate-map}
 \Psi_{\mathbf x}:
 T_{\mathbf x}^{\mathrm{QT,wg};p,q}WG\Omega_p(E)
 \longrightarrow\mathbb T_{\mathbf x}^p,
 \qquad
 [\mathbf Z,\varphi]_{\mathbf x}
 \longmapsto
 \left(
  \pi_2(\mathbf Z)^1,
  C_{\mathbf Z}+\varphi
 \right),
\end{equation}
is a linear bijection.  On this selected class, the restriction of the
Qian--Tudor tangent metric $\widetilde d_p$ satisfies
\begin{equation}\label{eq:exact-tangent-metric}
 \widetilde d_p(\tau,\widetilde\tau)
 =\max\left\{
   \|h-\widetilde h\|_{p\text{-var}},
   \|\kappa-\widetilde\kappa\|_{q\text{-var};2}
  \right\}.
\end{equation}
Consequently,
\begin{equation}\label{eq:tangent-metric-equivalence}
 \widetilde d_p(\tau,0)
 \leq\|\tau\|_{\mathbb T_{\mathbf x}^p}
 \leq2\widetilde d_p(\tau,0),
\end{equation}
and $\mathbb T_{\mathbf x}^p$ is a Banach space.
\end{proposition}

\begin{proof}
Put $h=\pi_2(\mathbf Z)^1$ and
$\kappa=C_{\mathbf Z}+\varphi$. The two cross blocks satisfy
\begin{align}
 \delta\pi_{1,2}(\mathbf Z)_{s,u,t}
 &=x_{s,u}\otimes h_{u,t},\label{eq:cross-chen-one}\\
 \delta\pi_{2,1}(\mathbf Z)_{s,u,t}
 &=h_{s,u}\otimes x_{u,t}.
 \label{eq:cross-chen-two}
\end{align}
The residual $\varphi$ is additive. Hence
\eqref{eq:linearized-chen} holds.

The symmetry condition requires one qualification.  Qian and Tudor's
curve-equivalence construction is the ambient differential structure.  We
select its weakly geometric subclass; differentiating the group-like
symmetry of every curve in that subclass gives
\begin{equation}\label{eq:curve-symmetry}
 \operatorname{Sym}\bigl(\mathbf x^\varepsilon\bigr)^2_{s,t}
 =\frac12\bigl(x_{s,t}+\varepsilon h_{s,t}\bigr)^{\otimes2}.
\end{equation}
Differentiating \eqref{eq:curve-symmetry} at zero gives
\eqref{eq:linearized-symmetry}. Thus the symmetry constraint in
Definition~\ref{def:tangent-fibre} is forced by admissibility; it is not an
extra smoothness condition.

Conversely, let $(h,\kappa)\in\mathbb T_{\mathbf x}^p$. Consider the
quotient of $G^2(E_x\oplus E_h)$ that removes the $h$-self brackets and the
$x$--$h$ cross brackets. This quotient is isomorphic to
$G^2(E_x)\times E_h$. Since its kernel is central and $p>2$, the
Lyons--Victoir extension \cite{LyonsVictoir2007} gives
\begin{equation}\label{eq:joint-lift-existence}
 \mathbf Z\in WG\Omega_p(E\oplus E),
 \qquad
 \pi_1(\mathbf Z)=\mathbf x,
 \qquad
 \pi_2(\mathbf Z)^1=h.
\end{equation}
Set $a:=\kappa-C_{\mathbf Z}$. Equations
\eqref{eq:linearized-chen}, \eqref{eq:cross-chen-one}, and
\eqref{eq:cross-chen-two} imply $\delta a=0$. Joint weak-geometric symmetry
and \eqref{eq:linearized-symmetry} imply $\operatorname{Sym}a=0$. Hence
$a$ is the increment map of a path in
$\mathcal V_0^q(\mathfrak{so}(E))$, and $(\mathbf Z,a)$ represents
$(h,\kappa)$. This proves surjectivity. Injectivity is the defining
Qian--Tudor equivalence, which identifies representatives with the same
$h$ and $C_{\mathbf Z}+\varphi$.

Rewriting the restricted Qian--Tudor metric in these two coordinates gives
\eqref{eq:exact-tangent-metric}, and
\eqref{eq:tangent-metric-equivalence} follows.
It remains to verify completeness, which is not part of that metric
identification. The spaces $\mathcal V_0^p(E)$ and
$\mathcal V_2^q(E\otimes E)$ are complete in their variation norms. Indeed,
a variation-norm Cauchy sequence is uniformly Cauchy; passage to the limit
on each finite partition preserves the Cauchy bound, after which the
supremum over partitions gives convergence in variation. The relations
\eqref{eq:linearized-chen} and \eqref{eq:linearized-symmetry} are closed
under this convergence. Therefore $\mathbb T_{\mathbf x}^p$ is a closed
linear subspace of the Banach product
$\mathcal V_0^p(E)\times\mathcal V_2^q(E\otimes E)$.
\end{proof}

\begin{remark}[The Young--central subspace]\label{rem:young-central}
Let $C_0^{r\text{-var}}(F)$ denote the closure of smooth based paths in the
$r$-variation norm. If
\[
 h\in C_0^{q\text{-var}}(E),
 \qquad
 a\in C_0^{q\text{-var}}(\mathfrak{so}(E)),
\]
Young integration defines
\begin{equation}\label{eq:young-cross-term}
 \mathcal C^{x,h}_{s,t}
 :=\int_s^t x_{s,r}\otimes dh_r
   +\int_s^t h_{s,r}\otimes dx_r.
\end{equation}
The map
\begin{equation}\label{eq:young-central-embedding}
 (h,a)\longmapsto(h,\mathcal C^{x,h}+a)
\end{equation}
is continuous and linear into $\mathbb T_{\mathbf x}^p$, because
$1/p+1/q=3/p>1$.  No surjectivity of this selected embedding is used or
claimed; the theorem is formulated on all pairs in
Definition~\ref{def:tangent-fibre}.
\end{remark}

\subsection{Quantitative lifts and bounded representations}

For a based path $\mathbf Z$ in a finite-dimensional homogeneous group $G$,
write
\begin{equation}\label{eq:group-variation-size}
 \rho_p(\mathbf Z)
 :=\left(\sup_{D=(t_i)}
   \sum_i\|\mathbf Z_{t_i,t_{i+1}}\|_G^p\right)^{1/p}.
\end{equation}

\begin{lemma}[Quantitative quotient lift]\label{lem:quantitative-lift}
Let $G$ be a fixed step-two normed Carnot group, and let
$K\subset\exp(W_2)$ be a fixed closed central subgroup. If $y$ is a based
$G/K$-valued path of finite $p$-variation, then there is a lift
$\widetilde y$ such that
\begin{equation}\label{eq:quantitative-lift}
 \pi(\widetilde y)=y,
 \qquad
 \rho_p(\widetilde y)
 \leq C_{\mathrm{LV}}\rho_p(y),
\end{equation}
where $C_{\mathrm{LV}}$ depends only on $p$, $G$, $K$, and the chosen
homogeneous norms.
\end{lemma}

\begin{proof}
This is the quantitative corollary extracted from the estimates in
Lyons--Victoir, Proposition 6, Lemma 11, and Lemma 13
\cite{LyonsVictoir2007}. Their homogeneous
section satisfies
\begin{equation}\label{eq:homogeneous-section-bound}
 \|\bar g\|_{G/K}\leq\|i(\bar g)\|_G
 \leq2\|\bar g\|_{G/K}.
\end{equation}
After the variation-to-H\"older time change, the step-two dyadic recurrence
in the proof is
\begin{equation}\label{eq:dyadic-recurrence}
 a_{m+1}
 \leq2^{1/p-1/2}a_m+2^{2+1/p-1/2}C.
\end{equation}
Since $p>2$, the coefficient $2^{1/p-1/2}$ is strictly smaller than one, so
\begin{equation}\label{eq:dyadic-bound}
 \sup_m a_m
 \leq
 \frac{2^{2+1/p-1/2}}{1-2^{1/p-1/2}}\,C.
\end{equation}
The dyadic extension estimate and reversal of the time change give
\eqref{eq:quantitative-lift}. The qualitative extension theorem alone does
not state this bound.
\end{proof}

For every second-level block $B$ of a joint lift, fixed homogeneous and
tensor norms give a constant $C_{\mathrm{blk}}$ such that
\begin{equation}\label{eq:block-bound}
 \|B\|_{q\text{-var};2}
 \leq C_{\mathrm{blk}}\rho_p(\mathbf Z)^2.
\end{equation}

\begin{lemma}[Bounded realizations of tangent balls]
\label{lem:bounded-realization}
Fix $\mathbf x\in WG\Omega_p(E)$ and $R<\infty$. There is a constant
\begin{equation}\label{eq:bounded-realization-constant}
 C_R=C_R\bigl(p,d,\text{tensor and homogeneous norms},
                   \rho_p(\mathbf x),R\bigr)<\infty
\end{equation}
such that every $\tau\in\mathbb T_{\mathbf x}^p$ with
$\|\tau\|_{\mathbb T_{\mathbf x}^p}\leq R$ has a representation
$(\mathbf Z_\tau,a_\tau)$ satisfying
\begin{equation}\label{eq:bounded-realization}
 \rho_p(\mathbf Z_\tau)+\|a_\tau\|_{q\text{-var}}
 \leq C_R.
\end{equation}
The representations may be selected to be homogeneous on each tangent ray.
No continuity, measurability, canonicity, or additivity of the selection is
asserted.
\end{lemma}

\begin{proof}
Apply Lemma~\ref{lem:quantitative-lift} to the quotient path
$(\mathbf x,h)$ used in \eqref{eq:joint-lift-existence}. The resulting joint
lift can be chosen so that
\begin{equation}\label{eq:joint-lift-bound}
 \rho_p(\mathbf Z_\tau)
 \leq C_{\mathrm{LV}}
 \bigl(\rho_p(\mathbf x)+c_{\mathrm{quot}}\|h\|_{p\text{-var}}\bigr).
\end{equation}
Set $a_\tau=\kappa-C_{\mathbf Z_\tau}$. It is additive and antisymmetric by
the proof of Proposition~\ref{prop:tangent-identification}. The block bound
\eqref{eq:block-bound} gives
\begin{align}
 \|a_\tau\|_{q\text{-var}}
 &\leq\|\kappa\|_{q\text{-var};2}
      +2C_{\mathrm{blk}}\rho_p(\mathbf Z_\tau)^2\notag\\
 &\leq R+2C_{\mathrm{blk}}C_{\mathrm{LV}}^2
  \bigl(\rho_p(\mathbf x)+c_{\mathrm{quot}}R\bigr)^2.
 \label{eq:central-residual-bound}
\end{align}
This proves \eqref{eq:bounded-realization}.

For ray homogeneity, select one unit vector $u$ on each one-dimensional
real subspace of $\mathbb T_{\mathbf x}^p$ and choose one bounded
representation $(\mathbf Z_u,a_u)$. If $\tau=\lambda u$, set
\begin{equation}\label{eq:ray-realization}
 \mathbf Z_{\lambda u}:=D_{1,\lambda*}\mathbf Z_u,
 \qquad
 a_{\lambda u}:=\lambda a_u.
\end{equation}
The first auxiliary level and the cross blocks scale by $\lambda$, the
self-area scales by $\lambda^2$, and the central correction scales by
$\lambda$. Thus the budget is bounded on every tangent ball. This is a
set-theoretic construction.
\end{proof}

\subsection{Radial curves and representation boundaries}

\begin{lemma}[Exact radial realization]\label{lem:exact-radial}
If $(\mathbf Z,a)\in\operatorname{Rep}_{\mathbf x}(h,\kappa)$, then
$\mathbf x^{\varepsilon;\mathbf Z,a}$ defined in
\eqref{eq:radial-realization} belongs to $WG\Omega_p(E)$ for every
$\varepsilon\in\mathbb R$. Its level-one and level-two first derivatives at
zero are $h$ and $\kappa$, respectively, in the levelwise variation sense.
\end{lemma}

\begin{proof}
The joint Chen identities give the cross terms in
\eqref{eq:linearized-chen}; the auxiliary marginal gives the quadratic
term $\varepsilon^2\mathbb h_{\mathbf Z}$. Substitution in Chen's relation
for \eqref{eq:radial-realization} yields
\[
 \delta\bigl(\mathbf x^{\varepsilon;\mathbf Z,a}\bigr)^2
 =(x+\varepsilon h)\otimes(x+\varepsilon h).
\]
The corresponding symmetric identity follows from
\eqref{eq:linearized-symmetry} and
$\operatorname{Sym}\mathbb h_{\mathbf Z}=h^{\otimes2}/2$. Finite variation
is immediate from the defining norms.
\end{proof}

If $(\mathbf Z,a)$ and $(\widetilde{\mathbf Z},\widetilde a)$ represent the
same tangent, their radial curves have identical first levels and
\begin{equation}\label{eq:radial-representative-difference}
 \bigl(\mathbf x^{\varepsilon;\mathbf Z,a}\bigr)^2
 -\bigl(\mathbf x^{\varepsilon;\widetilde{\mathbf Z},\widetilde a}\bigr)^2
 =\varepsilon^2
   \bigl(\mathbb h_{\mathbf Z}-\mathbb h_{\widetilde{\mathbf Z}}\bigr).
\end{equation}
Equation~\eqref{eq:radial-representative-difference} does not by itself
prove analytic representation independence for changing, unbounded
representatives. The proof of Theorem~\ref{thm:intrinsic-hadamard} uses
uniqueness of the intrinsic tangent equation and states uniform remainders
only under bounded realization budgets.

\subsection{Exact recoding of admissible secants}

We first fix a ray-homogeneous self-lift
\begin{equation}\label{eq:self-lift}
 \mathcal L(h)=(1,h,\mathbb L(h))\in WG\Omega_p(E)
\end{equation}
obtained from Lemma~\ref{lem:quantitative-lift}, with
\begin{equation}\label{eq:self-lift-bound}
 \mathcal L(\lambda h)=\delta_\lambda\mathcal L(h),
 \qquad
 \|\mathbb L(h)\|_{q\text{-var};2}
 \leq C_{\mathrm{self}}\|h\|_{p\text{-var}}^2.
\end{equation}
The choice is not required to be continuous.

\begin{lemma}[Exact recoding of admissible secants]
\label{lem:secant-recoding}
Let $\varepsilon_n\to0$, $\varepsilon_n\neq0$, let
$\tau_n=(h_n,\kappa_n)\to\tau$ in $\mathbb T_{\mathbf x}^p$, and let
$\mathbf x_n=(1,x_n,\mathbb x_n)\in WG\Omega_p(E)$ satisfy
\eqref{eq:strong-contact}. Define
\begin{align}
 h_n^*&:=\frac{x_n-x}{\varepsilon_n},
 \label{eq:recoded-first-level}\\
 \kappa_n^*&:=
 \frac{\mathbb x_n-\mathbb x
       -\varepsilon_n^2\mathbb L(h_n^*)}{\varepsilon_n}.
 \label{eq:recoded-second-level}
\end{align}
Then
\begin{equation}\label{eq:recoded-tangent-convergence}
 \tau_n^*:=(h_n^*,\kappa_n^*)
 \in\mathbb T_{\mathbf x}^p,
 \qquad
 \tau_n^*\longrightarrow\tau.
\end{equation}
There are representations
$(\mathbf Z_n^*,a_n^*)\in\operatorname{Rep}_{\mathbf x}(\tau_n^*)$ such
that
\begin{equation}\label{eq:recoded-budget}
 \sup_n\mathfrak b(\mathbf Z_n^*,a_n^*)<\infty
\end{equation}
and
\begin{equation}\label{eq:exact-secant-identity}
 \mathbf x_n
 =\mathbf x^{\varepsilon_n;\mathbf Z_n^*,a_n^*}
\end{equation}
for every $n$.
\end{lemma}

\begin{proof}
Since $x_n=x+\varepsilon_nh_n^*$, Chen's relation for $\mathbf x_n$ and
$\mathbf x$ gives
\begin{align}
 \delta(\mathbb x_n-\mathbb x)_{s,u,t}
 ={}&\varepsilon_n\bigl(
  x_{s,u}\otimes(h_n^*)_{u,t}
  +(h_n^*)_{s,u}\otimes x_{u,t}\bigr)\notag\\
 &+\varepsilon_n^2(h_n^*)_{s,u}\otimes(h_n^*)_{u,t}.
 \label{eq:secant-chen-expansion}
\end{align}
Because
$\delta\mathbb L(h_n^*)=h_n^*\otimes h_n^*$,
equation \eqref{eq:recoded-second-level} implies
\begin{equation}\label{eq:recoded-chen}
 \delta\kappa_n^*
 =x\otimes h_n^*+h_n^*\otimes x.
\end{equation}
The step-two symmetry of $\mathbf x_n$, $\mathbf x$, and
$\mathcal L(h_n^*)$ similarly gives
\begin{equation}\label{eq:recoded-symmetry}
 \operatorname{Sym}\kappa_n^*
 =\frac12\bigl(x\otimes h_n^*+h_n^*\otimes x\bigr).
\end{equation}
Thus $\tau_n^*\in\mathbb T_{\mathbf x}^p$.

Strong contact gives
\begin{equation}\label{eq:recoded-first-convergence}
 \|h_n^*-h_n\|_{p\text{-var}}=o(1),
\end{equation}
so $(h_n^*)$ is bounded. By
\eqref{eq:strong-contact}, \eqref{eq:recoded-second-level}, and
\eqref{eq:self-lift-bound},
\begin{align}
 \|\kappa_n^*-\kappa_n\|_{q\text{-var};2}
 \leq{}&
 \frac{\|\mathbb x_n-\mathbb x
       -\varepsilon_n\kappa_n\|_{q\text{-var};2}}{|\varepsilon_n|}
 +|\varepsilon_n|\|\mathbb L(h_n^*)\|_{q\text{-var};2}
 \longrightarrow0.
 \label{eq:recoded-second-convergence}
\end{align}
This proves \eqref{eq:recoded-tangent-convergence}.

Let $K_{\mathrm{cross}}$ be the closed central subgroup generated by the
cross brackets in $G^2(E\oplus E)$. Then
\begin{equation}\label{eq:cross-quotient}
 G^2(E\oplus E)/K_{\mathrm{cross}}
 \simeq G^2(E)\times G^2(E).
\end{equation}
Apply Lemma~\ref{lem:quantitative-lift} to the quotient path
$(\mathbf x,\mathcal L(h_n^*))$. It yields a joint lift $\mathbf Z_n^*$
that preserves both marginals:
\begin{equation}\label{eq:recoded-marginals}
 \pi_1(\mathbf Z_n^*)=\mathbf x,
 \qquad
 \pi_2(\mathbf Z_n^*)=\mathcal L(h_n^*),
 \qquad
 \sup_n\rho_p(\mathbf Z_n^*)<\infty.
\end{equation}
Define
\begin{equation}\label{eq:recoded-central-residual}
 a_n^*:=\kappa_n^*-C_{\mathbf Z_n^*}.
\end{equation}
The proof of Proposition~\ref{prop:tangent-identification} shows that
$a_n^*$ is additive and antisymmetric. Equations
\eqref{eq:block-bound}, \eqref{eq:recoded-tangent-convergence}, and
\eqref{eq:recoded-marginals} give
$\sup_n\|a_n^*\|_{q\text{-var}}<\infty$, which proves
\eqref{eq:recoded-budget}. Finally,
\begin{align*}
 \bigl(\mathbf x^{\varepsilon_n;\mathbf Z_n^*,a_n^*}\bigr)^1
 &=x+\varepsilon_nh_n^*=x_n,\\
 \bigl(\mathbf x^{\varepsilon_n;\mathbf Z_n^*,a_n^*}\bigr)^2
 &=\mathbb x+\varepsilon_n\kappa_n^*
   +\varepsilon_n^2\mathbb L(h_n^*)
 =\mathbb x_n.
\end{align*}
This proves \eqref{eq:exact-secant-identity}.
\end{proof}

\begin{corollary}[Equivalent radial form of strong contact]
\label{cor:radial-contact}
Every admissible secant admits the exact representation
\eqref{eq:exact-secant-identity} with
$\tau_n^*\to\tau$ and uniformly bounded realization budgets. Conversely,
if $\tau_n\to\tau$ and
$\mathbf x_n=\mathbf x^{\varepsilon_n;\mathbf Z_n,a_n}$ for a uniformly
bounded sequence of representations, then $\mathbf x_n$ has strong
levelwise variation contact with $(\varepsilon_n,\tau_n)$.
\end{corollary}

\begin{proof}
The first assertion is Lemma~\ref{lem:secant-recoding}. For the converse,
the first-level remainder is zero, while the second-level remainder is
$\varepsilon_n^2\mathbb h_{\mathbf Z_n}$. The block estimate
\eqref{eq:block-bound} and the uniform budget give
\[
 \frac{\|\varepsilon_n^2\mathbb h_{\mathbf Z_n}\|_{q\text{-var};2}}
      {|\varepsilon_n|}
 \longrightarrow0.
\]
\end{proof}
\section{Uniform flow-jet expansion along full rough tangents}
\label{sec:flow-tangent}
\label{sec:flow-jets}

The obstruction in this section is representation dependence.  A tangent
\(\tau=(h,\kappa)\) can be displayed by several joint lifts of \((x,h)\).
The cross areas and the self-area of \(h\) can change with that display.
The first variation must nevertheless depend only on \((h,\kappa)\).  We
first construct this intrinsic variation and prove its uniqueness.  The
uniform nonlinear remainder needed later is stated separately at the end
of the section.

\subsection{Common reset intervals and the augmented flow}
\label{subsec:common-reset}

Let \(I=[u,v]\subset[0,T]\).  The backward reset flow is
\begin{equation}
 \phi_t^{\mathbf x,I}(y)
 =y+\int_t^v H_i\bigl(\phi_r^{\mathbf x,I}(y)\bigr)
      \,d\mathbf x_r^i,
 \qquad \phi_v^{\mathbf x,I}(y)=y.                                \label{eq:backward-reset-flow}
\end{equation}
It is defined directly for weakly geometric rough paths.  We use a time
reversal when invoking forward RDE estimates and retain the backward
notation in the displayed equations.

Write
\[
 J=\partial_y\phi,\qquad
 K=\partial_{yy}\phi,\qquad
 L=\partial_{yyy}\phi,
\]
and set
\[
 \mathcal E_{\mathrm{jet}}=C_b(\mathbb R;\mathbb R^4),\qquad
 \Theta=(\phi-\operatorname{Id},J-1,K,L).
\]
For a \(\mathcal E_{\mathrm{jet}}\)-valued path \(G\), set
\[
 \|G\|_{\mathcal V_p(I;\mathcal E_{\mathrm{jet}})}
 =\sup_{t\in I}\|G_t\|_{\mathcal E_{\mathrm{jet}}}
  +\|G\|_{p\text{-var};I;\mathcal E_{\mathrm{jet}}}.
\]
For a physical four-jet
\(\mathfrak J=(\mathfrak J_0,\ldots,\mathfrak J_3)\), we also use
\[
 \|\mathfrak J\|_{\mathfrak V_3^\Delta(I)}
 =\sum_{\ell=0}^3\sup_{y\in\mathbb R}
 \left\{
   \|\mathfrak J_\ell(\cdot,y)\|_{\infty;I}
   +\|\mathfrak J_\ell(\cdot,y)\|_{p\text{-var};I}
 \right\}.
\]
After applying a cutoff outside the region reached by the physical jets,
the augmented equation takes the form
\begin{equation}
 d\Theta_t=V_i(\Theta_t)\,d\mathbf x_t^i,\qquad \Theta_v=0.          \label{eq:augmented-reset-rde}
\end{equation}
Under \(H_i\in C_b^9(\mathbb R)\), the cutoff Nemytskii fields have more
than the
\(C_b^3(\mathcal E_{\mathrm{jet}};\mathcal E_{\mathrm{jet}})\)
regularity used below.  On an interval on which the
base rough control is sufficiently small, the cutoff solution stays on the
physical jet graph.  The four components then coincide with
\[
 \bigl(\phi-\operatorname{Id},
       \partial_y\phi-1,\partial_{yy}\phi,\partial_{yyy}\phi\bigr).
\]

Choose a control \(\omega_{\mathbf x}\) such that
\[
 |x_{s,t}|\le \omega_{\mathbf x}(s,t)^{1/p},
 \qquad
 |\mathbb x_{s,t}|\le \omega_{\mathbf x}(s,t)^{2/p}.
\]
The Davie expansion of \eqref{eq:augmented-reset-rde} is
\begin{align}
 \Theta_{s,t}
   &=V_i(\Theta_s)x_{s,t}^i
     +Q_{j,i}(\Theta_s)\mathbb x_{s,t}^{j,i}
     +\Theta^\natural_{s,t},                                      \label{eq:base-davie}\\
 |\Theta^\natural_{s,t}|
   &\le C_\Theta\omega_{\mathbf x}(s,t)^{3/p},                     \label{eq:base-davie-rem}
\end{align}
where
\[
 A_i=DV_i,\qquad
 Q_{j,i}=DV_iV_j,\qquad
 B_{j,i}=DQ_{j,i}.
\]

We next fix the reset endpoints.  Consider radial realizations
\begin{equation}
 \begin{split}
   ( \mathbf x^{\varepsilon,\tau})^1
      &=x+\varepsilon h,\\
   ( \mathbf x^{\varepsilon,\tau})^2
      &=\mathbb x+\varepsilon\kappa
        +\varepsilon^2\mathbb h,
 \end{split}                                                       \label{eq:radial-realization-flow}
\end{equation}
where \(\mathbb h\) is the self-area supplied by an admissible
representative.  On bounded tangent and realization sets, a control for
\(\mathbf x^{\varepsilon,\tau}\) satisfies
\begin{equation}
 \omega_\varepsilon(s,t)
 \le C\left\{
   \omega_{\mathbf x}(s,t)
   +|\varepsilon|^p\omega_h(s,t)
   +|\varepsilon|^q\omega_\kappa(s,t)
   +|\varepsilon|^p\omega_{\mathbb h}(s,t)
 \right\}.                                                        \label{eq:perturbed-control}
\end{equation}
Choose a finite partition \(\Pi(\mathbf x)\) so that the base term in
\eqref{eq:perturbed-control} is below one half of the local reset
threshold on each interval.  The other controls have uniformly bounded
total mass.  Hence one \(\varepsilon_0>0\) makes their combined
contribution smaller than the remaining half for every
\(|\varepsilon|\le\varepsilon_0\) and every tangent in the fixed ball.
The endpoints in \(\Pi(\mathbf x)\) therefore do not depend on
\(\varepsilon\), \(\tau\), or the concentration profile of the directional
variation.

\subsection{The intrinsic Davie increment}
\label{subsec:intrinsic-davie}

For a one-index path \(g\) and a two-index map \(k\), let
\[
 \|g\|_{r\text{-var};I}
  =\left(\sup_{\pi\subset I}
       \sum_{[s,t]\in\pi}|g_{s,t}|^r\right)^{1/r},
 \qquad
 \|k\|_{r\text{-var};I}
  =\left(\sup_{\pi\subset I}
       \sum_{[s,t]\in\pi}|k_{s,t}|^r\right)^{1/r}.
\]
The second definition does not require additivity.  The tangent fibre is
\begin{equation}
 \mathbb T_{\mathbf x}^p
 =\left\{(h,\kappa):
 \begin{array}{l}
 h_0=0,\quad \|h\|_{p\text{-var}}<\infty,\quad
 \|\kappa\|_{q\text{-var}}<\infty,\quad \kappa_{t,t}=0,\\
 (\delta\kappa)_{s,u,t}
 =x_{s,u}\otimes h_{u,t}+h_{s,u}\otimes x_{u,t},\\
 \operatorname{Sym}\kappa_{s,t}
 =\frac12(x_{s,t}\otimes h_{s,t}
          +h_{s,t}\otimes x_{s,t})
 \end{array}\right\},                                             \label{eq:tangent-fibre-flow}
\end{equation}
with norm
\[
 \|(h,\kappa)\|_{\mathbb T_{\mathbf x}^p}
 =\|h\|_{p\text{-var}}+\|\kappa\|_{q\text{-var}}.
\]

For \(\tau=(h,\kappa)\in\mathbb T_{\mathbf x}^p\), define
\begin{equation}
 \begin{split}
 \Xi^\tau_{s,t}(P_s)
 ={}&A_i(\Theta_s)P_sx_{s,t}^i
     +V_i(\Theta_s)h_{s,t}^i\\
    &+B_{j,i}(\Theta_s)P_s\mathbb x_{s,t}^{j,i}
     +Q_{j,i}(\Theta_s)\kappa_{s,t}^{j,i}.
 \end{split}                                                       \label{eq:intrinsic-xi}
\end{equation}
A continuous \(\mathcal E_{\mathrm{jet}}\)-valued path \(P\), with
\(P_v=0\), is an intrinsic
tangent solution if
\begin{equation}
 P_{s,t}=\Xi^\tau_{s,t}(P_s)+P^\natural_{s,t},\qquad
 |P^\natural_{s,t}|\le C_\tau\omega(s,t)^{3/p}                    \label{eq:intrinsic-solution}
\end{equation}
for a finite control \(\omega\).  Its first-order controlled remainder is
\begin{equation}
 R^P_{s,t}
 =P_{s,t}-A_i(\Theta_s)P_sx_{s,t}^i
          -V_i(\Theta_s)h_{s,t}^i.                                \label{eq:controlled-rem-P}
\end{equation}

\begin{proposition}[Intrinsic reset-flow variation]
\label{prop:intrinsic-flow-variation}
Assume
\[
 V_i\in
 C_b^3(\mathcal E_{\mathrm{jet}};\mathcal E_{\mathrm{jet}}).
\]
For every
\[
 \tau=(h,\kappa)\in\mathbb T_{\mathbf x}^p
\]
there is a unique intrinsic tangent solution \(P^\tau\).  The map
\[
  \mathbb T_{\mathbf x}^p
  \longrightarrow\mathcal V_p(I;\mathcal E_{\mathrm{jet}}),
  \qquad \tau\longmapsto P^\tau ,
\]
is continuous and linear.  More precisely,
\begin{equation}
 \|P^\tau\|_{\mathcal V_p(I;\mathcal E_{\mathrm{jet}})}
 +\|R^{P^\tau}\|_{q\text{-var};I;\mathcal E_{\mathrm{jet}}}
 \le C_I\|\tau\|_{\mathbb T_{\mathbf x}^p}.                        \label{eq:intrinsic-flow-bound}
\end{equation}
Every finite-norm joint-lift representative of \(\tau\) yields the same
path \(P^\tau\).  On a bounded tangent set, \(C_I\) can be chosen uniformly
over \(I\in\Pi(\mathbf x)\) and over representatives with the prescribed
quantitative realization budget.
\end{proposition}

\subsection{Cancellation of the three-point defect}
\label{subsec:defect-cancellation}

We give the cancellation because it is the step that uses the complete
two-level tangent.  Put
\[
 r=\|\tau\|_{\mathbb T_{\mathbf x}^p}.
\]
The case \(r=0\) is immediate.  Otherwise set
\(\bar h=h/r\), \(\bar\kappa=\kappa/r\), and use the control
\begin{equation}
 \omega(s,t)=\omega_{\mathbf x}(s,t)
  +\|\bar h\|_{p\text{-var};[s,t]}^p
  +\|\bar\kappa\|_{q\text{-var};[s,t]}^q.                          \label{eq:normalized-control}
\end{equation}
Then
\begin{equation}
 |h_{s,t}|\le r\omega(s,t)^{1/p},\qquad
 |\kappa_{s,t}|\le r\omega(s,t)^{2/p}.                             \label{eq:tangent-control}
\end{equation}

Fix \(s<u<t\), and write a superscript \(s\) or \(u\) for evaluation at
\(\Theta_s\) or \(\Theta_u\).  By additivity of \(x\) and \(h\), Chen's
identity for \(\mathbb x\), and
\eqref{eq:strategy-linear-chen},
\begin{align}
 \delta\Xi^\tau(P)_{s,u,t}
 :={}&\Xi^\tau_{s,t}(P_s)
       -\Xi^\tau_{s,u}(P_s)-\Xi^\tau_{u,t}(P_u) \notag\\
 ={}&(A_i^sP_s-A_i^uP_u)x_{u,t}^i
     +(V_i^s-V_i^u)h_{u,t}^i                                      \notag\\
 &+(B_{j,i}^sP_s-B_{j,i}^uP_u)\mathbb x_{u,t}^{j,i}
     +(Q_{j,i}^s-Q_{j,i}^u)\kappa_{u,t}^{j,i}                      \notag\\
 &+B_{j,i}^sP_sx_{s,u}^jx_{u,t}^i
     +Q_{j,i}^s\bigl(
       x_{s,u}^jh_{u,t}^i+h_{s,u}^jx_{u,t}^i\bigr).                \label{eq:delta-xi-full}
\end{align}

Suppose first that \(P\) has the first-order controlled form
\begin{equation}
 P_{s,u}
 =A_k^sP_sx_{s,u}^k+V_k^sh_{s,u}^k+\rho^P_{s,u},
 \qquad
 |\rho^P_{s,u}|\le K_P\omega(s,u)^{2/p}.                           \label{eq:P-controlled-form}
\end{equation}
Taylor's formula and \eqref{eq:base-davie} give
\begin{align}
 A_i^uP_u-A_i^sP_s
   &=B_{k,i}^sP_sx_{s,u}^k
     +Q_{k,i}^sh_{s,u}^k+\rho^{AP,i}_{s,u},                        \label{eq:AP-expansion}\\
 V_i^u-V_i^s
   &=Q_{k,i}^sx_{s,u}^k+\rho^{V,i}_{s,u},                          \label{eq:V-expansion}
\end{align}
where
\begin{align}
 |\rho^{AP,i}_{s,u}|
   &\le C\bigl(\|P\|_{\infty;I}+K_P+r\bigr)
             \omega(s,u)^{2/p},                                  \label{eq:AP-rem}\\
 |\rho^{V,i}_{s,u}|
   &\le C\omega(s,u)^{2/p}.                                       \label{eq:V-rem}
\end{align}
The identity behind \eqref{eq:AP-expansion} is
\[
 DA_i[V_k]P+A_iA_kP=D(DV_iV_k)P=DQ_{k,i}P.
\]
The remaining coefficient increments satisfy
\begin{align}
 |B_{j,i}^uP_u-B_{j,i}^sP_s|
   &\le C\bigl(\|P\|_{\infty;I}+K_P+r\bigr)
             \omega(s,u)^{1/p},                                  \label{eq:BP-bound}\\
 |Q_{j,i}^u-Q_{j,i}^s|
   &\le C\omega(s,u)^{1/p}.                                       \label{eq:Q-bound}
\end{align}

Substitution of \eqref{eq:AP-expansion} and
\eqref{eq:V-expansion} into \eqref{eq:delta-xi-full} leaves the following
terms at order \(2/p\):
\begin{align*}
 &-B_{j,i}^sP_sx_{s,u}^jx_{u,t}^i
  -Q_{j,i}^sh_{s,u}^jx_{u,t}^i
  -Q_{j,i}^sx_{s,u}^jh_{u,t}^i\\
 &\quad
  +B_{j,i}^sP_sx_{s,u}^jx_{u,t}^i
  +Q_{j,i}^s\bigl(
      x_{s,u}^jh_{u,t}^i+h_{s,u}^jx_{u,t}^i\bigr)=0.
\end{align*}
The first cancellation comes from Chen's identity for \(\mathbb x\).  The
other two come from the linearized Chen identity for \(\kappa\).  It follows
from \eqref{eq:tangent-control}, \eqref{eq:AP-rem},
\eqref{eq:V-rem}, \eqref{eq:BP-bound}, and \eqref{eq:Q-bound} that
\begin{equation}
 |\delta\Xi^\tau(P)_{s,u,t}|
 \le C\bigl(\|P\|_{\infty;I}+K_P+r\bigr)
        \omega(s,t)^{3/p}.                                       \label{eq:delta-xi-bound}
\end{equation}
Since \(2<p<3\),
\[
 \frac3p>1,
\]
and the remaining defect is in the sewing regime.  The symmetry condition
in \eqref{eq:strategy-linear-sym} is not used in this three-point
cancellation.  Its role is to remove the symmetric part of the residual
between two joint-lift representations, as shown below.

\subsection{Existence from a representative}
\label{subsec:representative-existence}

Choose one finite-norm joint-lift representative of
\(\tau=(h,\kappa)\).  Write it as
\[
 \mathbf Z\in WG\Omega_p(\mathbb R^d\oplus\mathbb R^d),
 \qquad \pi_x\mathbf Z=\mathbf x,\qquad
 (\pi_h\mathbf Z)^1=h,
\]
together with an additive antisymmetric path \(a\) of finite
\(q\)-variation.  Let
\begin{equation}
 C_{s,t}^{j,i}
 =\mathbf Z_{s,t}^{x_j,h_i}
  +\mathbf Z_{s,t}^{h_j,x_i}.                               \label{eq:flow-cross-sum}
\end{equation}
The representative is chosen so that
\begin{equation}
 \kappa=C+a.                                                       \label{eq:kappa-cross-central}
\end{equation}
The associated radial rough path has levels
\[
 x+\varepsilon h,\qquad
 \mathbb x+\varepsilon(C+a)+\varepsilon^2\mathbb h_{\mathbf Z}.
\]

For this fixed joint driver, solve the mixed rough/Young equation
\begin{equation}
 \boxed{
 dP_t
 =A_i(\Theta_t)P_t\,d\mathbf Z_t^{x,i}
  +V_i(\Theta_t)\,d\mathbf Z_t^{h,i}
  +\sum_{i<j}[V_i,V_j](\Theta_t)\,da_t^{ij},
 \qquad P_v=0 .
 }                                                               \label{eq:representative-variation}
\end{equation}
The last integral is a Young integral because
\[
 \frac1p+\frac1q=\frac3p>1.
\]
The two rough integrals in \eqref{eq:representative-variation} are defined
with the same joint lift \(\mathbf Z\).  There is no stand-alone rough
integral against a general finite-\(p\)-variation path \(h\).

To identify the first variation, one may equivalently consider the
fixed-driver family
\begin{equation}
 \begin{split}
 d\Theta_t^\varepsilon
  ={}&V_i(\Theta_t^\varepsilon)\,d\mathbf Z_t^{x,i}
    +\varepsilon V_i(\Theta_t^\varepsilon)\,d\mathbf Z_t^{h,i}\\
   &+\varepsilon\sum_{i<j}
     [V_i,V_j](\Theta_t^\varepsilon)\,da_t^{ij},
 \qquad \Theta_v^\varepsilon=0 .
 \end{split}                                                       \label{eq:fixed-driver-family}
\end{equation}
The fixed-driver It\^o--Lyons regularity theorem
\cite{Bailleul2015} applies to this affine vector-field parameter.  Its
derivative at zero solves
\eqref{eq:representative-variation}.  This statement differentiates the
vector fields while holding \(\mathbf Z\) fixed.  It is not a claim of
Fr\'echet differentiability with respect to the rough driver.

On \(\mathcal E_{\mathrm{jet}}\times\mathcal E_{\mathrm{jet}}\), define
\[
 \widehat V_i^x(\theta,p)
   =(V_i(\theta),A_i(\theta)p),\qquad
 \widehat V_i^h(\theta,p)
   =(0,V_i(\theta)).
\]
Their second-level actions on the \(P\)-component are
\begin{align}
 \operatorname{pr}_P\!\left(
 D\widehat V_i^x\,\widehat V_j^x\right)
   &=B_{j,i}(\Theta)P,&
 \operatorname{pr}_P\!\left(
 D\widehat V_i^h\,\widehat V_j^x\right)
   &=Q_{j,i}(\Theta),                                             \label{eq:extended-fields-one}\\
 \operatorname{pr}_P\!\left(
 D\widehat V_i^x\,\widehat V_j^h\right)
   &=Q_{j,i}(\Theta),&
 \operatorname{pr}_P\!\left(
 D\widehat V_i^h\,\widehat V_j^h\right)
   &=0.                                                           \label{eq:extended-fields-two}
\end{align}
Thus the self-area of \(h\) has no first-order effect, and the two cross
blocks enter only through the sum \(C\).  The Davie expansion of
\eqref{eq:representative-variation} is
\begin{align}
 P_{s,t}
 ={}&A_i^sP_sx_{s,t}^i+V_i^sh_{s,t}^i
     +B_{j,i}^sP_s\mathbb x_{s,t}^{j,i}
     +Q_{j,i}^sC_{s,t}^{j,i} \notag\\
 &+\sum_{i<j}[V_i,V_j](\Theta_s)a_{s,t}^{ij}
     +P^{\natural,\mathbf Z,a}_{s,t},                             \label{eq:representative-davie}
\end{align}
with a \(3/p\)-order remainder for the fixed representative.

\subsection{Representative independence}
\label{subsec:representative-independence}

The cross-block Chen identities imply
\[
 (\delta C)_{s,u,t}
 =x_{s,u}\otimes h_{u,t}
  +h_{s,u}\otimes x_{u,t}.
\]
Comparison with \eqref{eq:strategy-linear-chen} shows that
\(a=\kappa-C\) is additive.  Weak-geometric symmetry of the joint lift gives
\[
 \operatorname{Sym}C_{s,t}
 =\frac12\bigl(
 x_{s,t}\otimes h_{s,t}
 +h_{s,t}\otimes x_{s,t}\bigr).
\]
Equation \eqref{eq:strategy-linear-sym} therefore implies
\(\operatorname{Sym}a=0\).

We use the convention
\[
 a_{s,t}
 =\sum_{i<j}a_{s,t}^{ij}
   (e_i\otimes e_j-e_j\otimes e_i),
 \qquad
 [V_i,V_j]=DV_jV_i-DV_iV_j.
\]
With \(Q_{j,i}=DV_iV_j\), direct contraction gives
\begin{equation}
 Q_{j,i}(\Theta_s)a_{s,t}^{j,i}
 =\sum_{i<j}[V_i,V_j](\Theta_s)a_{s,t}^{ij}.                       \label{eq:bracket-contraction}
\end{equation}
There is no factor \(1/2\) under this normalization.  It follows that
\[
 Q_{j,i}^sC_{s,t}^{j,i}
 +\sum_{i<j}[V_i,V_j](\Theta_s)a_{s,t}^{ij}
 =Q_{j,i}^s\kappa_{s,t}^{j,i}.
\]
Hence \eqref{eq:representative-davie} has the intrinsic increment
\eqref{eq:intrinsic-xi}.  This proves existence of an intrinsic tangent
solution.

\subsection{Uniqueness, estimates, and linearity}
\label{subsec:intrinsic-uniqueness}

The representative solution has a finite first-order controlled norm, so
\eqref{eq:delta-xi-bound} applies to its intrinsic increment.  The sewing
lemma from controlled rough integration \cite{Gubinelli2004} gives a unique
additive increment \(\mathcal I\Xi^\tau(P)\) such that
\begin{equation}
 \left|
   (\mathcal I\Xi^\tau(P))_{s,t}
   -\Xi^\tau_{s,t}(P_s)
 \right|
 \le C\bigl(\|P\|_{\infty;I}+K_P+r\bigr)
       \omega(s,t)^{3/p}.                                        \label{eq:sewn-increment}
\end{equation}
The increments of the representative solution are additive and satisfy the
same estimate.  They therefore coincide with
\(\mathcal I\Xi^\tau(P)\).  The exact defect cancellation above and the
uniqueness argument below are proved in this section rather than imported
from that reference.

The sewing estimate also closes the local a priori bound.  On
\([s,t]\subset I\), define
\[
 \begin{split}
 \|P\|_{\star;[s,t]}
 ={}&\sup_{r\in[s,t]}|P_r|
  +\sup_{a<b\in[s,t]}
       \frac{|P_{a,b}|}{\omega(a,b)^{1/p}}\\
 &+\sup_{a<b\in[s,t]}
       \frac{|R^P_{a,b}|}{\omega(a,b)^{2/p}} .
 \end{split}
\]
For \(\omega(s,t)\le\rho\), equations
\eqref{eq:intrinsic-xi}, \eqref{eq:delta-xi-bound}, and
\eqref{eq:sewn-increment} give
\begin{equation}
 \|P\|_{\star;[s,t]}
 \le C\bigl(|P_t|+r\bigr)
      +C\rho^{1/p}\|P\|_{\star;[s,t]}.                             \label{eq:star-absorption}
\end{equation}
After decreasing \(\rho\), the last term is absorbed.  In particular,
\begin{equation}
 |P^\natural_{s,t}|\le Cr\omega(s,t)^{3/p},\qquad
 |R^P_{s,t}|\le Cr\omega(s,t)^{2/p}                               \label{eq:intrinsic-remainders}
\end{equation}
on each small-control interval.  A finite greedy partition extends these
estimates to \(I\).

Let \(P\) and \(\widetilde P\) be two intrinsic solutions, and put
\(D=P-\widetilde P\).  Their source terms cancel:
\begin{equation}
 D_{s,t}
 =A_i^sD_sx_{s,t}^i
  +B_{j,i}^sD_s\mathbb x_{s,t}^{j,i}
  +D^\natural_{s,t},\qquad
 |D^\natural_{s,t}|\le C_D\omega(s,t)^{3/p}.                       \label{eq:homogeneous-difference}
\end{equation}
For a partition \(\pi\) of \([s,t]\),
\[
 \sum_{[r,r']\in\pi}|D^\natural_{r,r'}|
 \le C_D
 \left(\max_{[r,r']\in\pi}\omega(r,r')\right)^{3/p-1}
 \omega(s,t),
\]
which tends to zero with the \(\omega\)-mesh.  Equation
\eqref{eq:homogeneous-difference} therefore identifies \(D\) with the
solution of the homogeneous linear RDE
\[
 dD_t=A_i(\Theta_t)D_t\,d\mathbf x_t^i,\qquad D_v=0.
\]
On an interval with \(\omega(s,t)\le\rho\), the linear rough estimate yields
\[
 \sup_{r\in[s,t]}|D_r|
 \le |D_t|+C\rho^{1/p}\sup_{r\in[s,t]}|D_r|.
\]
Choose \(\rho\) so that the last term is absorbed.  Backward iteration over
a finite greedy partition gives \(D=0\).

The absorbed estimate \eqref{eq:star-absorption}, with the sources retained,
gives
\begin{align}
 &\sup_{r\in[s,t]}|P_r|
  +\|P\|_{p\text{-var};[s,t]}
  +\|R^P\|_{q\text{-var};[s,t]} \notag\\
 &\hspace{35mm}\le C\bigl(|P_t|+r\bigr)                            \label{eq:local-intrinsic-estimate}
\end{align}
on every small-control interval.  Backward iteration proves
\eqref{eq:intrinsic-flow-bound}.  Normalization by \(r\) in
\eqref{eq:normalized-control} prevents the constant from acquiring an
additional dependence on the size of the direction.

If \(\tau_1,\tau_2\in\mathbb T_{\mathbf x}^p\) and
\(\alpha,\beta\in\mathbb R\), then
\(\alpha P^{\tau_1}+\beta P^{\tau_2}\) satisfies the intrinsic equation
with source \(\alpha\tau_1+\beta\tau_2\).  Uniqueness gives
\[
 P^{\alpha\tau_1+\beta\tau_2}
 =\alpha P^{\tau_1}+\beta P^{\tau_2}.
\]
Applying \eqref{eq:intrinsic-flow-bound} to
\(\tau_1-\tau_2\) proves continuity.  Since every representative produces
the same intrinsic increment, uniqueness also proves independence of the
joint lift, the allocation between cross area and central correction, and
the self-area assigned to \(h\).

\subsection{Passage to the physical jet and the uniform-expansion interface}
\label{subsec:physical-jet}

For the scalar backward flow, the shifted four-jet vector field is
\[
 \mathsf H_i(r,j,k,\ell)
 =
 \begin{pmatrix}
 -H_i(r)\\
 -H_i'(r)j\\
 -H_i''(r)j^2-H_i'(r)k\\
 -H_i'''(r)j^3-3H_i''(r)jk-H_i'(r)\ell
 \end{pmatrix}.
\]
The cutoff makes the polynomial jet coordinates bounded.  Spatial
difference quotients and the closedness of differentiation identify the
components of \(P^\tau\) with the first variations of
\[
 \phi,\qquad \partial_y\phi,\qquad
 \partial_{yy}\phi,\qquad\partial_{yyy}\phi.
\]
Consequently Proposition~\ref{prop:intrinsic-flow-variation} supplies the
representation-free derivative of the full reset-flow jet in the
levelwise variation norm.

\medskip
For the first flow component \(U^\tau\), the backward sign convention and
the bracket normalization used above give
\begin{equation}
 \begin{split}
 dU_t^\tau
 ={}&-DH_i(\phi_t)U_t^\tau\,d\mathbf Z_t^{x,i}
     -H_i(\phi_t)\,d\mathbf Z_t^{h,i}\\
    &+\sum_{i<j}[H_i,H_j](\phi_t)\,da_t^{ij},
 \qquad U_v^\tau=0,
 \end{split}                                                       \label{eq:physical-first-variation}
\end{equation}
where
\[
 [H_i,H_j]=DH_jH_i-DH_iH_j.
\]
Thus the central perturbation has the plus sign in
\eqref{eq:physical-first-variation}.  The notation
\(d\mathbf Z^h\) records the cross lift required for a general
finite-\(p\)-variation direction.

\subsubsection{Interface with the uniform expansion}
The coefficient analysis in Section~\ref{sec:generator-expansion} invokes
a separate uniform nonlinear flow-jet proposition.  For radial
realizations with a common bounded representative budget, that proposition
has the estimate
\begin{equation}
 \sup_{I\in\Pi(\mathbf x)}
 \sup_{\|\tau\|_{\mathbb T_{\mathbf x}^p}\le1}
 \left\|
   \mathfrak J_I(\varepsilon;\tau)
   -\mathfrak J_I(0)-\varepsilon P_I^\tau
 \right\|_{\mathfrak V_3^\Delta(I)}
 \le C\varepsilon^2.                                              \label{eq:required-uniform-four-jet}
\end{equation}
Here \(\mathfrak J_I\) contains the flow and the three spatial jets on the
reset interval \(I\).  Proposition~\ref{prop:intrinsic-flow-variation}
identifies the derivative \(P_I^\tau\) and proves its representative
independence.  The bound \eqref{eq:required-uniform-four-jet} is a separate
estimate: it uses the cutoff augmented RDE, second-order fixed-driver
parameter estimates, and uniform control of the representative drivers.
The weighted expansions of \(F\), \(F_y\), and \(F_z\) use both results.
\subsection{Uniform nonlinear expansion on bounded tangent sets}
\label{subsec:uniform-full-jet}

The intrinsic equation identifies the only possible first variation.  We
now prove the uniform nonlinear estimate needed to pass that variation to
the transformed generator.  The argument is carried out for a fixed joint
driver.  Uniformity follows from the quantitative bound on the chosen
representatives, not from continuity of a lift selection.

For \(r\ge1\), write
\[
 \mathfrak J_I^{[r]}(\varepsilon;\mathbf Z,a)
 =
 \bigl(
 \phi^\varepsilon-\operatorname{Id},
 \partial_y\phi^\varepsilon-1,
 \partial_y^2\phi^\varepsilon,\ldots,
 \partial_y^r\phi^\varepsilon
 \bigr)
\]
and set
\begin{equation}
 \|G\|_{\mathfrak V_r^\Delta(I)}
 =
 \sum_{\ell=0}^r\sup_{y\in\mathbb R}
 \left\{
  \|G_\ell(\cdot,y)\|_{\infty;I}
  +\|G_\ell(\cdot,y)\|_{p\text{-var};I}
 \right\}.                                                       \label{eq:full-jet-levelwise-norm}
\end{equation}
For a three-component coefficient tuple \(R=(R_1,R_2,R_3)\), set
\begin{equation}
 \|R\|_{\mathfrak C_3(I)}
 =
 \sum_{j=1}^3\sup_{y\in\mathbb R}
 \left\{
  \|R_j(\cdot,y)\|_{\infty;I}
  +\|R_j(\cdot,y)\|_{p\text{-var};I}
 \right\}.                                                       \label{eq:rational-coefficient-norm}
\end{equation}
The same expression applied to
\((\phi-\operatorname{Id},\partial_y\phi-1,\ldots)\) is denoted by
\(\|\mathfrak J_I^{[r]}\|_{\mathfrak V_r^{\mathrm{aff}}(I)}\).

Let \(\rho_p\) denote the homogeneous \(p\)-variation size of a
weakly geometric rough path.  A representative of
\(\tau=(h,\kappa)\in\mathbb T_{\mathbf x}^p\) consists of
\[
 \mathbf Z\in
 WG\Omega_p(\mathbb R_x^d\oplus\mathbb R_h^d)
 \quad\text{and}\quad
 a\in\mathcal V_0^q(\mathfrak{so}(d))
\]
such that
\[
 \pi_x\mathbf Z=\mathbf x,\qquad
 (\pi_h\mathbf Z)^1=h,\qquad
 \kappa=C_{\mathbf Z}+a,
\]
where
\[
 C_{\mathbf Z}^{j,i}
 =\mathbf Z^{x_j,h_i}+\mathbf Z^{h_j,x_i}.
\]
Put \(\mathbb h_{\mathbf Z}=\mathbf Z^{h,h}\).  The corresponding radial
rough path is
\begin{equation}
 \begin{split}
  (\mathbf x^{\varepsilon;\mathbf Z,a})^1
    &=x+\varepsilon h,\\
  (\mathbf x^{\varepsilon;\mathbf Z,a})^2
    &=\mathbb x+\varepsilon\kappa
      +\varepsilon^2\mathbb h_{\mathbf Z}.
 \end{split}                                                       \label{eq:full-jet-radial-path}
\end{equation}

\begin{proposition}[Uniform full-tangent flow-jet expansion]
\label{prop:uniform-full-jet-expansion}
Assume \(H_i\in C_b^9(\mathbb R)\).  Fix
\(M_{\mathbf x},L_{\rm br}<\infty\).  There are constants
\[
 \delta>0,\qquad \varepsilon_0>0,\qquad C<\infty,\qquad N<\infty
\]
with the following property.  For every
\(\mathbf x\in WG\Omega_p(\mathbb R^d)\) satisfying
\(\rho_p(\mathbf x)\le M_{\mathbf x}\), there is a partition
\(\Pi(\mathbf x)\) with at most \(N\) intervals.  Its endpoints depend only
on the base control.  Let
\[
 \|\tau\|_{\mathbb T_{\mathbf x}^p}\le1
\]
and choose any representative satisfying
\begin{equation}
 \rho_p(\mathbf Z)+\|a\|_{q\text{-var}}\le L_{\rm br}.             \label{eq:full-jet-representative-budget}
\end{equation}
Then, for \(|\varepsilon|\le\varepsilon_0\), every interval
\(I\in\Pi(\mathbf x)\) is a valid reset interval for
\(\mathbf x^{\varepsilon;\mathbf Z,a}\), and
\begin{equation}
 \frac12\le
 \partial_y\phi_t^{\varepsilon,I}(y)
 \le\frac32
 \quad\text{for all }(t,y)\in I\times\mathbb R.                   \label{eq:full-jet-J-lower}
\end{equation}

There is a representative first variation
\(P_I^{\mathbf Z,a,[5]}\) such that
\begin{align}
 &\sup_{I\in\Pi(\mathbf x)}
 \left\|
 \mathfrak J_I^{[5]}(\varepsilon;\mathbf Z,a)
 -\mathfrak J_I^{[5]}(0)
 -\varepsilon P_I^{\mathbf Z,a,[5]}
 \right\|_{\mathfrak V_5^\Delta(I)}
 \le C\varepsilon^2,                                              \label{eq:full-five-jet-remainder}\\
 &\sup_{I\in\Pi(\mathbf x)}
 \left\{
 \|\mathfrak J_I^{[5]}(\varepsilon;\mathbf Z,a)
   \|_{\mathfrak V_5^{\mathrm{aff}}(I)}
 +\|P_I^{\mathbf Z,a,[5]}\|_{\mathfrak V_5^\Delta(I)}
 \right\}
 \le C.                                                          \label{eq:full-five-jet-bound}
\end{align}
Moreover,
\begin{equation}
 \sup_{I\in\Pi(\mathbf x)}
 \|P_I^{\mathbf Z,a,[5]}\|_{\mathfrak V_5^\Delta(I)}
 \le C\|\tau\|_{\mathbb T_{\mathbf x}^p}.                         \label{eq:full-five-jet-linear-bound}
\end{equation}
The projection of \(P_I^{\mathbf Z,a,[5]}\) onto spatial orders
\(0,\ldots,3\) is the intrinsic variation \(P_I^\tau\) from
Proposition~\ref{prop:intrinsic-flow-variation}.  Consequently,
\begin{equation}
 \boxed{
 \sup_{I\in\Pi(\mathbf x)}
 \left\|
 \mathfrak J_I^{[3]}(\varepsilon;\mathbf Z,a)
 -\mathfrak J_I^{[3]}(0)-\varepsilon P_I^\tau
 \right\|_{\mathfrak V_3^\Delta(I)}
 \le C\varepsilon^2
 }                                                               \label{eq:intrinsic-uniform-four-jet}
\end{equation}
and
\begin{equation}
 \sup_{I\in\Pi(\mathbf x)}
 \|P_I^\tau\|_{\mathfrak V_3^\Delta(I)}
 \le C\|\tau\|_{\mathbb T_{\mathbf x}^p}.                         \label{eq:intrinsic-jet-linear-bound}
\end{equation}
In particular, the fourth and fifth spatial derivatives obey
\begin{equation}
 \sup_{\substack{I\in\Pi(\mathbf x)\\
                  |\varepsilon|\le\varepsilon_0}}
 \sum_{\ell=4}^5
 \sup_y\left\{
 \|\partial_y^\ell\phi^{\varepsilon,I}(\cdot,y)\|_{\infty;I}
 +\|\partial_y^\ell\phi^{\varepsilon,I}(\cdot,y)\|_{p\text{-var};I}
 \right\}
 \le C.                                                          \label{eq:fourth-fifth-flow-bound}
\end{equation}

Let
\(\psi_t^{\varepsilon,I}=(\phi_t^{\varepsilon,I})^{-1}\), and define
\[
 \mathfrak I_I^{[3]}(\varepsilon)
 =
 \bigl(
 \psi^\varepsilon-\operatorname{Id},
 \partial_y\psi^\varepsilon-1,
 \partial_y^2\psi^\varepsilon,
 \partial_y^3\psi^\varepsilon
 \bigr).
\]
There is a continuous linear inverse-flow variation
\(\dot{\mathfrak I}_I^\tau\) for which
\begin{equation}
 \sup_{I\in\Pi(\mathbf x)}
 \left\|
 \mathfrak I_I^{[3]}(\varepsilon)
 -\mathfrak I_I^{[3]}(0)
 -\varepsilon\dot{\mathfrak I}_I^\tau
 \right\|_{\mathfrak V_3^\Delta(I)}
 \le C\varepsilon^2.                                              \label{eq:inverse-three-jet-remainder}
\end{equation}
It satisfies
\begin{equation}
 \sup_{I\in\Pi(\mathbf x)}
 \|\dot{\mathfrak I}_I^\tau\|_{\mathfrak V_3^\Delta(I)}
 \le C\|\tau\|_{\mathbb T_{\mathbf x}^p}.                         \label{eq:inverse-jet-linear-bound}
\end{equation}
Its first component is
\begin{equation}
 \dot\psi^\tau
 =-\left(\frac{U^\tau}{J}\right)\circ\psi,                         \label{eq:inverse-first-variation}
\end{equation}
and its remaining components are the first three spatial derivatives of
\eqref{eq:inverse-first-variation}.

Finally, put
\[
 \mathcal Q(\phi,J,K,L)
 =
 \left(
 J^{-1},\frac KJ,\frac LJ-\frac{K^2}{J^2}
 \right).
\]
The same constants give
\begin{equation}
 \sup_{I\in\Pi(\mathbf x)}
 \left\|
 \mathcal Q^\varepsilon_I-\mathcal Q^0_I
 -\varepsilon\dot{\mathcal Q}_I^\tau
 \right\|_{\mathfrak C_3(I)}
 \le C\varepsilon^2,                                              \label{eq:rational-jet-remainder}
\end{equation}
where
\begin{equation}
 \dot{\mathcal Q}^\tau
 =
 \left(
 -\frac{\dot J}{J^2},\
 \frac{\dot K}{J}-\frac{K\dot J}{J^2},\
 \frac{\dot L}{J}-\frac{L\dot J}{J^2}
 -\frac{2K\dot K}{J^2}
 +\frac{2K^2\dot J}{J^3}
 \right).                                                        \label{eq:rational-jet-derivative}
\end{equation}
All constants depend only on
\[
 p,d,T,M_{\mathbf x},L_{\rm br},
 \max_i\|H_i\|_{C_b^9},
\]
the fixed tensor norms, and the fixed cutoff profile.  They do not depend
on \(\tau\), its selected representative within
\eqref{eq:full-jet-representative-budget}, the sign of \(\varepsilon\), or
the concentration profile of the directional variation.
\end{proposition}

\subsubsection{The joined driver and the cutoff equation}
\label{subsubsec:joined-cutoff-driver}

Fix one tangent and one representative satisfying
\eqref{eq:full-jet-representative-budget}.  Regard \(a\) as a path with
values in
\[
 \mathbb R^{d(d-1)/2}.
\]
We use the convention
\[
 a_{s,t}
 =\sum_{i<j}a_{s,t}^{ij}
 (e_i\otimes e_j-e_j\otimes e_i).
\]
The Young joint lift of \(\mathbf Z\) and \(a\) is well defined because
\[
 \frac1p+\frac1q=\frac3p>1,
 \qquad
 \frac2q=\frac4p>1.
\]
Denote this fixed weakly geometric rough path by
\(\mathbf W=\mathbf W^{\mathbf Z,a}\).  Young--Loeve estimates for the
two cross blocks and the self-lift of \(a\) give
\begin{equation}
 \rho_p(\mathbf W)
 \le C_{p,d}\bigl(
  \rho_p(\mathbf Z)+\|a\|_{q\text{-var}}
  +\rho_p(\mathbf Z)\|a\|_{q\text{-var}}
  +\|a\|_{q\text{-var}}^2
 \bigr)
 \le M_W,                                                        \label{eq:joined-driver-budget}
\end{equation}
where \(M_W\) depends only on \(p,d,L_{\rm br}\).

We use an order-five auxiliary jet.  The extra orders belong to the proof,
not to the conclusion: carrying the fifth coordinate closes the spatial
difference-quotient identification of the fourth derivative, and projection
then gives the physical four-jet remainder required in
\eqref{eq:intrinsic-uniform-four-jet}.  No fifth-order physical response is
asserted.  For a scalar vector field \(F\), let \(\mathcal P_5F\) be its
fifth prolongation:
\[
 \bigl(\mathcal P_5F(z)\bigr)_\ell
 =
 \left.
 \partial_y^\ell F(g(y))
 \right|_{\partial_y^kg=z_k,\ 0\le k\le\ell},
 \qquad 0\le\ell\le5.
\]
Each component is a Fa\`a di Bruno polynomial.  It depends on
\(z_1,\ldots,z_\ell\).  Work on
\[
 \mathcal E_{\mathrm{jet}}^{[5]}
 =C_b(\mathbb R;\mathbb R^6)
\]
after shifting
\[
 (z_0,z_1,\ldots,z_5)
 \longmapsto
 (z_0-\operatorname{Id},z_1-1,z_2,\ldots,z_5).
\]
Choose a smooth cutoff that is one when
\(\max_{1\le\ell\le5}|z_\ell-\mathbf1_{\{\ell=1\}}|\le1\) and zero
when this maximum is at least two.  Apply it to components
\(1,\ldots,5\) of \(\mathcal P_5(-H_i)\), leaving the bounded zeroth
component unchanged.  The resulting Nemytskii fields are denoted by
\(A_i\).

The order-five prolongation uses derivatives of \(H_i\) through order five.
Four Fr\'echet derivatives with respect to the jet coordinates use
derivatives through order nine.  Hence
\begin{equation}
 A_i\in
 C_b^4\bigl(
 \mathcal E_{\mathrm{jet}}^{[5]};
 \mathcal E_{\mathrm{jet}}^{[5]}
 \bigr).                                                         \label{eq:rough-cutoff-regularity}
\end{equation}
With the convention
\[
 [A_i,A_j]=DA_jA_i-DA_iA_j,
\]
put \(C_{ij}=[A_i,A_j]\).  Then
\begin{equation}
 C_{ij}\in
 C_b^3\bigl(
 \mathcal E_{\mathrm{jet}}^{[5]};
 \mathcal E_{\mathrm{jet}}^{[5]}
 \bigr).                                                         \label{eq:bracket-cutoff-regularity}
\end{equation}
On the region where the cutoff is one, \(C_{ij}\) is the fifth
prolongation of
\[
 [H_i,H_j]=DH_jH_i-DH_iH_j.
\]

Index the coordinates of \(\mathbf W\) by \(x_i,h_i,a_{ij}\), and define
the affine field family
\begin{equation}
 \mathcal A^\varepsilon_{x_i}=A_i,\qquad
 \mathcal A^\varepsilon_{h_i}=\varepsilon A_i,\qquad
 \mathcal A^\varepsilon_{a_{ij}}=\varepsilon C_{ij}.              \label{eq:affine-field-family}
\end{equation}
After reversing time, solve
\begin{equation}
 d\Theta_t^\varepsilon
 =\mathcal A^\varepsilon_\lambda
   (\Theta_t^\varepsilon)\,d\mathbf W_t^\lambda,
 \qquad
 \Theta_v^\varepsilon=0.                                         \label{eq:fixed-joined-rde}
\end{equation}
The rough integral in \eqref{eq:fixed-joined-rde} is one integral against
the fixed driver \(\mathbf W\).  In mixed notation it is
\begin{equation}
 \begin{split}
 d\Theta_t^\varepsilon
 ={}&A_i(\Theta_t^\varepsilon)\,d\mathbf Z_t^{x,i}
 +\varepsilon A_i(\Theta_t^\varepsilon)\,d\mathbf Z_t^{h,i}\\
 &+\varepsilon\sum_{i<j}
 C_{ij}(\Theta_t^\varepsilon)\,da_t^{ij}.
 \end{split}                                                       \label{eq:fixed-joined-mixed-rde}
\end{equation}
The first two terms retain both cross blocks and the \(h\)-self area of
\(\mathbf Z\).  The last term is a Young integral with all mixed
interactions supplied by \(\mathbf W\).

The central-translation identity can be checked at the Davie level.  The
push-forward of \(\mathbf Z\) under
\((u,v)\mapsto u+\varepsilon v\) has second level
\[
 \mathbb x+\varepsilon C_{\mathbf Z}
 +\varepsilon^2\mathbb h_{\mathbf Z}.
\]
The additional tensor \(\varepsilon a\) contributes
\[
 Q_{j,i}a^{j,i}
 =\sum_{i<j}[A_i,A_j]a^{ij}
\]
under the displayed normalization.  The two equations have the same
Davie increment and a \(3/p\)-order defect.  Uniqueness therefore
identifies \eqref{eq:fixed-joined-rde} with the cutoff RDE driven by
\eqref{eq:full-jet-radial-path}.  There is no factor \(1/2\).

\subsubsection{A uniform fixed-driver linear estimate}
\label{subsubsec:fixed-driver-linear}

For the fixed reference \(\mathbf W\), let
\(\mathcal D_{\mathbf W}^p(I;F)\) be the controlled-path space with norm
\[
 \|R\|_{\mathcal D_{\mathbf W}^p}
 =
 \|R\|_\infty+\|R'\|_\infty
 +\|R'\|_{p\text{-var}}
 +\|R^\sharp\|_{q\text{-var}}.
\]
Consider the terminal linear equation
\begin{equation}
 dR_t
 =M_\lambda(t)R_t\,d\mathbf W_t^\lambda
 +G_\lambda(t)\,d\mathbf W_t^\lambda,
 \qquad R_v=r_v.                                                  \label{eq:joined-linear-rde}
\end{equation}
Assume that \(M\) and \(G\) have controlled norms bounded by \(K\).  There
is a constant
\[
 C_{\rm lin}=C_{\rm lin}(p,M_W,K)
\]
such that
\begin{equation}
 \|R\|_{\mathcal D_{\mathbf W}^p(I)}
 \le C_{\rm lin}\left(
 |r_v|+\|G\|_{\mathcal D_{\mathbf W}^p(I)}
 \right).                                                        \label{eq:joined-linear-estimate}
\end{equation}

We record the dependence in \eqref{eq:joined-linear-estimate}.  Let
\(\omega_{\mathbf W}\) be a control for the two levels of \(\mathbf W\).
Controlled sewing \cite{Gubinelli2004} applied to
\eqref{eq:joined-linear-rde} gives, whenever
\(\omega_{\mathbf W}(s,t)\le\eta\),
\begin{align}
 \|R\|_{\mathcal D_{\mathbf W}^p;[s,t]}
 \le{}&
 c_0\left(
 |R_t|+\|G\|_{\mathcal D_{\mathbf W}^p;[s,t]}
 \right) \notag\\
 &+c_1\left(
 \eta^{1/p}+\eta^{2/p}+\eta^{3/p-1}
 \right)
 \|R\|_{\mathcal D_{\mathbf W}^p;[s,t]}.                          \label{eq:local-linear-absorption}
\end{align}
All three exponents are positive because \(2<p<3\).  Choose
\(\eta=\eta(p,K)\) so that the second-line coefficient is at most
\(1/2\).  The greedy partition associated with
\(\omega_{\mathbf W}\) has at most
\begin{equation}
 1+\frac{\omega_{\mathbf W}(u,v)}{\eta}
 \le1+\frac{M_W^p}{\eta}                                         \label{eq:joined-greedy-count}
\end{equation}
intervals.  Absorption on each interval followed by backward iteration
proves \eqref{eq:joined-linear-estimate}.  The auxiliary endpoints used in
this estimate may depend on \(\mathbf W\), but
\eqref{eq:joined-driver-budget} makes their number uniform.  They do not
alter the reset endpoints in \(\Pi(\mathbf x)\).

The same calculation, with globally bounded fields and terminal value
zero, gives
\begin{equation}
 \sup_{\substack{|\varepsilon|\le1\\
                  I\in\Pi(\mathbf x)}}
 \|\Theta^\varepsilon\|_{\mathcal D_{\mathbf W}^p(I;
 \mathcal E_{\mathrm{jet}}^{[5]})}
 \le C_\Theta.                                                    \label{eq:joined-state-bound}
\end{equation}
Indeed, on each interval in the greedy partition, the controlled integral
of \(\mathcal A^\varepsilon(\Theta^\varepsilon)\) is bounded by the
global \(C_b^3\) field norm plus the absorbable multiple in
\eqref{eq:local-linear-absorption}.  Iteration uses only
\eqref{eq:joined-greedy-count}.  This is the bounded-set,
fixed-reference regularity regime of the It\^o--Lyons map
\cite{Bailleul2015}.  The driver is held fixed throughout the calculation.

\subsubsection{Difference, first variation, and quadratic remainder}
\label{subsubsec:joined-quadratic-remainder}

Write
\[
 \mathcal A^\varepsilon
 =\mathcal A^0+\varepsilon\dot{\mathcal A},
\]
where
\[
 \dot{\mathcal A}_{x_i}=0,\qquad
 \dot{\mathcal A}_{h_i}=A_i,\qquad
 \dot{\mathcal A}_{a_{ij}}=C_{ij}.
\]
Let
\[
 \Delta^\varepsilon=\Theta^\varepsilon-\Theta^0.
\]
The exact difference equation is
\begin{equation}
 d\Delta^\varepsilon
 =\overline{D\mathcal A^0}_\lambda
   \Delta^\varepsilon\,d\mathbf W^\lambda
 +\varepsilon
   \dot{\mathcal A}_\lambda(\Theta^\varepsilon)
   \,d\mathbf W^\lambda,
 \qquad \Delta_v^\varepsilon=0,                                  \label{eq:joined-exact-difference}
\end{equation}
where
\[
 \overline{D\mathcal A^0}_\lambda
 =
 \int_0^1D\mathcal A^0_\lambda
 \bigl(\Theta^0+r\Delta^\varepsilon\bigr)\,dr.
\]
Equations \eqref{eq:joined-state-bound} and
\eqref{eq:rough-cutoff-regularity}--\eqref{eq:bracket-cutoff-regularity}
give a common controlled norm for the coefficients and the source in
\eqref{eq:joined-exact-difference}.  Hence
\begin{equation}
 \|\Delta^\varepsilon\|_{\mathcal D_{\mathbf W}^p(I)}
 \le C|\varepsilon|.                                              \label{eq:joined-lipschitz-parameter}
\end{equation}

Define \(P^{\mathbf Z,a,[5]}\) by
\begin{equation}
 dP_t
 =D\mathcal A^0_\lambda(\Theta_t^0)P_t\,d\mathbf W_t^\lambda
 +\dot{\mathcal A}_\lambda(\Theta_t^0)\,d\mathbf W_t^\lambda,
 \qquad P_v=0.                                                    \label{eq:joined-first-variation}
\end{equation}
The linear estimate gives
\begin{equation}
 \|P^{\mathbf Z,a,[5]}\|_{\mathcal D_{\mathbf W}^p(I)}
 \le C.                                                          \label{eq:joined-first-variation-bound}
\end{equation}
Set
\[
 R^\varepsilon
 =\Delta^\varepsilon-\varepsilon P^{\mathbf Z,a,[5]}.
\]
Subtracting \(\varepsilon\) times
\eqref{eq:joined-first-variation} from
\eqref{eq:joined-exact-difference} yields
\begin{equation}
 \begin{split}
 dR^\varepsilon
 ={}&D\mathcal A^0_\lambda(\Theta^0)R^\varepsilon
       \,d\mathbf W^\lambda\\
 &+Q_{\mathcal A^0,\lambda}
   (\Theta^0,\Delta^\varepsilon)\,d\mathbf W^\lambda\\
 &+\varepsilon\left[
   \dot{\mathcal A}_\lambda(\Theta^\varepsilon)
   -\dot{\mathcal A}_\lambda(\Theta^0)
   \right]d\mathbf W^\lambda,
 \qquad R_v^\varepsilon=0,
 \end{split}                                                       \label{eq:joined-exact-remainder}
\end{equation}
where
\[
 Q_{\mathcal A^0}(\theta,\delta)
 =\mathcal A^0(\theta+\delta)
  -\mathcal A^0(\theta)-D\mathcal A^0(\theta)\delta.
\]

We use the following fixed-driver composition estimate.  Its explicit
form is needed because the norm includes both the Gubinelli derivative and
the controlled remainder.

\begin{lemma}[Quadratic controlled composition]
\label{lem:joined-controlled-composition}
Let \(2<p<3\), \(q=p/2\), and let \(\mathbf W\) range over a family with
\(\rho_p(\mathbf W)\le M_W\).  Let \(E\) and \(\bar E\) be Banach spaces,
and suppose that
\[
 X,\Delta\in\mathcal D_{\mathbf W}^p(I;E),
 \qquad
 \|X\|_{\mathcal D_{\mathbf W}^p}
 +\|\Delta\|_{\mathcal D_{\mathbf W}^p}\le K.
\]
Assume that every segment \(X_t+r\Delta_t\), \(0\le r\le1\), lies in a
common convex strip on which the first four Fr\'echet derivatives of
\(F:E\to\bar E\) are bounded.  Then
\[
 Q_F(X,\Delta)=F(X+\Delta)-F(X)-DF(X)\Delta
 \in\mathcal D_{\mathbf W}^p(I;\bar E)
\]
and
\begin{equation}
 \|Q_F(X,\Delta)\|_{\mathcal D_{\mathbf W}^p}
 \le C\|\Delta\|_{\mathcal D_{\mathbf W}^p}^2.                    \label{eq:abstract-controlled-quadratic}
\end{equation}
If the first three derivatives of \(G:E\to\bar E\) are bounded on the
same strip, then
\begin{equation}
 \|G(X+\Delta)-G(X)\|_{\mathcal D_{\mathbf W}^p}
 \le C\|\Delta\|_{\mathcal D_{\mathbf W}^p}.                      \label{eq:abstract-controlled-linear}
\end{equation}
The constants in these estimates depend only on \(p,M_W,K\), the stated
derivative bounds, and the fixed tensor norms.  In particular, they are
uniform over intervals, drivers, and states that share these budgets.
\end{lemma}

\begin{proof}
Write
\[
 X_{s,t}=X'_sW_{s,t}+X^\sharp_{s,t},
 \qquad
 \Delta_{s,t}=\Delta'_sW_{s,t}+\Delta^\sharp_{s,t},
\]
where \(W\) denotes the first level of \(\mathbf W\), and set
\(d=\|\Delta\|_{\mathcal D_{\mathbf W}^p}\).  The assertion is immediate
when \(d=0\).  Otherwise apply the variation controls of \(X\) and
\(d^{-1}\Delta\), together with a control for \(\mathbf W\).  After
rescaling their sum, there is a control \(\omega\) with
\(\omega(u,v)\le1\), where the rescaling factor is bounded only in terms of
\(p,M_W,K\), such that
\begin{align}
 |X_{s,t}|+|X'_{s,t}|
 &\le C\omega(s,t)^{1/p},
 &|X^\sharp_{s,t}|
 &\le C\omega(s,t)^{2/p},                                      \label{eq:composition-base-control}\\
 |\Delta_{s,t}|+|\Delta'_{s,t}|
 &\le Cd\,\omega(s,t)^{1/p},
 &|\Delta^\sharp_{s,t}|
 &\le Cd\,\omega(s,t)^{2/p}.                                  \label{eq:composition-direction-control}
\end{align}
The constants also absorb the suprema of the values and controlled
derivatives.  This normalization is the point at which the common
joined-driver budget enters the estimate.

Define
\[
 \Gamma_F(x,z)=F(x+z)-F(x)-DF(x)z.
\]
Taylor's formula and the \(C_b^4\) bound give, uniformly on the relevant
segments,
\begin{equation}
 \begin{gathered}
  |\Gamma_F(x,z)|\le C|z|^2,
  \qquad
  \|D_x\Gamma_F(x,z)\|\le C|z|^2,
  \qquad
  \|D_z\Gamma_F(x,z)\|\le C|z|,\\
  \|D_{xx}^2\Gamma_F(x,z)\|\le C|z|^2,
  \qquad
  \|D_{xz}^2\Gamma_F(x,z)\|\le C|z|,
  \qquad
  \|D_{zz}^2\Gamma_F(x,z)\|\le C.
 \end{gathered}                                                   \label{eq:composition-gamma-derivatives}
\end{equation}
For example,
\begin{align*}
 D_x\Gamma_F(x,z)
 &=\int_0^1(1-r)D^3F(x+rz)[z,z],dr,\\
 D_z\Gamma_F(x,z)
 &=\int_0^1D^2F(x+rz)[z],dr.
\end{align*}
The first estimate in the second line of
\eqref{eq:composition-gamma-derivatives} is the second-order Taylor
remainder for \(D^2F(x+z)\); this is precisely where the fourth derivative
of \(F\) is used.

Put \(Z=(X,\Delta)\) and \(Z'=(X',\Delta')\).  The canonical Gubinelli
derivative of \(Q_F(X,\Delta)=\Gamma_F(Z)\) is
\begin{equation}
 \begin{split}
  Q'_t
  ={}&A_tX'_t+B_t\Delta'_t,\\
  A_t
  ={}&DF(X_t+\Delta_t)-DF(X_t)-D^2F(X_t)[\Delta_t],\\
  B_t
  ={}&DF(X_t+\Delta_t)-DF(X_t).
 \end{split}                                                       \label{eq:composition-gubinelli-derivative}
\end{equation}
The derivative bounds in \eqref{eq:composition-gamma-derivatives}, applied
to increments of \((X,\Delta)\), imply
\begin{equation}
 \|A\|_\infty+\|A\|_{p\text{-var}}\le Cd^2,
 \qquad
 \|B\|_\infty+\|B\|_{p\text{-var}}\le Cd.                       \label{eq:composition-coefficient-variation}
\end{equation}
The product variation inequality, together with
\eqref{eq:composition-base-control}--\eqref{eq:composition-direction-control},
therefore gives
\begin{equation}
 \|Q'\|_\infty+\|Q'\|_{p\text{-var}}\le Cd^2.                    \label{eq:composition-derivative-bound}
\end{equation}

It remains to estimate the controlled remainder.  Taylor's formula on the
product space \(E\times E\) gives the exact identity
\begin{equation}
 \begin{split}
  Q^\sharp_{s,t}
  ={}&D\Gamma_F(Z_s)Z^\sharp_{s,t}\\
  &+\int_0^1(1-r)D^2\Gamma_F(Z_s+rZ_{s,t})
       [Z_{s,t},Z_{s,t}]\,dr,
 \end{split}                                                       \label{eq:composition-controlled-remainder}
\end{equation}
where \(Z^\sharp=(X^\sharp,\Delta^\sharp)\).  Along the segment in
\eqref{eq:composition-controlled-remainder}, its second coordinate is
bounded by \(Cd\).  The six block estimates in
\eqref{eq:composition-gamma-derivatives} and
\eqref{eq:composition-base-control}--\eqref{eq:composition-direction-control}
then yield
\begin{equation}
 |Q^\sharp_{s,t}|
 \le Cd^2\omega(s,t)^{2/p}.                                      \label{eq:composition-remainder-control}
\end{equation}
Since \(q=p/2\), this implies
\(\|Q^\sharp\|_{q\text{-var}}\le Cd^2\).  Together with
\(|Q_t|\le Cd^2\) and \eqref{eq:composition-derivative-bound}, it proves
\eqref{eq:abstract-controlled-quadratic}.

For \(\Lambda_G(x,z)=G(x+z)-G(x)\), the corresponding block bounds are
\[
 \begin{gathered}
  |\Lambda_G|\le C|z|,
  \qquad \|D_x\Lambda_G\|\le C|z|,
  \qquad \|D_z\Lambda_G\|\le C,\\
  \|D_{xx}^2\Lambda_G\|\le C|z|,
  \qquad
  \|D_{xz}^2\Lambda_G\|+\|D_{zz}^2\Lambda_G\|\le C.
 \end{gathered}
\]
The same derivative and remainder calculation gives a bound by
\(C(d+d^2)\).  Because \(d\le K\), this is bounded by \(C_Kd\), which
proves \eqref{eq:abstract-controlled-linear}.  Every control and derivative
constant used above is uniform under the budgets in the statement.
\end{proof}

Apply the lemma with \(X=\Theta^0\) and
\(\Delta=\Delta^\varepsilon\).  For \(F=\mathcal A^0\), the required
four derivatives are supplied by \eqref{eq:rough-cutoff-regularity}.  For
\(G=\dot{\mathcal A}\), three derivatives follow from
\eqref{eq:rough-cutoff-regularity} and
\eqref{eq:bracket-cutoff-regularity}.  The common state bound
\eqref{eq:joined-state-bound} and driver budget
\eqref{eq:joined-driver-budget} therefore give
\begin{align}
 \|Q_{\mathcal A^0}
   (\Theta^0,\Delta^\varepsilon)
 \|_{\mathcal D_{\mathbf W}^p}
 &\le C
 \|\Delta^\varepsilon\|_{\mathcal D_{\mathbf W}^p}^2,             \label{eq:joined-composition-quadratic}\\
 \|\dot{\mathcal A}(\Theta^\varepsilon)
   -\dot{\mathcal A}(\Theta^0)
 \|_{\mathcal D_{\mathbf W}^p}
 &\le C
 \|\Delta^\varepsilon\|_{\mathcal D_{\mathbf W}^p}.               \label{eq:joined-composition-linear}
\end{align}

Combining \eqref{eq:joined-lipschitz-parameter} with the two composition
estimates shows that the source in \eqref{eq:joined-exact-remainder} has
controlled norm at most \(C\varepsilon^2\).  A final application of
\eqref{eq:joined-linear-estimate} gives
\begin{equation}
 \boxed{
 \|R^\varepsilon\|_{\mathcal D_{\mathbf W}^p(I)}
 \le C\varepsilon^2.
 }                                                               \label{eq:joined-controlled-quadratic}
\end{equation}
No reset endpoint or local absorption threshold in this calculation
depends on the direction.

\subsubsection{Common reset endpoints and removal of the cutoff}
\label{subsubsec:full-jet-common-reset}

Let \(\delta_*>0\) be the local rough-control threshold for the uncut
order-five jet equation.  Choose
\(\delta>0\) so that \(C\delta\le\delta_*/2\).  Select a base control
\(\omega_{\mathbf x}\) and a partition
\(\Pi(\mathbf x)\) such that
\begin{equation}
 \omega_{\mathbf x}(u,v)\le\delta
 \quad\text{for every }I=[u,v]\in\Pi(\mathbf x).                  \label{eq:base-reset-threshold}
\end{equation}
The threshold \(\delta\) is selected below from the cutoff radius.  A
greedy construction gives
\[
 \#\Pi(\mathbf x)
 \le1+\frac{\omega_{\mathbf x}(0,T)}{\delta}
 \le N(M_{\mathbf x},\delta).
\]

For the radial rough path
\eqref{eq:full-jet-radial-path}, choose controls for \(h,\kappa\), and
\(\mathbb h_{\mathbf Z}\).  The levelwise rough control satisfies
\begin{equation}
 \omega_{\varepsilon,\tau}(s,t)
 \le C\left\{
 \omega_{\mathbf x}(s,t)
 +|\varepsilon|^p\omega_h(s,t)
 +|\varepsilon|^q\omega_\kappa(s,t)
 +|\varepsilon|^p\omega_{\mathbb h_{\mathbf Z}}(s,t)
 \right\}.                                                       \label{eq:full-radial-control}
\end{equation}
On the tangent unit ball, the first three directional controls have total
mass bounded by a function of \(L_{\rm br}\).  Choose one
\(\varepsilon_0\) so that their contribution in
\eqref{eq:full-radial-control} is at most \(\delta_*/2\).  This choice
works on every base interval, even when the directional variation is
concentrated inside one of them.

The unperturbed prolonged flow obeys
\begin{equation}
 \sup_{t\in I}
 \|\Theta_t^0\|_{\mathcal E_{\mathrm{jet}}^{[5]}}
 \le C_0\omega_{\mathbf x}(u,v)^{1/p}.                            \label{eq:base-jet-smallness}
\end{equation}
Reduce \(\delta\) until the right-hand side is at most \(1/4\).  By
\eqref{eq:joined-lipschitz-parameter}, reduce
\(\varepsilon_0\) further so that
\[
 \sup_{t\in I}
 \|\Theta_t^\varepsilon-\Theta_t^0
 \|_{\mathcal E_{\mathrm{jet}}^{[5]}}
 \le\frac14.
\]
The cutoff solution then remains in the region where the cutoff equals
one.  First-exit uniqueness identifies it with the uncut prolonged flow
on all of \(I\).  Its first spatial derivative satisfies
\eqref{eq:full-jet-J-lower}.

It remains to identify the prolonged coordinates.  Keep
\(\mathbf W\) and \(\varepsilon\) fixed.  For a spatial increment
\(\eta\ne0\), set
\[
 \Delta_\eta g(\cdot,y)
 =\frac{g(\cdot,y+\eta)-g(\cdot,y)}{\eta}.
\]
Subtracting the equations at \(y+\eta\) and \(y\), and using the
mean-value formula, gives a linear equation for
\(\Delta_\eta\phi^\varepsilon\).  Its coefficient is the mean of
\(DH_i\) and its source difference from the candidate \(J^\varepsilon\)
is bounded by \(C|\eta|\) in the fixed-driver controlled norm.  The
estimate \eqref{eq:joined-linear-estimate} therefore gives
\[
 \sup_y
 \|\Delta_\eta\phi^\varepsilon-J^\varepsilon
 \|_{\mathcal D_{\mathbf W}^p(I)}
 \le C|\eta|.
\]
Repeating the same subtraction in the equations for successive
prolongation coordinates gives
\begin{equation}
 \sup_y
 \|\Delta_\eta\partial_y^\ell\phi^\varepsilon
       -\partial_y^{\ell+1}\phi^\varepsilon
 \|_{\mathcal D_{\mathbf W}^p(I)}
 \le C|\eta|,
 \qquad 0\le\ell\le4.                                             \label{eq:five-jet-spatial-differences}
\end{equation}
At each step, the new source is a finite sum of a bounded derivative of
\(H\), a previously controlled spatial difference, and one mean-coefficient
difference.  The latter is \(O(|\eta|)\) because the next derivative of
\(H\) is bounded.  The last step uses derivatives through order nine.
Letting \(\eta\to0\) proves that the six coordinates are the physical
spatial derivatives through order five.

The embedding
\begin{equation}
 \|G\|_{\mathfrak V_5^\Delta(I)}
 \le C(1+M_W)
 \|G\|_{\mathcal D_{\mathbf W}^p(I;
 \mathcal E_{\mathrm{jet}}^{[5]})}                               \label{eq:controlled-to-levelwise}
\end{equation}
follows from
\(G_{s,t}=G'_sW_{s,t}+G^\sharp_{s,t}\), the product bound for the
first term, and the \(q\)-variation bound for the remainder.  Equations
\eqref{eq:joined-state-bound},
\eqref{eq:joined-first-variation-bound}, and
\eqref{eq:joined-controlled-quadratic} now give
\eqref{eq:full-five-jet-remainder}--\eqref{eq:full-five-jet-bound}.

\subsubsection{Identification with the intrinsic variation}
\label{subsubsec:full-jet-intrinsic-identification}

At \(\varepsilon=0\), equation
\eqref{eq:joined-first-variation} reads
\begin{equation}
 \begin{split}
 dP_t
 ={}&DA_i(\Theta_t^0)P_t\,d\mathbf Z_t^{x,i}
     +A_i(\Theta_t^0)\,d\mathbf Z_t^{h,i}\\
    &+\sum_{i<j}C_{ij}(\Theta_t^0)\,da_t^{ij},
 \qquad P_v=0.
 \end{split}                                                       \label{eq:full-jet-representative-variation}
\end{equation}
The cutoff is inactive, so the first four coordinates of
\eqref{eq:full-jet-representative-variation} are precisely the
representative equation used in
Proposition~\ref{prop:intrinsic-flow-variation}.  Their Davie increment is
\eqref{eq:intrinsic-xi}.  Intrinsic sewing uniqueness gives
\[
 \operatorname{pr}_{0:3}P_I^{\mathbf Z,a,[5]}=P_I^\tau.
\]
The physical-jet identities in
\eqref{eq:five-jet-spatial-differences} also show that the higher
coordinates are the fourth and fifth spatial derivatives of the same
first component.  Thus the full order-five variation is independent of the
representative.  On the tangent unit ball,
\eqref{eq:joined-first-variation-bound} and
\eqref{eq:controlled-to-levelwise} give a uniform bound.  Apply this bound
to a ray-homogeneous bounded realization of
\(\tau/\|\tau\|_{\mathbb T_{\mathbf x}^p}\) and use linearity of the
intrinsic first component and its spatial derivatives.  This proves
\eqref{eq:full-five-jet-linear-bound}.

The graph defined by the five successive spatial derivative identities is
closed in \(\mathfrak V_5^\Delta(I)\).  Every parameter difference quotient
lies in this graph, and \eqref{eq:joined-controlled-quadratic} gives its
limit.  Hence parameter differentiation commutes with spatial
differentiation through order five.  Projection gives
\eqref{eq:intrinsic-uniform-four-jet}.  Estimate
\eqref{eq:intrinsic-jet-linear-bound} is
\eqref{eq:intrinsic-flow-bound} in the physical jet norm.

\subsubsection{Inverse-flow and rational-jet calculus}
\label{subsubsec:inverse-jet-calculus}

We finish with the deterministic composition step.  The product estimate
\begin{equation}
 \|fg\|_{\mathcal V_p(I)}
 \le
 \|f\|_{\infty;I}\|g\|_{\mathcal V_p(I)}
 +\|g\|_{\infty;I}\|f\|_{\mathcal V_p(I)}                         \label{eq:levelwise-product}
\end{equation}
shows that the levelwise variation spaces are Banach algebras.  If
\(g_t\) and \(u_t\) are time-dependent spatial maps, the identity
\begin{equation}
 \begin{split}
 g_t\circ u_t-g_s\circ u_s
 ={}&(g_t-g_s)\circ u_t\\
 &+\int_0^1
 Dg_s\bigl(u_s+r(u_t-u_s)\bigr)
 (u_t-u_s)\,dr
 \end{split}                                                       \label{eq:composition-increment}
\end{equation}
and \eqref{eq:levelwise-product} control their composition in
\(p\)-variation.  Applying \eqref{eq:composition-increment} to first and
second differences shows that composition has a uniform second-order
Taylor bound whenever the required spatial derivatives are uniformly
bounded.

Let \(\phi=\operatorname{Id}+g\), assume
\(\inf\partial_y\phi\ge1/2\), and let
\(\psi=\phi^{-1}\).  The displacement \(g\) is bounded, so
\(\phi(y)\to\pm\infty\) as \(y\to\pm\infty\).  The derivative lower bound
therefore makes \(\phi\) a global diffeomorphism.  Its first three inverse
derivatives are
\begin{align}
 \partial_y\psi
 &=\frac1{J\circ\psi},                                            \label{eq:inverse-jet-one}\\
 \partial_y^2\psi
 &=-\frac{K\circ\psi}{(J\circ\psi)^3},                            \label{eq:inverse-jet-two}\\
 \partial_y^3\psi
 &=\frac{3(K\circ\psi)^2}{(J\circ\psi)^5}
   -\frac{L\circ\psi}{(J\circ\psi)^4}.                            \label{eq:inverse-jet-three}
\end{align}
Differentiating the identity \(\phi\circ\psi=\operatorname{Id}\) gives,
for a variation \(u\),
\begin{equation}
 D\operatorname{Inv}(g)[u]
 =-\left(\frac uJ\right)\circ\psi.                                \label{eq:first-inverse-differential}
\end{equation}
For two variations \(u,v\), a second differentiation gives
\begin{equation}
 \begin{split}
 D^2\operatorname{Inv}(g)[u,v]
 =-\frac1{J\circ\psi}\bigl\{
 &(u'\circ\psi)D\operatorname{Inv}(g)[v]\\
 &+(v'\circ\psi)D\operatorname{Inv}(g)[u]\\
 &+(K\circ\psi)
 D\operatorname{Inv}(g)[u]
 D\operatorname{Inv}(g)[v]
 \bigr\}.
 \end{split}                                                       \label{eq:second-inverse-differential}
\end{equation}
Differentiate
\eqref{eq:first-inverse-differential}--\eqref{eq:second-inverse-differential}
three times in space.  Every term is a product of a power of
\((J\circ\psi)^{-1}\), a derivative of \(\phi\) of order at most five,
and a derivative of \(u\) or \(v\) of order at most four.  Equations
\eqref{eq:levelwise-product}--\eqref{eq:composition-increment} therefore
give
\begin{equation}
 \|D\operatorname{Inv}(g)[u]\|_{\mathfrak V_3^\Delta}
 \le C\|u\|_{\mathfrak V_5^\Delta},\qquad
 \|D^2\operatorname{Inv}(g)[u,v]\|_{\mathfrak V_3^\Delta}
 \le C\|u\|_{\mathfrak V_5^\Delta}
       \|v\|_{\mathfrak V_5^\Delta}.                              \label{eq:inverse-C2-bound}
\end{equation}
The constant is uniform under
\eqref{eq:full-jet-J-lower} and
\eqref{eq:full-five-jet-bound}.  Taylor's integral formula,
\eqref{eq:full-five-jet-remainder}, and
\eqref{eq:inverse-C2-bound} prove
\eqref{eq:inverse-three-jet-remainder}.  Formula
\eqref{eq:inverse-first-variation} is
\eqref{eq:first-inverse-differential} with \(u=U^\tau\).
The first estimate in \eqref{eq:inverse-C2-bound}, together with
\eqref{eq:full-five-jet-linear-bound}, gives
\eqref{eq:inverse-jet-linear-bound}.

On the set \(J\ge1/2\), scalar inversion and multiplication are twice
continuously differentiable in the levelwise Banach algebra.  Applying
their second-order Taylor formulas to \(J,K,L\) gives
\eqref{eq:rational-jet-remainder}.  Direct differentiation gives
\eqref{eq:rational-jet-derivative}.  This completes the proof of
Proposition~\ref{prop:uniform-full-jet-expansion}.

\begin{remark}[Uniformity and the weakly geometric domain]
\label{rem:full-jet-uniformity}
The proposition is uniform for changing tangents and changing
representatives only when
\eqref{eq:full-jet-representative-budget} holds with one
\(L_{\rm br}\).  For any fixed finite-norm representative, the same
argument applies with a constant depending on its displayed norm.
First-order representative independence is stronger: it follows from
intrinsic sewing uniqueness and does not require a common bound over all
representatives.  A strongly contacting sequence of actual rough-path
secants can be recoded using the bounded representatives supplied by the
quantitative realization lemma.

Every driver used above is weakly geometric.  The proof uses Chen's
relation, weak-geometric symmetry, finite variation, and the fixed-reference
controlled RDE estimates.  It does not use smooth approximation in the
same \(p\)-variation topology.  The fixed-reference regularity result is not
a Fr\'echet differentiability theorem for rough-path space.
\end{remark}
\section{Differentiation of the transformed generator}
\label{sec:generator}
\label{sec:generator-expansion}

The reset-flow derivative must be propagated through a generator whose
growth is quadratic in the martingale variable.  Uniform convergence of the
flow jets without a weight in \(z\) is insufficient for this purpose.  The
required output is an expansion of \(F\) with weight
\(1+\lvert z\rvert^2\), together with first-order parameter bounds and
state-continuity estimates for \(F_y\) and \(F_z\).  These are the estimates
used in the conditional source norm of
Section~\ref{sec:local-transformed-bsde}.

\subsection{Weighted coefficient spaces}
\label{subsec:weighted-coefficient-spaces}

Fix the common reset partition \(\Pi(\mathbf x)\).  For a progressively
measurable coefficient \(G=G(\omega,t,y,z)\), \(M<\infty\), and
\(r\geq0\), define
\begin{equation}
 \|G\|_{\mathcal W_{M,r}}
 :=\operatorname*{ess\,sup}_{\substack{
       \omega,\ I\in\Pi(\mathbf x),\ t\in I,\ |y|\leq M\\
       z\in\mathbb R^m}}
 \frac{|G(\omega,t,y,z)|}{1+|z|^r}.
 \label{eq:weighted-coefficient-norm}
\end{equation}
The Euclidean norm is used when \(G\) is vector valued.  The restriction to
\(|y|\leq M\) is part of the statement.  The local transformed solutions
and their convex interpolants remain in one such strip.

Let \(\tau=(h,\kappa)\in\mathbb T_{\mathbf x}^p\), and choose
\(\mathfrak r=(\mathbf Z,a)\in\operatorname{Rep}_{\mathbf x}(\tau)\).
Throughout this section,
\begin{equation}
 \|\tau\|_{\mathbb T_{\mathbf x}^p}\leq1,
 \qquad
 \mathfrak b(\mathbf Z,a)\leq L_{\mathrm{br}},
 \label{eq:coefficient-realization-budget}
\end{equation}
where \(L_{\mathrm{br}}<\infty\) is fixed.  Let
\(\mathbf x^{\varepsilon;\mathfrak r}\) be the radial realization from
\eqref{eq:radial-realization}.  On \(I=[u,v]\), write
\[
 \phi^{\varepsilon;\mathfrak r},\qquad
 J^{\varepsilon;\mathfrak r}=\partial_y\phi^{\varepsilon;\mathfrak r},
 \qquad
 K^{\varepsilon;\mathfrak r}=\partial_{yy}\phi^{\varepsilon;\mathfrak r},
 \qquad
 L^{\varepsilon;\mathfrak r}=\partial_{yyy}\phi^{\varepsilon;\mathfrak r}
\]
for the reset-flow jets.  Base quantities have no superscript.  Denote their
intrinsic first variation by
\begin{equation}
 (u_I^\tau,j_I^\tau,k_I^\tau,\ell_I^\tau)
 :=P_I^\tau.
 \label{eq:intrinsic-jet-components}
\end{equation}
Propositions~\ref{prop:intrinsic-flow-variation} and
\ref{prop:uniform-full-jet-expansion} imply that this vector is continuous
linear in \(\tau\), is independent of \(\mathfrak r\), and satisfies the
quadratic jet remainder \eqref{eq:intrinsic-uniform-four-jet} and the
linear bound \eqref{eq:intrinsic-jet-linear-bound}.  The same proposition
gives
\[
 \frac12\leq J^{\varepsilon;\mathfrak r}\leq\frac32
\]
through \eqref{eq:full-jet-J-lower}.
The constant is uniform under \eqref{eq:coefficient-realization-budget}.
The same augmented-flow estimate gives a uniform zeroth-order bound for
\(\partial_y^4\phi^{\varepsilon;\mathfrak r}\); see
\eqref{eq:fourth-fifth-flow-bound}.  This bound follows by
solving the triangular fourth spatial variational RDE on the common bounded
driver set.  Its coefficients contain derivatives of \(H_i\) and lower
spatial jets only, and are bounded under Assumption~\ref{ass:coefficients}.

Define
\begin{equation}
 F^{\varepsilon;\mathfrak r}(t,y,z)
 :=\frac{1}{J^{\varepsilon;\mathfrak r}_t(y)}
 f\bigl(t,\phi^{\varepsilon;\mathfrak r}_t(y),
          J^{\varepsilon;\mathfrak r}_t(y)z\bigr)
 +\frac12
  \frac{K^{\varepsilon;\mathfrak r}_t(y)}
       {J^{\varepsilon;\mathfrak r}_t(y)}|z|^2.
 \label{eq:perturbed-transformed-generator}
\end{equation}

\subsection{The full-tangent coefficient expansion}
\label{subsec:full-tangent-coefficient-expansion}

In the next display, \(f,f_y,f_z\) are evaluated at
\((t,\phi_t(y),J_t(y)z)\).  Set
\begin{equation}
 \begin{aligned}
 \dot F_I[\tau](t,y,z)
 :={}&-\frac{j_I^\tau}{J^2}f
 +\frac1J\left(f_yu_I^\tau+f_z\mathbin{\cdot}(j_I^\tau z)\right)\\
 &+\frac12\left(
   \frac{k_I^\tau}{J}-\frac{Kj_I^\tau}{J^2}
  \right)|z|^2.
 \end{aligned}
 \label{eq:full-tangent-generator-derivative}
\end{equation}

\begin{proposition}[Weighted full-tangent generator expansion]
\label{prop:weighted-full-tangent-generator}
For every \(M<\infty\) and \(L_{\mathrm{br}}<\infty\), there are
\(\varepsilon_0>0\) and \(C_M<\infty\) such that, for every
\(I\in\Pi(\mathbf x)\), every \(\tau\) and representative satisfying
\eqref{eq:coefficient-realization-budget}, and every
\(0<|\varepsilon|\leq\varepsilon_0\),
\begin{align}
 \|F^{\varepsilon;\mathfrak r}-F^0
       -\varepsilon\dot F_I[\tau]\|_{\mathcal W_{M,2}}
 &\leq C_M\varepsilon^2,
 \label{eq:full-F-quadratic-remainder}\\
 \|F_y^{\varepsilon;\mathfrak r}-F_y^0\|_{\mathcal W_{M,2}}
 &\leq C_M|\varepsilon|,
 \label{eq:full-Fy-parameter-bound}\\
 \|F_z^{\varepsilon;\mathfrak r}-F_z^0\|_{\mathcal W_{M,1}}
 &\leq C_M|\varepsilon|.
 \label{eq:full-Fz-parameter-bound}
\end{align}
Moreover,
\begin{equation}
 \|\dot F_I[\tau]\|_{\mathcal W_{M,2}}
 \leq C_M\|\tau\|_{\mathbb T_{\mathbf x}^p},
 \label{eq:dot-F-continuous-bound}
\end{equation}
and \(\tau\mapsto\dot F_I[\tau]\) is continuous linear.  The family also
satisfies, uniformly for \(|y|,|\bar y|\leq M\),
\begin{align}
 &|F_y^{\varepsilon;\mathfrak r}(t,y,z)
      -F_y^{\varepsilon;\mathfrak r}(t,\bar y,\bar z)|
 \notag\\
 &\quad\leq C_M\bigl[
 (1+|z|^2+|\bar z|^2)|y-\bar y|
 +(1+|z|+|\bar z|)|z-\bar z|
 \bigr],
 \label{eq:full-Fy-state-continuity}\\
 &|F_z^{\varepsilon;\mathfrak r}(t,y,z)
      -F_z^{\varepsilon;\mathfrak r}(t,\bar y,\bar z)|
 \notag\\
 &\quad\leq C_M\bigl[
 (1+|z|+|\bar z|)|y-\bar y|+|z-\bar z|
 \bigr].
 \label{eq:full-Fz-state-continuity}
\end{align}
All constants are uniform over changing tangents and representatives subject
to the common bound \eqref{eq:coefficient-realization-budget}.  No
uniformity over representatives with unbounded realization budgets is
asserted.
\end{proposition}

\begin{proof}
Estimate \eqref{eq:intrinsic-uniform-four-jet} and continuity of the
intrinsic jet give
\begin{align}
 &|\phi^{\varepsilon;\mathfrak r}-\phi|
 +|J^{\varepsilon;\mathfrak r}-J|
 +|K^{\varepsilon;\mathfrak r}-K|
 +|L^{\varepsilon;\mathfrak r}-L|
 \leq C|\varepsilon|,
 \label{eq:jet-first-difference}\\
 &|u_I^\tau|+|j_I^\tau|+|k_I^\tau|+|\ell_I^\tau|
 \leq C\|\tau\|_{\mathbb T_{\mathbf x}^p}
 \label{eq:jet-derivative-bound}
\end{align}
in the corresponding uniform spatial-jet norm.  Put
\[
 w^{\varepsilon;\mathfrak r}
 :=(\phi^{\varepsilon;\mathfrak r},
     J^{\varepsilon;\mathfrak r}z),
 \qquad
 w^0:=(\phi,Jz).
\]
Then
\begin{align}
 |w^{\varepsilon;\mathfrak r}-w^0|
 &\leq C|\varepsilon|(1+|z|),
 \label{eq:w-first-difference}\\
 |w^{\varepsilon;\mathfrak r}-w^0
   -\varepsilon(u_I^\tau,j_I^\tau z)|
 &\leq C\varepsilon^2(1+|z|).
 \label{eq:w-second-remainder}
\end{align}
Taylor's formula and the bounded second derivative of \(f\) yield
\begin{align}
 &|f(t,w^{\varepsilon;\mathfrak r})-f(t,w^0)
   -\varepsilon Df(t,w^0)(u_I^\tau,j_I^\tau z)|
 \notag\\
 &\hspace{45mm}\leq C_M\varepsilon^2(1+|z|^2).
 \label{eq:f-composition-remainder}
\end{align}

On the range in \eqref{eq:full-jet-J-lower}, the maps
\[
 J\longmapsto J^{-1},
 \qquad
 (J,K)\longmapsto K/J,
 \qquad
 (J,K,L)\longmapsto L/J-K^2/J^2
\]
have bounded first and second derivatives.  Their Taylor expansions,
combined with \eqref{eq:intrinsic-uniform-four-jet} and
\eqref{eq:f-composition-remainder}, prove
\eqref{eq:full-F-quadratic-remainder} and identify its first-order term as
\eqref{eq:full-tangent-generator-derivative}.

Direct differentiation of \eqref{eq:perturbed-transformed-generator} gives
\begin{equation}
 F_z^{\varepsilon;\mathfrak r}
 =f_z(t,\phi^{\varepsilon;\mathfrak r},
          J^{\varepsilon;\mathfrak r}z)
  +\frac{K^{\varepsilon;\mathfrak r}}
         {J^{\varepsilon;\mathfrak r}}z,
 \label{eq:exact-Fz-formula}
\end{equation}
and
\begin{align}
 F_y^{\varepsilon;\mathfrak r}
 ={}&f_y(t,\phi^{\varepsilon;\mathfrak r},
            J^{\varepsilon;\mathfrak r}z)
 +\frac{K^{\varepsilon;\mathfrak r}}
        {J^{\varepsilon;\mathfrak r}}
  f_z(t,\phi^{\varepsilon;\mathfrak r},
          J^{\varepsilon;\mathfrak r}z)\mathbin{\cdot}z
 \notag\\
 &-\frac{K^{\varepsilon;\mathfrak r}}
         {(J^{\varepsilon;\mathfrak r})^2}
  f(t,\phi^{\varepsilon;\mathfrak r},
      J^{\varepsilon;\mathfrak r}z)
 \notag\\
 &+\frac12\left(
  \frac{L^{\varepsilon;\mathfrak r}}
       {J^{\varepsilon;\mathfrak r}}
  -\frac{(K^{\varepsilon;\mathfrak r})^2}
       {(J^{\varepsilon;\mathfrak r})^2}
 \right)|z|^2.
 \label{eq:exact-Fy-formula}
\end{align}
The bounded second derivative of \(f\),
\eqref{eq:jet-first-difference}, and the rational jet bounds give
\eqref{eq:full-Fy-parameter-bound} and
\eqref{eq:full-Fz-parameter-bound}.  Differentiating the right-hand sides
of \eqref{eq:exact-Fz-formula}--\eqref{eq:exact-Fy-formula} with respect to
the state variables gives the weights in
\eqref{eq:full-Fy-state-continuity}--
\eqref{eq:full-Fz-state-continuity}.  The only new deterministic quantity is
\(\partial_y L^{\varepsilon;\mathfrak r}\), which is the uniformly bounded
fourth spatial derivative of the flow.  The bounded derivatives of \(f\)
from Assumption~\ref{ass:coefficients} control the remaining terms.

The intrinsic flow derivative is continuous linear and independent of the
representative.  Formula
\eqref{eq:full-tangent-generator-derivative} therefore has the same
properties.  Bound \eqref{eq:dot-F-continuous-bound} follows from
\eqref{eq:jet-derivative-bound}, the bounded-strip growth of \(f\), and the
uniform rational jet bounds.  Every estimate used above is uniform under
\eqref{eq:coefficient-realization-budget}, which proves the stated
quantifiers.
\end{proof}

\begin{corollary}[Changing tangents and representatives]
\label{cor:changing-tangent-generator}
Let \(\varepsilon_n\to0\), \(\varepsilon_n\neq0\), and
\(\tau_n\to\tau\) in \(\mathbb T_{\mathbf x}^p\).  Choose
\(\mathfrak r_n\in\operatorname{Rep}_{\mathbf x}(\tau_n)\) with
\(\sup_n\mathfrak b(\mathfrak r_n)<\infty\).  Then, for every
\(M<\infty\),
\begin{equation}
 \left\|
 \frac{F^{\varepsilon_n;\mathfrak r_n}-F^0}{\varepsilon_n}
 -\dot F_I[\tau]
 \right\|_{\mathcal W_{M,2}}
 \leq C_M|\varepsilon_n|
      +C_M\|\tau_n-\tau\|_{\mathbb T_{\mathbf x}^p}
 \longrightarrow0.
 \label{eq:changing-tangent-generator-limit}
\end{equation}
The limit does not depend on the chosen bounded representatives.
\end{corollary}

\begin{proof}
Add and subtract \(\dot F_I[\tau_n]\).  The first difference is controlled by
\eqref{eq:full-F-quadratic-remainder} after division by
\(|\varepsilon_n|\).  The second is controlled by
\eqref{eq:dot-F-continuous-bound} and linearity.
\end{proof}
\section{Local analysis of the transformed quadratic BSDE}
\label{sec:local-transformed-bsde}

The transformed equation has a different stability structure from the
original Lipschitz BSDE.  Its derivative in the martingale variable is a BMO
coefficient, while the positive part of its derivative in the state variable
may contain a term of order \(\lvert Z\rvert^2\).  The latter is critical for
conditional exponential integrability.  Results for stochastic Lipschitz
coefficients of order \(K^{2\alpha}\), \(\alpha<1\), do not cover this case
directly; compare \cite[Theorem~3.5]{BriandConfortola2008}.  The purpose of this
section is to isolate the probability estimates that remain valid at the
critical exponent.  The coefficient remainder needed for differentiability is
identified at the end of the section and is not assumed here.

\subsection{Local stochastic norms}
\label{subsec:probability-input}

We work on a filtered probability space
\[
 (\Omega,\mathcal F,\mathbb F,\mathbb P),\qquad
 \mathbb F=(\mathcal F_t)_{0\leq t\leq T},
\]
satisfying the usual conditions.  The filtration is the usual augmentation of
the filtration generated by an \(m\)-dimensional Brownian motion \(W\).  More
generally, all arguments below require the martingale representation property
with respect to \(W\).  This assumption excludes an additional martingale
orthogonal to \(W\) from the BSDE.  The representation property is preserved
under each equivalent Girsanov measure used below.

For an interval \(I=[u,v]\), let \(\mathcal T_I\) denote the stopping times
with values in \(I\).  We use
\[
 \|Z\|_{\mathbb H^2_{\mathrm{BMO}}(\mathbb P;I)}^2
 :=\sup_{\tau\in\mathcal T_I}
 \left\|
  \mathbb E_\tau^{\mathbb P}\int_\tau^v\lvert Z_s\rvert^2\,ds
 \right\|_\infty .
\]
Thus the BMO object is the martingale \(Z\mathbin{\cdot}W\), rather than the
integrand \(Z\) by itself.

On a reset interval, the transformed equation is
\begin{equation}
 \widetilde Y_t
 =\zeta+\int_t^v
 F(s,\widetilde Y_s,\widetilde Z_s)\,ds
 -\int_t^v\widetilde Z_s\,dW_s .
 \label{eq:transformed-local-bsde}
\end{equation}
On a common reset interval, the deterministic flow estimates give
\begin{equation}
 \sup_{t\in I,\,y\in\mathbb R}
 \left(
  |\phi_t(y)-y|+|J_t(y)|+|J_t(y)^{-1}|+|K_t(y)|
 \right)\leq C_{\mathrm{fl}}.
 \label{eq:global-reset-flow-bound}
\end{equation}
Since the original generator is globally Lipschitz, there are constants
\(a_0,a_1,a_2,\gamma,\ell_z\geq0\) such that, for all \(y,z\),
\begin{align}
 \lvert F(t,y,z)\rvert
 &\leq a_0+a_1\lvert y\rvert+a_2\lvert z\rvert
       +\frac{\gamma}{2}\lvert z\rvert^2,
 \label{eq:local-quadratic-growth}\\
 \lvert F_z(t,y,z)\rvert
 &\leq \ell_z+\gamma\lvert z\rvert .
 \label{eq:local-fz-growth}
\end{align}
The constants are uniform over the fixed rough, tangent, and realization
budgets and over the reset intervals in the common partition.

\subsection{Bounded solutions and the internal BMO estimate}
\label{subsec:bounded-solution-bmo}

To verify the growth hypothesis in
\cite[Theorem~2.3]{Kobylanski2000}, let
\begin{equation}
 \mathfrak a(t,y,z)
 :=
 \begin{cases}
  \Pi_{[-a_1,a_1]}\!\bigl(F(t,y,z)/y\bigr),&y\neq0,\\
  0,&y=0,
 \end{cases}
 \qquad
 F_0(t,y,z):=F(t,y,z)-\mathfrak a(t,y,z)y,
 \label{eq:kobylanski-projection-split}
\end{equation}
where \(\Pi_{[-a_1,a_1]}\) denotes metric projection onto the indicated
interval.  For \(y\neq0\), put \(r=F(t,y,z)/y\).  From
\eqref{eq:local-quadratic-growth},
\[
 \operatorname{dist}\!\bigl(r,[-a_1,a_1]\bigr)
 =\bigl(|r|-a_1\bigr)^+
 \leq
 \frac{a_0+a_2|z|+\frac{\gamma}{2}|z|^2}{|y|}.
\]
Therefore
\[
 |F_0(t,y,z)|
 =|y|\,
  \operatorname{dist}\!\bigl(r,[-a_1,a_1]\bigr)
 \leq a_0+a_2|z|+\frac{\gamma}{2}|z|^2.
\]
The same inequality at \(y=0\) follows directly from
\eqref{eq:local-quadratic-growth}.  Absorbing the linear \(z\)-term gives
\begin{equation}
 |F_0(t,y,z)|
 \leq a_0+\frac{a_2^2}{2}
       +\frac{\gamma+1}{2}|z|^2,
 \qquad
 -a_1\leq\mathfrak a(t,y,z)\leq a_1.
 \label{eq:kobylanski-F0-bound}
\end{equation}
Moreover,
\[
 \mathfrak a(t,y,z)y
 =\Pi_{[-a_1|y|,\,a_1|y|]}\!\bigl(F(t,y,z)\bigr).
\]
The Carath\'eodory measurability of \(F\) and the defining formula make
\(\mathfrak a\) jointly progressive--Borel measurable and, in particular,
progressively measurable for each \((y,z)\).  The displayed projection
identity also shows that \(\mathfrak a y\), and hence \(F_0\), is jointly
progressive--Borel measurable and continuous in \((y,z)\), including at
\(y=0\).  Condition (H1) requires the decomposition coefficient only to be
jointly measurable and bounded between deterministic constants; the
continuity hypothesis in Theorem~2.3 is imposed on the generator \(F\),
which is already Carath\'eodory.  Thus the possible discontinuity of
\(\mathfrak a\) at \(y=0\) is immaterial.  Hence (H1) holds, for example,
with
\[
 \beta_0=-a_1,\qquad
 \alpha_0=a_1,\qquad
 b=a_0+\frac{a_2^2}{2},\qquad
 c(r)=\frac{\gamma+1}{2}+r.
\]
Applying Theorem~2.3 to the time-shifted equation on \([0,v-u]\), together
with \(\zeta\in L^\infty(\mathcal F_v)\), gives at least one bounded scalar
solution of the local equation on the bounded deterministic interval
\(I=[u,v]\).  The following estimate makes its state bound independent of
the particular reset interval and rough perturbation.

\begin{proposition}[Uniform scalar state estimate]
\label{prop:uniform-scalar-state}
Suppose \eqref{eq:local-quadratic-growth} holds and
\(\|\zeta\|_\infty\leq M_\zeta\).  Every bounded solution of
\eqref{eq:transformed-local-bsde} satisfies
\begin{equation}
 \|\widetilde Y\|_{\mathbb S^\infty(I)}
 \leq e^{a_1|I|}\bigl(M_\zeta+\bar a_0|I|\bigr),
 \qquad
 \bar a_0:=a_0+\frac{a_2^2}{2},
 \quad \bar\gamma:=\gamma+1.
 \label{eq:uniform-scalar-state}
\end{equation}
To initialize the construction, fix a preliminary deterministic reset
partition
\[
 \Pi^{\mathrm{ex}}
 =\{0=s_0<s_1<\cdots<s_{N_{\mathrm{ex}}}=T\}
\]
using only the deterministic rough-flow estimates.  It may be chosen
uniformly over the fixed rough, tangent, and realization budgets, after
shrinking the perturbation radius if necessary.  Let \(D_{\mathrm{ex}}\)
be a deterministic bound such that
\begin{equation}
 \sup_{\substack{I_j=[s_j,s_{j+1}]\in\Pi^{\mathrm{ex}}\\
        \text{base or admissible perturbed flow}}}
 \ \sup_{t\in I_j,\,y\in\mathbb R}
 |\phi_t^{I_j}(y)-y|
 \leq D_{\mathrm{ex}}.
 \label{eq:preliminary-flow-displacement}
\end{equation}
Define
\begin{align}
 \widetilde M_{\mathrm{ex}}
 &:=e^{a_1T}\left(
      \|\xi\|_\infty+\bar a_0T
      +N_{\mathrm{ex}}D_{\mathrm{ex}}\right),
 \label{eq:preliminary-transformed-state-bound}\\
 M_Y&:=\widetilde M_{\mathrm{ex}}+D_{\mathrm{ex}},
 \qquad
 M_*:=e^{a_1T}\bigl(M_Y+\bar a_0T\bigr).
 \label{eq:global-uniform-state-bound}
\end{align}
Backward construction on \(\Pi^{\mathrm{ex}}\) then gives
\(\|\widetilde Y\|_{\mathbb S^\infty(I_j)}
\leq\widetilde M_{\mathrm{ex}}\) and
\(\|Y\|_{\mathbb S^\infty([0,T])}\leq M_Y\).
On every later deterministic refinement, any bounded local transformed
solution whose terminal value is the already constructed original-coordinate
solution at the right endpoint satisfies
\(\|\widetilde Y\|_{\mathbb S^\infty(I)}\leq M_*\).
In particular, \(M_*\) is fixed before the critical BMO partition is chosen
and does not depend on the number of intervals in that later refinement.
\end{proposition}

\begin{proof}
The elementary inequality
\(
 a_2|z|\leq a_2^2/2+|z|^2/2
\)
turns \eqref{eq:local-quadratic-growth} into
\[
 |F(t,y,z)|
 \leq \bar a_0+a_1|y|+\frac{\bar\gamma}{2}|z|^2.
\]
For \(t\in I=[u,v]\), set
\[
 R_t:=e^{a_1(t-u)}|\widetilde Y_t|
       +\bar a_0\int_u^t e^{a_1(s-u)}\,ds.
\]
Tanaka's formula and the preceding bound give, after localization,
\[
 dR_t
 \geq -\frac{\bar\gamma}{2}e^{a_1(t-u)}
          |\widetilde Z_t|^2dt
       +e^{a_1(t-u)}\operatorname{sgn}(\widetilde Y_t)
          \widetilde Z_t\,dW_t
       +e^{a_1(t-u)}dL_t^0(\widetilde Y).
\]
Since \(e^{2a_1(t-u)}\geq e^{a_1(t-u)}\), It\^o's formula shows that
\(\exp(\bar\gamma R)\) is a local submartingale.  The solution under
consideration is bounded, so the stopped exponentials are uniformly
integrable.  Conditional expectation at time \(t\), followed by the bound
on \(\zeta\), yields
\[
 R_t\leq e^{a_1(v-u)}M_\zeta
       +\bar a_0\int_u^v e^{a_1(s-u)}\,ds.
\]
This proves \eqref{eq:uniform-scalar-state}.  It remains to track the change
of reset coordinates at the interfaces.  Write
\(m_j=\|\widetilde Y\|_{\mathbb S^\infty(I_j)}\) on
\(\Pi^{\mathrm{ex}}\).  For \(j<N_{\mathrm{ex}}-1\), the terminal value for
\(I_j\) is the original-coordinate value reconstructed at the left endpoint
of \(I_{j+1}\); hence
\[
 \|\zeta_j\|_\infty
 \leq m_{j+1}+D_{\mathrm{ex}},
 \qquad
 \|\zeta_{N_{\mathrm{ex}}-1}\|_\infty
 \leq\|\xi\|_\infty.
\]
Iterating \eqref{eq:uniform-scalar-state}, bounding every accumulated
exponential factor by \(e^{a_1T}\), and using
\(\sum_j|I_j|=T\) gives
\eqref{eq:preliminary-transformed-state-bound}.  Reconstruction and
\eqref{eq:preliminary-flow-displacement} give
\(\|Y\|_{\mathbb S^\infty([0,T])}\leq M_Y\).

Now let \(I=[u,v]\) belong to any later deterministic refinement and take
its terminal value to be the constructed original-coordinate value \(Y_v\).
Then \(\|\zeta\|_\infty\leq M_Y\), so another application of
\eqref{eq:uniform-scalar-state} gives
\[
 \|\widetilde Y\|_{\mathbb S^\infty(I)}
 \leq e^{a_1|I|}\bigl(M_Y+\bar a_0|I|\bigr)
 \leq M_*.
\]
This proves \eqref{eq:global-uniform-state-bound} without using the number
or mesh of the later critical partition.
\end{proof}

The use of \cite[Theorem~2.3]{Kobylanski2000} above is an existence
argument.  It does not provide uniqueness for the present transformed
generator.  In particular, the fixed small quadratic coefficient in the
bound on \((F_y)^+\) does not verify the infinitesimal condition in
\cite[Theorem~2.6]{Kobylanski2000}, which requires a bound with
\(\varepsilon|z|^2\) for every \(\varepsilon>0\).  Uniqueness will instead
follow from Proposition~\ref{prop:critical-linear-bsde}.  The BMO estimate
below depends only on \eqref{eq:local-quadratic-growth} and the common bound
\eqref{eq:global-uniform-state-bound}.

\begin{proposition}[Stopping-time BMO estimate]
\label{prop:stopping-time-bmo}
Let \((\widetilde Y,\widetilde Z)\) be a bounded solution of
\eqref{eq:transformed-local-bsde}, with
\(\|\widetilde Y\|_{\mathbb S^\infty(I)}\leq M_*\), and suppose that
\eqref{eq:local-quadratic-growth} holds.  For every \(\lambda>\gamma\),
\begin{equation}
 \|\widetilde Z\|_{\mathbb H^2_{\mathrm{BMO}}(\mathbb P;I)}^2
 \leq \mathcal B_\lambda^2,
 \label{eq:uniform-z-bmo}
\end{equation}
where one admissible bound is
\begin{equation}
 \mathcal B_\lambda^2
 =\frac{4e^{2\lambda M_*}}{\lambda(\lambda-\gamma)}
 \left[
  2+\lambda T
  \left(a_0+a_1M_*+\frac{a_2^2}{\lambda-\gamma}\right)
 \right].
 \label{eq:explicit-z-bmo}
\end{equation}
\end{proposition}

\begin{proof}
Fix \(\tau\in\mathcal T_I\), localize the stochastic integral, and apply
It\^o's formula to \(e^{\lambda\widetilde Y}\) between \(\tau\) and \(v\).
Since
\[
 a_2\lvert z\rvert
 \leq \frac{\lambda-\gamma}{4}\lvert z\rvert^2
      +\frac{a_2^2}{\lambda-\gamma},
\]
the finite-variation part is bounded from below by
\[
 e^{\lambda\widetilde Y_s}
 \left\{
  \frac{\lambda(\lambda-\gamma)}{4}
       \lvert\widetilde Z_s\rvert^2
  -\lambda\left(
    a_0+a_1M_*+\frac{a_2^2}{\lambda-\gamma}
   \right)
 \right\}ds.
\]
Taking conditional expectations, using
\(e^{-\lambda M_*}\leq e^{\lambda\widetilde Y}\leq e^{\lambda M_*}\),
using \(|I|\leq T\), and then removing the localization gives
\eqref{eq:explicit-z-bmo}.  The passage to the limit uses Fatou's lemma for
the nonnegative energy term.
\end{proof}

For later use, consider two solutions in the same bounded family and define
the mean coefficient
\begin{equation}
 \beta_s
 =\int_0^1 F_z\bigl(
     s,Y_s^0+\theta(Y_s^1-Y_s^0),
       Z_s^0+\theta(Z_s^1-Z_s^0)
   \bigr)\,d\theta .
 \label{eq:z-mean-coefficient}
\end{equation}
Equations \eqref{eq:local-fz-growth} and
\eqref{eq:uniform-z-bmo} imply
\begin{equation}
 \left\|\int_u^\cdot\beta_s\,dW_s\right\|_{\mathrm{BMO}_2(\mathbb P;I)}
 \leq n_*,
 \qquad
 n_*^2:=2\ell_z^2T+8\gamma^2\mathcal B_\lambda^2.
 \label{eq:mean-coefficient-bmo}
\end{equation}
This estimate is for the mean coefficient itself.  A BMO estimate for
\(F_z(t,Y_t,Z_t)\) at one solution would not be sufficient for the
difference equation.

The mean \(y\)-coefficient requires an absolute estimate, rather than only a
one-sided bound.  For the same pair of solutions, and for any fixed
member \(\widehat F\) of the base or perturbed generator family, set
\[
 \alpha_s^{0,1}
 :=\int_0^1 \widehat F_y\bigl(
 s,Y_s^0+\theta(Y_s^1-Y_s^0),
 Z_s^0+\theta(Z_s^1-Z_s^0)
 \bigr)\,d\theta .
\]
On the common bounded state strip, the exact formula
\eqref{eq:exact-Fy-formula}, the uniform flow-jet bounds, and the convexity
of \(z\mapsto |z|^2\) give
\begin{equation}
 |\alpha_s^{0,1}|
 \leq C_\alpha\bigl(1+|Z_s^0|^2+|Z_s^1|^2\bigr).
 \label{eq:mean-alpha-absolute-growth}
\end{equation}
For \(j=0,1\), the BMO energy bound, applied at the deterministic stopping
time \(u\), gives
\[
 \mathbb E\int_u^v|Z_s^j|^2\,ds
 \leq
 \|Z^j\mathbin{\cdot}W\|_{\mathrm{BMO}_2(\mathbb P;I)}^2
 <\infty .
\]
Each nonnegative energy integral is therefore finite almost surely, and
\begin{equation}
 \int_u^v|\alpha_s^{0,1}|\,ds
 \leq C_\alpha\left(
 |I|+\int_u^v|Z_s^0|^2\,ds+\int_u^v|Z_s^1|^2\,ds
 \right)<\infty
 \quad\mathbb P\text{-a.s.}
 \label{eq:mean-alpha-pathwise-integrability}
\end{equation}
The same pathwise assertion holds under every Girsanov measure used below,
because those measures are equivalent to \(\mathbb P\).

\subsection{Uniform changes of measure}
\label{subsec:uniform-measure-change}

The next proposition states the uniform form used below.  Its proof is given
in Appendix~\ref{app:bmo-reverse-holder}.

\begin{proposition}[Uniform BMO--Girsanov package]
\label{prop:uniform-bmo-girsanov}
Let \(\{N^\theta\}_{\theta\in\Theta}\) be continuous martingales on \(I\)
such that
\[
 \sup_{\theta\in\Theta}
 \|N^\theta\|_{\mathrm{BMO}_2(\mathbb P;I)}\leq n_*.
\]
For \(u\leq t\leq v\), set
\[
 L_t^\theta
 :=\exp\left(
  N_t^\theta-N_u^\theta
  -\frac12(\langle N^\theta\rangle_t-
              \langle N^\theta\rangle_u)
 \right),
\]
so that \(L_u^\theta=1\), and define
\(d\mathbb Q^\theta/d\mathbb P=L_v^\theta\) on \(\mathcal F_v\).
There exist exponents \(r_+,r_->1\) and constants
\(C_+,C_-,C_{\mathrm{eq}}<\infty\), depending only on \(n_*\), such that,
uniformly in \(\theta\) and \(\tau\in\mathcal T_I\),
\begin{align}
 \mathbb E_\tau^{\mathbb P}
 \left[\left(\frac{L_v^\theta}{L_\tau^\theta}\right)^{r_+}\right]
 &\leq C_+,
 \label{eq:forward-reverse-holder}\\
 \mathbb E_\tau^{\mathbb Q^\theta}
 \left[\left(\frac{L_\tau^\theta}{L_v^\theta}\right)^{r_-}\right]
 &\leq C_-.
 \label{eq:inverse-reverse-holder}
\end{align}
Moreover, if \(X\) is a continuous \(\mathbb P\)-local martingale and
\[
 \widetilde X^\theta
 :=X-\langle X,N^\theta\rangle,
\]
then
\begin{equation}
 C_{\mathrm{eq}}^{-1}\|X\|_{\mathrm{BMO}_2(\mathbb P;I)}
 \leq
 \|\widetilde X^\theta\|_{\mathrm{BMO}_2(\mathbb Q^\theta;I)}
 \leq
 C_{\mathrm{eq}}\|X\|_{\mathrm{BMO}_2(\mathbb P;I)}.
 \label{eq:bmo-norm-equivalence}
\end{equation}
The exponents are selected after the budget \(n_*\) is fixed.  The
proposition does not assert \(r_+=r_-=2\), nor does it prescribe the same
exponent under the two measures.
\end{proposition}

The proof of Proposition~\ref{prop:uniform-bmo-girsanov} uses three
specific measure-change results.  Lemma~1.2(1)--(2) of
\cite{AnkirchnerImkellerDosReis2007} gives the true density and the BMO
property of the self Girsanov transform.  Proposition~1 of
\cite{Kazamaki1983} gives a reverse-H\"older exponent greater than one.
The isomorphism and its norm bounds are also stated in
\cite[Theorems~1--2]{ChikvinidzeMania2014}.

\subsection{A critical local linear estimate}
\label{subsec:critical-linear-estimate}

For \(p>1\), an equivalent measure \(\mathbb Q\), and a progressively
measurable process \(r\), define
\begin{equation}
 \|r\|_{\mathfrak R_p(\mathbb Q;I)}
 :=\sup_{\tau\in\mathcal T_I}
 \left\|
  \mathbb E_\tau^{\mathbb Q}
  \left[\left(\int_\tau^v\lvert r_s\rvert\,ds\right)^p\right]
 \right\|_\infty^{1/p}.
 \label{eq:conditional-source-norm}
\end{equation}

Consider the linear BSDE
\begin{equation}
 R_t
 =\eta+\int_t^v
  \bigl(\alpha_sR_s+\beta_s\mathbin{\cdot}S_s+r_s\bigr)\,ds
  -\int_t^v S_s\,dW_s .
 \label{eq:critical-linear-bsde}
\end{equation}
The coefficient \(\alpha\) is allowed to be critical.  The required
assumption is placed on its integrating factor, not on a subcritical power
of a stochastic Lipschitz coefficient.

\begin{proposition}[Critical local linear BSDE]
\label{prop:critical-linear-bsde}
Let \(\alpha\), \(\beta\), and \(r\) be progressively measurable on
\(\Omega\times I\), with values in \(\mathbb R\), \(\mathbb R^m\), and
\(\mathbb R\), respectively.  Assume that \(\beta\mathbin{\cdot}W\in
\mathrm{BMO}_2(\mathbb P;I)\), and let
\[
 L_t^\beta
 :=\exp\left(
  \int_u^t\beta_s\,dW_s
  -\frac12\int_u^t\lvert\beta_s\rvert^2\,ds
 \right),\qquad
 \frac{d\mathbb Q}{d\mathbb P}\bigg|_{\mathcal F_v}=L_v^\beta .
\]
Thus \(L_u^\beta=1\).  Suppose the BMO--Girsanov constants are those of
Proposition~\ref{prop:uniform-bmo-girsanov}.  Let
\(r_A>1\), set \(p_A=r_A/(r_A-1)\), and assume
\begin{equation}
 \sup_{\tau\in\mathcal T_I}
 \left\|
  \mathbb E_\tau^{\mathbb Q}
  \exp\left(r_A\int_\tau^v\alpha_s^+\,ds\right)
 \right\|_\infty
 \leq K_A<\infty .
 \label{eq:critical-exponential-budget}
\end{equation}
Assume also that \(\int_u^v\lvert\alpha_s\rvert ds<\infty\) almost surely,
\(\eta\in L^\infty(\mathcal F_v)\), and
\(r\in\mathfrak R_{p_A}(\mathbb Q;I)\).  Then
\eqref{eq:critical-linear-bsde} has a unique solution in
\(\mathbb S^\infty(I)\times\mathbb H^2_{\mathrm{BMO}}(\mathbb P;I)\), and
\begin{equation}
 \|R\|_{\mathbb S^\infty(I)}
 \leq K_A^{1/r_A}
 \left(
  \|\eta\|_\infty+
  \|r\|_{\mathfrak R_{p_A}(\mathbb Q;I)}
 \right).
 \label{eq:critical-linear-sup-bound}
\end{equation}
Under \(\mathbb Q\), one admissible BMO bound is
\begin{equation}
 \|S\mathbin{\cdot}W^{\mathbb Q}\|_{
       \mathrm{BMO}_2(\mathbb Q;I)}^2
 \leq
 \|\eta\|_\infty^2
 +\frac{2\|R\|_\infty^2}{r_A}\log K_A
 +2\|R\|_\infty
   \|r\|_{\mathfrak R_{p_A}(\mathbb Q;I)}.
 \label{eq:critical-linear-bmo-bound}
\end{equation}
Consequently, \eqref{eq:bmo-norm-equivalence} gives a deterministic bound
for \(S\mathbin{\cdot}W\) under \(\mathbb P\).
\end{proposition}

The proof is given in Appendix~\ref{subsec:proof-critical-linear}.  It uses
the Girsanov transform, a variation-of-constants formula, the martingale
representation property, and a conditional It\^o estimate.  The
measure-change step parallels
\cite[Theorem~3.1]{AnkirchnerImkellerDosReis2007}, while
Proposition~\ref{prop:critical-linear-bsde} also admits the critical positive
coefficient in \eqref{eq:critical-exponential-budget}.

In the application, the mean coefficient in the state variable satisfies
\begin{equation}
 \alpha_s^+
 \leq a_s+\delta
 \left(\lvert Z_s^0\rvert^2+\lvert Z_s^1\rvert^2\right),
 \qquad
 \int_u^v a_s\,ds\leq A_I,
 \label{eq:critical-alpha-structure}
\end{equation}
where \(A_I\) is deterministic.  If the two martingale integrands have
\(\mathbb Q\)-BMO norm at most \(b_{\mathbb Q}\), the conditional energy
estimate in Appendix~\ref{subsec:conditional-energy} gives
\begin{equation}
 \sup_{\tau\in\mathcal T_I}
 \left\|
 \mathbb E_\tau^{\mathbb Q}
 \exp\left(r_A\int_\tau^v\alpha_s^+\,ds\right)
 \right\|_\infty
 \leq
 \frac{e^{r_AA_I}}
      {1-2r_A\delta b_{\mathbb Q}^2},
 \label{eq:critical-alpha-explicit-budget}
\end{equation}
provided
\begin{equation}
 2r_A\delta b_{\mathbb Q}^2<1.
 \label{eq:critical-reset-smallness}
\end{equation}
The common reset partition is chosen so that
\eqref{eq:critical-reset-smallness} holds uniformly.  This is a local
smallness condition on the positive quadratic part of \(F_y\); it is not a
consequence of BMO membership alone.

For the difference of two bounded solutions of
\eqref{eq:transformed-local-bsde}, the conditions of
Proposition~\ref{prop:critical-linear-bsde} follow directly:
\eqref{eq:mean-coefficient-bmo} gives
\(\beta\mathbin{\cdot}W\in\mathrm{BMO}_2(\mathbb P;I)\);
\eqref{eq:mean-alpha-pathwise-integrability} gives
\(\int_I|\alpha_s|\,ds<\infty\) almost surely; and
\eqref{eq:critical-alpha-explicit-budget}, under
\eqref{eq:critical-reset-smallness}, gives the required exponential moment
of \(\alpha^+\).  Since the terminal datum and source both vanish, they
belong to \(L^\infty(\mathcal F_v)\) and
\(\mathfrak R_{p_A}(\mathbb Q;I)\), respectively.
Proposition~\ref{prop:critical-linear-bsde} therefore yields uniqueness on
the bounded strip.  This argument does not use the full comparison
hypothesis of \cite[Theorem~2.6]{Kobylanski2000}.

\subsection{Reduction of the local difference quotient}
\label{subsec:remaining-difference-quotient-input}

For a parameterized family \(F^\varepsilon\), write
\[
 P^\varepsilon
 =\frac{\widetilde Y^\varepsilon-\widetilde Y^0}{\varepsilon},
 \qquad
 Q^\varepsilon
 =\frac{\widetilde Z^\varepsilon-\widetilde Z^0}{\varepsilon}.
\]
The exact mean-value equation has the form
\eqref{eq:critical-linear-bsde}, with coefficients
\begin{align*}
 \alpha_s^\varepsilon
 &=\int_0^1 F_y^\varepsilon
   \bigl(s,\widetilde Y_s^0+\theta
       (\widetilde Y_s^\varepsilon-\widetilde Y_s^0),
       \widetilde Z_s^0+\theta
       (\widetilde Z_s^\varepsilon-\widetilde Z_s^0)\bigr)\,d\theta,\\
 \beta_s^\varepsilon
 &=\int_0^1 F_z^\varepsilon
   \bigl(s,\widetilde Y_s^0+\theta
       (\widetilde Y_s^\varepsilon-\widetilde Y_s^0),
       \widetilde Z_s^0+\theta
       (\widetilde Z_s^\varepsilon-\widetilde Z_s^0)\bigr)\,d\theta,
\end{align*}
and parameter source
\begin{equation}
 \chi_s^\varepsilon
 =\frac{
   F^\varepsilon(s,\widetilde Y_s^0,\widetilde Z_s^0)
   -F^0(s,\widetilde Y_s^0,\widetilde Z_s^0)}{\varepsilon}.
 \label{eq:parameter-source}
\end{equation}
The estimates above give uniform bounds for \((P^\varepsilon,Q^\varepsilon)\)
once \(\chi^\varepsilon\) and the terminal quotient are uniformly bounded
in the norms of Proposition~\ref{prop:critical-linear-bsde}.  The Taylor
remainder is treated separately below.

Let \((U,V)\) be the candidate first variation.  The error
\((E^\varepsilon,D^\varepsilon)
=(P^\varepsilon-U,Q^\varepsilon-V)\) has source
\begin{equation}
 \rho^\varepsilon
 =(\alpha^\varepsilon-\alpha)U
  +(\beta^\varepsilon-\beta)\mathbin{\cdot}V
 +(\chi^\varepsilon-\dot F).
 \label{eq:difference-quotient-source-remainder}
\end{equation}
Let \(\eta^\varepsilon\) denote the corresponding terminal quotient minus
the terminal value of \(U\).  Once its coefficient, terminal, and source
assumptions have been verified in
Subsection~\ref{subsec:local-difference-quotient-closure},
Proposition~\ref{prop:critical-linear-bsde} reduces local differentiability
to the estimate
\begin{equation}
 \|\eta^\varepsilon\|_\infty
 +\|\rho^\varepsilon\|_{
       \mathfrak R_{p_A}(\mathbb Q^\varepsilon;I)}
 \longrightarrow0
 \label{eq:required-source-remainder}
\end{equation}
uniformly over the admissible changing tangents and secants.

The BMO estimates reduce the argument to
\eqref{eq:required-source-remainder}.  Its verification uses the following
consequences of the full-tangent flow and transformed-generator analysis:
\begin{enumerate}
 \item a weighted first-order expansion of \(F^\varepsilon\), uniformly on
 bounded state strips, with remainder controlled by
 \(1+\lvert z\rvert^2\);
 \item weighted state-continuity estimates for \(F_y^\varepsilon\) and
 \(F_z^\varepsilon\), with respective weights
 \(1+\lvert z\rvert^2+\lvert\bar z\rvert^2\) and
 \(1+\lvert z\rvert+\lvert\bar z\rvert\);
 \item conditional moment estimates for the products
 \[
  \lvert Z\rvert\lvert Q\rvert,
  \qquad \lvert Q\rvert^2,
  \qquad (1+\lvert Z\rvert+\lvert Q\rvert)\lvert V\rvert;
 \]
 \item convergence of the terminal quotient in \(L^\infty\), with constants
 uniform over the common fixed-endpoint partition.
\end{enumerate}
The next subsection combines these bounds with
Proposition~\ref{prop:critical-linear-bsde} and the terminal interface to
prove the local quadratic difference-quotient estimate.
\subsection{The local quotient estimate}
\label{subsec:local-difference-quotient-closure}
\label{sec:local-bsde}

We now verify the coefficient and conditional-moment bounds listed in
Subsection~\ref{subsec:remaining-difference-quotient-input}.  Fix a reset
interval \(I=[u,v]\).  Write
\((\widetilde Y,\widetilde Z)\) for the base transformed solution and
\((\widetilde Y^{\varepsilon;\mathfrak r},
  \widetilde Z^{\varepsilon;\mathfrak r})\) for the solution associated
with \(F^{\varepsilon;\mathfrak r}\).  The uniform bounded-solution and BMO
estimates of Subsections~\ref{subsec:bounded-solution-bmo}--
\ref{subsec:uniform-measure-change} give
\begin{equation}
 \sup_{\substack{0<|\varepsilon|\leq\varepsilon_0\\
        \|\tau\|_{\mathbb T_{\mathbf x}^p}\leq1,\\
        \mathfrak r\in\operatorname{Rep}_{\mathbf x}(\tau),\ 
        \mathfrak b(\mathfrak r)\leq L_{\mathrm{br}}}}
 \left(
  \|\widetilde Y^{\varepsilon;\mathfrak r}\|_{\mathbb S^\infty(I)}
  +\|\widetilde Z^{\varepsilon;\mathfrak r}\mathbin{\cdot}W
       \|_{\mathrm{BMO}_2(\mathbb P;I)}
 \right)<\infty.
 \label{eq:uniform-local-solution-family}
\end{equation}

Let \(\zeta^{\varepsilon;\mathfrak r}\) and \(\zeta^0\) be the terminal
values on \(I\).  The local statement uses the terminal interface
\begin{equation}
 \delta_\zeta(\varepsilon)
 :=\sup_{\substack{\|\tau\|_{\mathbb T_{\mathbf x}^p}\leq1\\
             \mathfrak r\in\operatorname{Rep}_{\mathbf x}(\tau),\ 
             \mathfrak b(\mathfrak r)\leq L_{\mathrm{br}}}}
 \left\|
  \frac{\zeta^{\varepsilon;\mathfrak r}-\zeta^0}{\varepsilon}
  -\dot\zeta[\tau]
 \right\|_\infty
 \longrightarrow0,
 \label{eq:local-terminal-interface}
\end{equation}
where \(\tau\mapsto\dot\zeta[\tau]\) is continuous linear.  On the final
reset interval, both terms are zero because the terminal condition is fixed.
On earlier intervals, \eqref{eq:local-terminal-interface} is the quantity
passed from the next interval by fixed-endpoint backward patching.

Define
\begin{equation}
 P^{\varepsilon;\mathfrak r}
 :=\frac{\widetilde Y^{\varepsilon;\mathfrak r}-
              \widetilde Y}{\varepsilon},
 \qquad
 Q^{\varepsilon;\mathfrak r}
 :=\frac{\widetilde Z^{\varepsilon;\mathfrak r}-
              \widetilde Z}{\varepsilon}.
 \label{eq:exact-local-quotients}
\end{equation}
For \(\theta\in[0,1]\), set
\[
 Y_s^{\varepsilon,\theta}
 :=\widetilde Y_s+
   \theta(\widetilde Y_s^{\varepsilon;\mathfrak r}-
                    \widetilde Y_s),
 \qquad
 Z_s^{\varepsilon,\theta}
 :=\widetilde Z_s+
   \theta(\widetilde Z_s^{\varepsilon;\mathfrak r}-
                    \widetilde Z_s),
\]
and define
\begin{align}
 \alpha_s^{\varepsilon;\mathfrak r}
 &:=\int_0^1
 F_y^{\varepsilon;\mathfrak r}
  (s,Y_s^{\varepsilon,\theta},Z_s^{\varepsilon,\theta})\,d\theta,
 \label{eq:exact-y-mean}\\
 \beta_s^{\varepsilon;\mathfrak r}
 &:=\int_0^1
 F_z^{\varepsilon;\mathfrak r}
  (s,Y_s^{\varepsilon,\theta},Z_s^{\varepsilon,\theta})\,d\theta,
 \label{eq:exact-z-mean}\\
 \chi_s^{\varepsilon;\mathfrak r}
 &:=\frac{
  F^{\varepsilon;\mathfrak r}
       (s,\widetilde Y_s,\widetilde Z_s)
  -F^0(s,\widetilde Y_s,\widetilde Z_s)}{\varepsilon}.
 \label{eq:exact-parameter-source}
\end{align}
Subtracting the two transformed BSDEs and using the fundamental theorem of
calculus gives the exact quotient equation
\begin{align}
 P_t^{\varepsilon;\mathfrak r}
 ={}&\eta^{\varepsilon;\mathfrak r}
 +\int_t^v\bigl(
   \alpha_s^{\varepsilon;\mathfrak r}
      P_s^{\varepsilon;\mathfrak r}
  +\beta_s^{\varepsilon;\mathfrak r}
      \mathbin{\cdot}Q_s^{\varepsilon;\mathfrak r}
  +\chi_s^{\varepsilon;\mathfrak r}
 \bigr)\,ds
 \notag\\
 &-\int_t^v Q_s^{\varepsilon;\mathfrak r}\,dW_s,
 \label{eq:exact-quotient-bsde}
\end{align}
where
\(\eta^{\varepsilon;\mathfrak r}
 =(\zeta^{\varepsilon;\mathfrak r}-\zeta^0)/\varepsilon\).

Set
\begin{equation}
 \alpha_s:=F_y^0(s,\widetilde Y_s,\widetilde Z_s),
 \qquad
 \beta_s:=F_z^0(s,\widetilde Y_s,\widetilde Z_s).
 \label{eq:base-linear-coefficients}
\end{equation}
The candidate first variation is the solution of
\begin{align}
 U_t^\tau
 ={}&\dot\zeta[\tau]
 +\int_t^v\bigl(
   \alpha_sU_s^\tau+\beta_s\mathbin{\cdot}V_s^\tau
   +\dot F_I[\tau](s,\widetilde Y_s,\widetilde Z_s)
 \bigr)\,ds
 -\int_t^v V_s^\tau\,dW_s.
 \label{eq:candidate-local-linear-bsde}
\end{align}
Let \(\mathbb Q^0\) be the measure generated by
\(\beta\mathbin{\cdot}W\).  The specialization
\(Z^0=Z^1=\widetilde Z\) in
\eqref{eq:mean-coefficient-bmo} and
\eqref{eq:mean-alpha-pathwise-integrability} gives, respectively,
\(\beta\mathbin{\cdot}W\in\mathrm{BMO}_2(\mathbb P;I)\) and
\(\int_I|\alpha_s|\,ds<\infty\) almost surely.  The same base-pair
specialization of \eqref{eq:critical-alpha-explicit-budget} gives the
required exponential moment of \(\alpha^+\).  The terminal interface
\eqref{eq:local-terminal-interface} gives
\(\dot\zeta[\tau]\in L^\infty(\mathcal F_v)\), while
\eqref{eq:dot-F-continuous-bound}, the BMO norm equivalence, and the
conditional energy moments \eqref{eq:conditional-energy-moments} give
\[
 \|\dot F_I[\tau](\cdot,\widetilde Y,\widetilde Z)\|_{
       \mathfrak R_{p_A}(\mathbb Q^0;I)}
 \leq C\|\tau\|_{\mathbb T_{\mathbf x}^p}.
\]
Proposition~\ref{prop:critical-linear-bsde} therefore applies and yields
\begin{equation}
 \|U^\tau\|_{\mathbb S^\infty(I)}
 +\|V^\tau\mathbin{\cdot}W\|_{
       \mathrm{BMO}_2(\mathbb P;I)}
 \leq C\|\tau\|_{\mathbb T_{\mathbf x}^p}.
 \label{eq:candidate-linear-bound}
\end{equation}
Thus \(\tau\mapsto(U^\tau,V^\tau)\) is continuous linear.

Before estimating the error, note from
\eqref{eq:full-F-quadratic-remainder} and
\eqref{eq:dot-F-continuous-bound} that
\begin{equation}
 |\chi_s^{\varepsilon;\mathfrak r}|
 \leq C(1+|\widetilde Z_s|^2).
 \label{eq:quotient-source-uniform-bound}
\end{equation}
For \eqref{eq:exact-quotient-bsde}, take
\(Z^0=\widetilde Z\) and
\(Z^1=\widetilde Z^{\varepsilon;\mathfrak r}\) in the common coefficient
checks.  Equation~\eqref{eq:mean-coefficient-bmo} gives
\(\beta^{\varepsilon;\mathfrak r}\mathbin{\cdot}W\in
\mathrm{BMO}_2(\mathbb P;I)\), and
\eqref{eq:mean-alpha-absolute-growth} together with the two BMO energy
bounds gives
\(\int_I|\alpha_s^{\varepsilon;\mathfrak r}|\,ds<\infty\) almost surely.
The exponential moment of
\((\alpha^{\varepsilon;\mathfrak r})^+\) is
\eqref{eq:critical-alpha-explicit-budget}.  After decreasing
\(\varepsilon_0\), if necessary,
\eqref{eq:local-terminal-interface} and the boundedness of the continuous
linear map \(\tau\mapsto\dot\zeta[\tau]\) give
\[
 \sup_{\varepsilon,\tau,\mathfrak r}
 \|\eta^{\varepsilon;\mathfrak r}\|_\infty<\infty .
\]
Let \(\mathbb Q^{\varepsilon;\mathfrak r}\) be the measure generated by
\(\beta^{\varepsilon;\mathfrak r}\mathbin{\cdot}W\).  Finally,
\eqref{eq:quotient-source-uniform-bound}, the BMO norm equivalence, and the
conditional energy moments \eqref{eq:conditional-energy-moments} give
\[
 \sup_{\varepsilon,\tau,\mathfrak r}
 \|\chi^{\varepsilon;\mathfrak r}\|_{
   \mathfrak R_{p_A}(\mathbb Q^{\varepsilon;\mathfrak r};I)}
 <\infty ,
\]
where the suprema range over the family in
\eqref{eq:uniform-local-solution-family}.
Proposition~\ref{prop:critical-linear-bsde} therefore applies uniformly to
\eqref{eq:exact-quotient-bsde} and yields
\begin{equation}
 \sup_{\substack{0<|\varepsilon|\leq\varepsilon_0\\
        \|\tau\|_{\mathbb T_{\mathbf x}^p}\leq1,\\
        \mathfrak r\in\operatorname{Rep}_{\mathbf x}(\tau),\ 
        \mathfrak b(\mathfrak r)\leq L_{\mathrm{br}}}}
 \left(
  \|P^{\varepsilon;\mathfrak r}\|_{\mathbb S^\infty(I)}
  +\|Q^{\varepsilon;\mathfrak r}\mathbin{\cdot}W\|_{
       \mathrm{BMO}_2(\mathbb P;I)}
 \right)<\infty.
 \label{eq:uniform-quotient-bound}
\end{equation}
This estimate precedes the Taylor-error estimate.

\begin{lemma}[Conditional BMO product estimate]
\label{lem:conditional-bmo-products}
Let \(\mathbb Q\) be one of the preceding Girsanov measures, and let
\(G^1,G^2,G^3\) be predictable integrands such that
\[
 \max_{1\leq j\leq3}
 \|G^j\mathbin{\cdot}W^{\mathbb Q}\|_{
      \mathrm{BMO}_2(\mathbb Q;I)}\leq B.
\]
For every \(s\geq1\), the process
\begin{align*}
 \Theta_t:={}&1+|G_t^1|^2+|G_t^2|^2
 +(1+|G_t^1|)|G_t^2|\\
 &+(1+|G_t^1|+|G_t^2|)|G_t^3|
\end{align*}
satisfies
\begin{equation}
 \|\Theta\|_{\mathfrak R_s(\mathbb Q;I)}
 \leq C_{s,B,|I|}.
 \label{eq:conditional-product-package}
\end{equation}
\end{lemma}

\begin{proof}
For \(A_j(\tau)=\int_\tau^v|G_t^j|^2dt\), the conditional energy estimate
\eqref{eq:conditional-energy-moments} gives all integer moments of
\(A_j(\tau)\), uniformly in \(\tau\).  Conditional Lyapunov inequalities
give the same conclusion for every finite positive exponent.  The terms in
\eqref{eq:conditional-product-package} are controlled by
\begin{align*}
 \int_\tau^v|G_t^i||G_t^j|dt
 &\leq A_i(\tau)^{1/2}A_j(\tau)^{1/2},\\
 \int_\tau^v|G_t^j|dt
 &\leq |I|^{1/2}A_j(\tau)^{1/2}.
\end{align*}
Conditional Cauchy--Schwarz and the energy moments prove
\eqref{eq:conditional-product-package}.
\end{proof}

Define the errors
\begin{equation}
 E^{\varepsilon;\mathfrak r}
 :=P^{\varepsilon;\mathfrak r}-U^\tau,
 \qquad
 D^{\varepsilon;\mathfrak r}
 :=Q^{\varepsilon;\mathfrak r}-V^\tau.
 \label{eq:local-error-processes}
\end{equation}
Subtracting \eqref{eq:candidate-local-linear-bsde} from
\eqref{eq:exact-quotient-bsde} yields
\begin{align}
 E_t^{\varepsilon;\mathfrak r}
 ={}&\eta_{\mathrm{err}}^{\varepsilon;\mathfrak r}
 +\int_t^v\bigl(
  \alpha_s^{\varepsilon;\mathfrak r}
       E_s^{\varepsilon;\mathfrak r}
 +\beta_s^{\varepsilon;\mathfrak r}
       \mathbin{\cdot}D_s^{\varepsilon;\mathfrak r}
 +\rho_s^{\varepsilon;\mathfrak r}
 \bigr)ds
 \notag\\
 &-\int_t^vD_s^{\varepsilon;\mathfrak r}\,dW_s,
 \label{eq:exact-error-bsde}
\end{align}
where
\begin{align}
 \eta_{\mathrm{err}}^{\varepsilon;\mathfrak r}
 &:=\eta^{\varepsilon;\mathfrak r}-\dot\zeta[\tau],
 \label{eq:terminal-error-definition}\\
 \rho^{\varepsilon;\mathfrak r}
 &:=(\alpha^{\varepsilon;\mathfrak r}-\alpha)U^\tau
  +(\beta^{\varepsilon;\mathfrak r}-\beta)
       \mathbin{\cdot}V^\tau
  +(\chi^{\varepsilon;\mathfrak r}-\dot F_I[\tau]).
 \label{eq:exact-error-source}
\end{align}

We estimate the three terms in \eqref{eq:exact-error-source}.  Since
\[
 Y^{\varepsilon,\theta}
 =\widetilde Y+\theta\varepsilon
       P^{\varepsilon;\mathfrak r},
 \qquad
 Z^{\varepsilon,\theta}
 =\widetilde Z+\theta\varepsilon
       Q^{\varepsilon;\mathfrak r},
\]
the parameter bounds
\eqref{eq:full-Fy-parameter-bound}--
\eqref{eq:full-Fz-parameter-bound}, the state bounds
\eqref{eq:full-Fy-state-continuity}--
\eqref{eq:full-Fz-state-continuity}, and
\eqref{eq:uniform-quotient-bound} give
\begin{align}
 |\alpha^{\varepsilon;\mathfrak r}-\alpha|
 &\leq C|\varepsilon|\bigl[
  1+|\widetilde Z|^2
  +|Q^{\varepsilon;\mathfrak r}|^2
  +(1+|\widetilde Z|)|Q^{\varepsilon;\mathfrak r}|
 \bigr],
 \label{eq:alpha-mean-difference}\\
 |\beta^{\varepsilon;\mathfrak r}-\beta|
 &\leq C|\varepsilon|\bigl[
  1+|\widetilde Z|+|Q^{\varepsilon;\mathfrak r}|
 \bigr],
 \label{eq:beta-mean-difference}\\
 |\chi^{\varepsilon;\mathfrak r}-\dot F_I[\tau]|
 &\leq C|\varepsilon|(1+|\widetilde Z|^2).
 \label{eq:chi-source-difference}
\end{align}
The constants absorb the uniform \(\mathbb S^\infty\) bounds for
\(P^{\varepsilon;\mathfrak r}\) and \(U^\tau\).  Combining
\eqref{eq:alpha-mean-difference}--\eqref{eq:chi-source-difference} gives
\begin{align}
 |\rho^{\varepsilon;\mathfrak r}|
 \leq C|\varepsilon|\bigl[&
  1+|\widetilde Z|^2
  +|Q^{\varepsilon;\mathfrak r}|^2
  +(1+|\widetilde Z|)|Q^{\varepsilon;\mathfrak r}|\\
 &+(1+|\widetilde Z|+|Q^{\varepsilon;\mathfrak r}|)|V^\tau|
 \bigr].
 \label{eq:pointwise-source-remainder}
\end{align}

Under the measure \(\mathbb Q^{\varepsilon;\mathfrak r}\) generated by
\(\beta^{\varepsilon;\mathfrak r}\mathbin{\cdot}W\), the uniform
BMO--Girsanov equivalence gives
\begin{align}
 &\|\widetilde Z\mathbin{\cdot}W^{\mathbb Q^{\varepsilon;\mathfrak r}}\|_{
       \mathrm{BMO}_2(\mathbb Q^{\varepsilon;\mathfrak r};I)}
 +\|Q^{\varepsilon;\mathfrak r}\mathbin{\cdot}
       W^{\mathbb Q^{\varepsilon;\mathfrak r}}\|_{
       \mathrm{BMO}_2(\mathbb Q^{\varepsilon;\mathfrak r};I)}
 \notag\\
 &\qquad
 +\|V^\tau\mathbin{\cdot}W^{\mathbb Q^{\varepsilon;\mathfrak r}}\|_{
       \mathrm{BMO}_2(\mathbb Q^{\varepsilon;\mathfrak r};I)}
 \leq B_{\mathrm{loc}}
 \label{eq:three-integrand-q-bmo}
\end{align}
with a deterministic constant.  Apply
Lemma~\ref{lem:conditional-bmo-products} with
\(G^1=\widetilde Z\),
\(G^2=Q^{\varepsilon;\mathfrak r}\), and \(G^3=V^\tau\).  For the conjugate
exponent \(p_A\) in Proposition~\ref{prop:critical-linear-bsde},
\begin{equation}
 \sup_{\substack{\|\tau\|_{\mathbb T_{\mathbf x}^p}\leq1\\
         \mathfrak r\in\operatorname{Rep}_{\mathbf x}(\tau),\ 
         \mathfrak b(\mathfrak r)\leq L_{\mathrm{br}}}}
 \|\rho^{\varepsilon;\mathfrak r}\|_{
   \mathfrak R_{p_A}(\mathbb Q^{\varepsilon;\mathfrak r};I)}
 \leq C|\varepsilon|.
 \label{eq:conditional-source-remainder-closed}
\end{equation}
This proves the conditional source convergence required in
\eqref{eq:required-source-remainder}.

\begin{theorem}[Local full-tangent differentiability]
\label{thm:local-full-tangent-differentiability}
Under Assumption~\ref{ass:coefficients}, the common reset and representative
budgets above, and the terminal interface
\eqref{eq:local-terminal-interface},
\begin{align}
 &\sup_{\substack{\|\tau\|_{\mathbb T_{\mathbf x}^p}\leq1\\
          \mathfrak r\in\operatorname{Rep}_{\mathbf x}(\tau),\ 
          \mathfrak b(\mathfrak r)\leq L_{\mathrm{br}}}}
 \left(
  \|P^{\varepsilon;\mathfrak r}-U^\tau\|_{\mathbb S^\infty(I)}
  +\|(Q^{\varepsilon;\mathfrak r}-V^\tau)
          \mathbin{\cdot}W\|_{\mathrm{BMO}_2(\mathbb P;I)}
 \right)
 \notag\\
 &\hspace{35mm}\leq
 C\bigl(\delta_\zeta(\varepsilon)+|\varepsilon|\bigr)
 \longrightarrow0.
 \label{eq:local-full-tangent-remainder}
\end{align}
The limit and its defining linear BSDE are independent of the representative.
If \(\varepsilon_n\to0\), \(\tau_n\to\tau\), and the representatives have
a common bounded budget, then the local quotients converge to
\((U^\tau,V^\tau)\) whenever the terminal interface holds along the same
sequence.
\end{theorem}

\begin{proof}
The coefficients in \eqref{eq:exact-error-bsde} are the same
\(\alpha^{\varepsilon;\mathfrak r}\) and
\(\beta^{\varepsilon;\mathfrak r}\) as in
\eqref{eq:exact-quotient-bsde}.  Hence
\(\beta^{\varepsilon;\mathfrak r}\mathbin{\cdot}W\) is BMO by
\eqref{eq:mean-coefficient-bmo},
\(\int_I|\alpha_s^{\varepsilon;\mathfrak r}|\,ds<\infty\) almost surely by
\eqref{eq:mean-alpha-absolute-growth} and the two BMO energy bounds, and
\eqref{eq:critical-alpha-explicit-budget} supplies the common exponential
moment of \((\alpha^{\varepsilon;\mathfrak r})^+\).  Moreover,
\[
 \|\eta_{\mathrm{err}}^{\varepsilon;\mathfrak r}\|_\infty
 \leq\delta_\zeta(\varepsilon),
 \qquad
 \|\rho^{\varepsilon;\mathfrak r}\|_{
   \mathfrak R_{p_A}(\mathbb Q^{\varepsilon;\mathfrak r};I)}
 \leq C|\varepsilon|
\]
by \eqref{eq:local-terminal-interface} and
\eqref{eq:conditional-source-remainder-closed}.  Thus the terminal datum is
in \(L^\infty(\mathcal F_v)\), and the source has the required conditional
norm.  Together with the coefficient estimates above, this permits an
application of Proposition~\ref{prop:critical-linear-bsde} to the exact
error equation.  Its \(\mathbb S^\infty\) estimate is of order
\(d_\varepsilon:=\delta_\zeta(\varepsilon)+|\varepsilon|\).  In the squared
BMO estimate \eqref{eq:critical-linear-bmo-bound}, the terminal term and the
term quadratic in \(\|E^{\varepsilon;\mathfrak r}\|_\infty\) are
\(O(d_\varepsilon^2)\).  The cross term is also
\(O(d_\varepsilon^2)\), since the conditional source norm is
\(O(|\varepsilon|)\).  Taking square roots and using the uniform BMO
equivalence between \(\mathbb Q^{\varepsilon;\mathfrak r}\) and
\(\mathbb P\) gives the same order for the martingale integrand.  This
proves \eqref{eq:local-full-tangent-remainder}.

Representative independence follows because \(\dot F_I[\tau]\) and
\(\dot\zeta[\tau]\) are intrinsic and the candidate linear BSDE is unique.
For changing tangents, add and subtract
\((U^{\tau_n},V^{\tau_n})\).  The first difference is controlled by
\eqref{eq:local-full-tangent-remainder}; the second converges by the
continuous linear bound \eqref{eq:candidate-linear-bound}.
\end{proof}

\begin{remark}[Exponent and comparison boundaries]
\label{rem:local-exponent-boundary}
The proof uses the exponents supplied by the uniform BMO budget.  It does not
assume a reverse-H\"older exponent equal to two.  The critical term in the
positive part of \(\alpha^{\varepsilon;\mathfrak r}\) is controlled by the
reset smallness condition \eqref{eq:critical-reset-smallness} and the
conditional energy estimate.  A stochastic-Lipschitz theorem with a
subcritical \(K^{2\alpha}\), \(\alpha<1\), is not used as a substitute for
this step.
\end{remark}
\section{Backward patching and reconstruction}
\label{sec:patching}

This section passes from the local transformed equations to the derivative
of the solution map on the full time interval.  The argument uses the
intrinsic flow variation of Proposition~\ref{prop:intrinsic-flow-variation},
the full-jet and weighted-generator expansions of
Propositions~\ref{prop:uniform-full-jet-expansion} and
\ref{prop:weighted-full-tangent-generator}, the uniform BMO estimates of
Proposition~\ref{prop:uniform-bmo-girsanov}, and the local
difference-quotient result, Theorem~\ref{thm:local-full-tangent-differentiability}.
The role of the present
section is global: it supplies a common deterministic partition, propagates
the terminal derivative backwards, and reconstructs the derivative in the
original coordinates.

For a subinterval $[s,t]$, set
\begin{equation}\label{eq:base-reset-control}
 \omega_{\mathbf x}(s,t)
 :=(t-s)+\|x\|_{p\text{-var};[s,t]}^p
       +\|\mathbb x\|_{q\text{-var};[s,t]}^q.
\end{equation}
This is a deterministic continuous control.

\begin{lemma}[Consistency under deterministic refinement]
\label{lem:partition-consistency}
Let $\Pi$ be a deterministic reset partition on which the local existence
and critical uniqueness estimates hold, and let $\Pi'$ be a deterministic
refinement on which the same estimates hold.  Iteration on $\Pi$ and on
$\Pi'$ gives the same reconstructed pair
$\mathcal S(\mathbf x)$.
\end{lemma}

\begin{proof}
It is enough to insert one point $w$ into an interval $[u,v]$.  The backward
rough flow satisfies the cocycle identity
\begin{equation}\label{eq:backward-flow-cocycle}
 \phi_t^{\mathbf x,[u,v]}
 =\phi_t^{\mathbf x,[u,w]}
   \circ\phi_w^{\mathbf x,[w,v]},
 \qquad u\leq t\leq w,
\end{equation}
and agrees with $\phi^{\mathbf x,[w,v]}$ on $[w,v]$.  Differentiating
\eqref{eq:backward-flow-cocycle} gives the corresponding identities for
$J$ and $K$.  The deterministic change of variables associated with this
composition, followed by It\^o's formula, shows that the restriction of a
transformed solution on $[u,v]$ solves the two reset equations on
$[u,w]$ and $[w,v]$, with the reconstructed value at $w$ as their common
terminal interface.  Conversely, the two reset solutions concatenate to
the one-interval solution.  Uniqueness of the bounded transformed BSDE on
each reset interval identifies the two constructions.  Intervalwise
uniqueness follows from the zero-terminal, zero-source case of
Proposition~\ref{prop:critical-linear-bsde}: the mean \(z\)-coefficient
is BMO by \eqref{eq:mean-coefficient-bmo},
\eqref{eq:mean-alpha-pathwise-integrability} gives
\(\int_I|\alpha_s|\,ds<\infty\) almost surely, the prescribed critical
budget gives the exponential moment of \(\alpha^+\), and the zero terminal
datum and source satisfy the \(L^\infty\) and
\(\mathfrak R_{p_A}\) conditions.  Finite induction
proves the assertion for $\Pi'$.
\end{proof}

Once the critical partition below has been constructed, the lemma shows
that every further deterministic critical refinement leaves the
reconstructed solution unchanged.

\begin{lemma}[A deterministic common reset partition]
\label{lem:common-reset-partition}
Fix $R,B<\infty$.  There are $\varepsilon_0>0$, a deterministic
partition
\begin{equation}\label{eq:common-reset-partition}
 \Pi=\{0=t_0<t_1<\cdots<t_N=T\},
 \qquad N\leq N_*<\infty,
\end{equation}
and constants $r_A>1$, $K_A<\infty$ such that the following assertions
hold simultaneously for every
\[
 \tau\in\mathbb T_{\mathbf x}^p,\qquad
 \|\tau\|_{\mathbb T_{\mathbf x}^p}\leq R,
\]
every $(\mathbf Z,a)\in\operatorname{Rep}_{\mathbf x}(\tau)$ with
$\mathfrak b(\mathbf Z,a)\leq B$, and every
$0<|\varepsilon|\leq\varepsilon_0$.

\begin{enumerate}
 \item On each $I_i=[t_i,t_{i+1}]$, the base flow and the flow driven by
 $\mathbf x^{\varepsilon;\mathbf Z,a}$ remain in the common reset regime.
 In particular, their first spatial derivatives are bounded away from
 zero, and all flow-jet estimates used in
 Propositions~\ref{prop:uniform-full-jet-expansion} and
 \ref{prop:weighted-full-tangent-generator} hold with constants independent of
 $i,\varepsilon,\tau,\mathbf Z$, and $a$.

 \item Let $\mathbb Q_i^\varepsilon$ be the local Girsanov measure generated
 by the mean $z$-coefficient in the transformed difference equation.  The
 corresponding stochastic exponentials satisfy common reverse-H\"older
 bounds, and the base and perturbed martingale integrands have
 $\mathbb Q_i^\varepsilon$-BMO norm at most a deterministic constant
 $b_{\mathbb Q}$.

 \item The positive part of the mean $y$-coefficient satisfies
 \begin{equation}\label{eq:common-critical-structure}
  (\alpha_s^{\varepsilon,i})^+
  \leq a_s^i+\delta_*
  \bigl(
    |\widetilde Z_s^i|^2
    +|\widetilde Z_s^{\varepsilon,i}|^2
  \bigr),
  \qquad
  \int_{t_i}^{t_{i+1}}a_s^i\,ds\leq A_*,
 \end{equation}
 where $A_*<\infty$ is deterministic and
 \begin{equation}\label{eq:common-critical-smallness}
  2r_A\delta_*b_{\mathbb Q}^2<1.
 \end{equation}
 Consequently,
 \begin{equation}\label{eq:common-critical-budget}
  \sup_{\sigma\in\mathcal T_{I_i}}
  \left\|
   \mathbb E_{\sigma}^{\mathbb Q_i^\varepsilon}
   \exp\left(
    r_A\int_\sigma^{t_{i+1}}
       (\alpha_s^{\varepsilon,i})^+\,ds
   \right)
  \right\|_\infty
  \leq
  K_A:=\frac{e^{r_AA_*}}
             {1-2r_A\delta_*b_{\mathbb Q}^2}.
 \end{equation}
\end{enumerate}
The partition and $N_*$ depend on the base driver, the coefficient bounds,
$R$, and $B$, but not on the particular tangent, its representative, or
$\varepsilon$.
\end{lemma}

\begin{proof}
The preliminary construction in
Proposition~\ref{prop:uniform-scalar-state} is completed before the present
partition is chosen.  Restrict its original-coordinate solution to a
candidate subinterval and transform it with the corresponding reset flow.
Equations~\eqref{eq:global-uniform-state-bound} and
\eqref{eq:uniform-z-bmo} give the bounds \(M_*\) and
\(\mathcal B_\lambda\), independently of the candidate refinement.  The
same estimates hold for any competing bounded local solution with the same
terminal interface.

The mean-coefficient estimate
\eqref{eq:mean-coefficient-bmo} and
Proposition~\ref{prop:uniform-bmo-girsanov} now give a deterministic bound
$b_{\mathbb Q}$ for the base and perturbed martingale integrands under
every local Girsanov measure, together with common reverse-H\"older ranges.
Choose $r_A>1$ within those ranges, and then choose $\delta_*>0$ so that
\eqref{eq:common-critical-smallness} holds.

We verify that this prescribed coefficient can be reached by a
deterministic reset.  On $I=[u,v]$, for either the base or a perturbed reset
flow, write
\begin{equation}\label{eq:critical-r-b-definition}
 r:=\frac KJ,
 \qquad
 b:=\partial_y\left(\frac KJ\right)
   =\frac LJ-\frac{K^2}{J^2}.
\end{equation}
All quantities in the next calculation may carry the perturbation
superscript.  The exact identity \eqref{eq:exact-Fy-formula} becomes
\begin{equation}\label{eq:critical-Fy-decomposition}
 F_y(t,y,z)
 =f_y+r f_z\mathbin{\cdot}z-\frac rJ f
  +\frac12 b|z|^2,
\end{equation}
where $f,f_y,f_z$ are evaluated at $(t,\phi_t(y),J_t(y)z)$.  Let $M_*$ be
the common deterministic strip containing the base and perturbed
transformed solutions and their convex interpolants.  On
$I\times[-M_*,M_*]$, the bounds on $f$, $f_y$, and $f_z$, the boundedness
of $\phi$, and the separation of $J$ from zero give deterministic constants
$c_0,C_{M_*}<\infty$ such that
\begin{align}
 (F_y(t,y,z))^+
 &\leq c_0+C_{M_*}(|r(t,y)|+|b(t,y)|)(1+|z|^2)                 \notag\\
 &\leq c_0+\theta_I(1+|z|^2),
 \label{eq:critical-Fy-strip-bound}\\
 \theta_I
 &:=C_{M_*}
   \sup_{(t,y)\in I\times[-M_*,M_*]}(|r(t,y)|+|b(t,y)|).
 \label{eq:critical-theta-definition}
\end{align}
Indeed, the two terms linear in $z$ in
\eqref{eq:critical-Fy-decomposition} are bounded using
$|z|\leq1+|z|^2$, while the last term contributes
$|b||z|^2/2$.

At the right reset endpoint,
\begin{equation}\label{eq:critical-reset-r-b-zero}
 J_v=1,
 \qquad K_v=L_v=0,
 \qquad r_v=b_v=0.
\end{equation}
The reset-flow estimates therefore yield, for the local driver control
$\omega_I$,
\begin{equation}\label{eq:critical-r-b-flow-bound}
 \sup_{(t,y)\in I\times[-M_*,M_*]}(|r(t,y)|+|b(t,y)|)
 \leq C_{M_*}'\omega_I^{1/p}.
\end{equation}
Choose $\eta_*>0$ so that
$C_{M_*}C_{M_*}'\eta_*^{1/p}\leq\delta_*$.  Thus
$\omega_I\leq\eta_*$ implies $\theta_I\leq\delta_*$, for the base flow and
for every perturbed flow in the common bounded family.  The same choice
keeps $J$ separated from zero and controls the spatial jets required by
the full-jet and generator expansions.

Let $\Pi_0$ be the finite union of the preliminary existence partition
\(\Pi^{\mathrm{ex}}\) and the deterministic reset partitions used by the
local flow-jet and BSDE estimates.  Refine each interval of $\Pi_0$ greedily
for the control \eqref{eq:base-reset-control} so that
\begin{equation}\label{eq:base-control-on-pieces}
 \omega_{\mathbf x}(t_i,t_{i+1})\leq\eta
 \quad\text{for every }i,
 \qquad
 N\leq
 N_0+\left\lceil\frac{\omega_{\mathbf x}(0,T)}{\eta}\right\rceil
 =:N_*,
\end{equation}
where $N_0$ is the number of intervals in $\Pi_0$ and $\eta>0$ is chosen
so that the base contribution in
\eqref{eq:perturbed-local-control} is at most $\eta_*/2$.  This common
refinement is deterministic.  The inclusion
\(\Pi^{\mathrm{ex}}\subset\Pi\) will allow the preliminary bounded
construction to be identified after local uniqueness has been established.
The two levels of a radial realization satisfy, uniformly
on the indicated tangent and representative sets,
\begin{equation}\label{eq:perturbed-local-control}
 \omega_{\varepsilon,\tau,\mathbf Z,a}(s,t)
 \leq C\left(
   \omega_{\mathbf x}(s,t)
   +|\varepsilon|^pR^p
   +|\varepsilon|^qR^q
   +|\varepsilon|^pB^p
 \right).
\end{equation}
Here the term of order $|\varepsilon|^q$ is the homogeneous contribution
of the first-order second-level displacement, whereas the self-area term
has order $|\varepsilon|^p$.  Let
\(\varepsilon_{\mathrm{pre}}>0\) be the radius fixed for the preliminary
existence construction, and let \(\varepsilon_{\mathrm{an}}>0\) be the
minimum of the radii required by the flow-jet and weighted-generator
estimates.  Choose \(\varepsilon_{\mathrm{crit}}(R,B)>0\) so that the three
perturbation terms in \eqref{eq:perturbed-local-control} contribute at most
\(\eta_*/2\), and set
\[
 \varepsilon_0(R,B)
 :=\min\{
   \varepsilon_{\mathrm{pre}},
   \varepsilon_{\mathrm{an}},
   \varepsilon_{\mathrm{crit}}(R,B)
 \}.
\]
The right-hand side of \eqref{eq:perturbed-local-control} is then below
$\eta_*$ on every interval, and all preceding local estimates hold on the
same perturbation range.
This remains true when the variation of a representative is concentrated
on one interval, because its total contribution is multiplied by a power
of $|\varepsilon|$.

It remains to pass from \eqref{eq:critical-Fy-strip-bound} to the mean
coefficient.  Set
\[
 Y_s^\lambda
 =(1-\lambda)\widetilde Y_s^i
   +\lambda\widetilde Y_s^{\varepsilon,i},
 \qquad
 Z_s^\lambda
 =(1-\lambda)\widetilde Z_s^i
   +\lambda\widetilde Z_s^{\varepsilon,i}.
\]
The exact difference equation has
\[
 \alpha_s^{\varepsilon,i}
 =\int_0^1
  F_y^{\varepsilon,I_i}(s,Y_s^\lambda,Z_s^\lambda)\,d\lambda.
\]
Since the positive part is convex and
\begin{equation}\label{eq:convex-z-square}
 |Z_s^\lambda|^2
 \leq(1-\lambda)|\widetilde Z_s^i|^2
      +\lambda|\widetilde Z_s^{\varepsilon,i}|^2,
\end{equation}
the bound $\theta_{I_i}\leq\delta_*$ gives
\begin{align}
 (\alpha_s^{\varepsilon,i})^+
 &\leq c_0+\delta_*
  \left[1+\frac12
   \bigl(|\widetilde Z_s^i|^2
        +|\widetilde Z_s^{\varepsilon,i}|^2\bigr)
  \right]                                                     \notag\\
 &\leq a_s^i+\delta_*
   \bigl(|\widetilde Z_s^i|^2
        +|\widetilde Z_s^{\varepsilon,i}|^2\bigr),
 \label{eq:critical-mean-coefficient-bound}
\end{align}
with the deterministic choice
$a_s^i\equiv c_0+\delta_*$ and
$A_*=(c_0+\delta_*)T$.  This is
\eqref{eq:common-critical-structure}.  Applying the conditional BMO energy
estimate to the two martingale integrands yields
\eqref{eq:common-critical-budget}.  No deterministic subdivision of their
realized quadratic variation is used.  The order of choices is therefore
$b_{\mathbb Q}$, then $r_A$, $\delta_*$, $\eta_*$, and finally
$\varepsilon_0$ after the deterministic partition has been fixed.

We finally remove the preliminary choices.  The backward-flow cocycle and
It\^o's formula show that the restriction of any preliminary reconstructed
solution, expressed in the reset coordinates of \(I_i\), is a bounded local
solution with the inherited original-coordinate terminal value.  On every
\(I_i\), any two
bounded transformed solutions with the same terminal interface satisfy the
common state and BMO bounds above.  Their difference has zero source and
terminal value.  For its coefficients,
\eqref{eq:mean-coefficient-bmo} gives
\(\beta\mathbin{\cdot}W\in\mathrm{BMO}_2(\mathbb P;I_i)\), and
\eqref{eq:mean-alpha-absolute-growth} together with the common BMO energy
bounds gives \(\int_{I_i}|\alpha_s|\,ds<\infty\) almost surely.
Equation~\eqref{eq:common-critical-budget} supplies the required exponential
moment of \(\alpha^+\).  Let \(\mathbb Q_i^{0,1}\) be the Girsanov measure
generated by this mean \(\beta\)-coefficient.  The zero terminal datum and
source belong to
\(L^\infty(\mathcal F_{t_{i+1}})\) and
\(\mathfrak R_{p_A}(\mathbb Q_i^{0,1};I_i)\).
Proposition~\ref{prop:critical-linear-bsde} therefore identifies the two
solutions.  Starting
from \(\xi\) on the last interval and proceeding backwards shows that every
preliminary bounded construction agrees with the unique iteration on
\(\Pi\).  If another deterministic partition satisfies the same critical
conditions, a common critical refinement and
Lemma~\ref{lem:partition-consistency} identify the two iterations.  Thus the
solution map \(\mathcal S\) is independent of both the preliminary solution
selection and the critical reset partition.
\end{proof}

We next describe the derivative on this partition.  For
$I_i=[t_i,t_{i+1}]$, write
\[
 \phi_i=\phi^{\mathbf x,I_i},\qquad
 J_i=\partial_y\phi_i,\qquad
 K_i=\partial_{yy}\phi_i,
\]
and let $(\widetilde Y^i,\widetilde Z^i)$ be the base transformed solution.
For $\tau\in\mathbb T_{\mathbf x}^p$, denote the intrinsic flow variations
by
\begin{equation}\label{eq:flow-variation-notation}
 u_i^\tau:=\dot\phi_i[\tau],
 \qquad
 j_i^\tau:=\dot J_i[\tau].
\end{equation}
The weighted-generator variation on $I_i$ is denoted by
$\dot F_i[\tau]$.

\begin{lemma}[Fixed-endpoint backward patching and reconstruction]
\label{lem:backward-patching}
On the common partition of Lemma~\ref{lem:common-reset-partition}, every
$\tau\in\mathbb T_{\mathbf x}^p$ determines a unique pair
\[
 A_{\mathbf x}\tau=(U^\tau,V^\tau)\in\mathcal X.
\]
It is obtained as follows.  Put $\eta_N^\tau=0$.  Starting with
$i=N-1$ and proceeding backwards, solve
\begin{align}
 \widetilde U_t^{\tau,i}
 ={}&\eta_{i+1}^\tau
 +\int_t^{t_{i+1}}
 \Bigl[
  \alpha_s^i\widetilde U_s^{\tau,i}
  +\beta_s^i\mathbin{\cdot}\widetilde V_s^{\tau,i}
  +\dot F_i[\tau]
 \Bigr]ds
 -\int_t^{t_{i+1}}\widetilde V_s^{\tau,i}\,dW_s,
 \label{eq:local-derivative-bsde}\\
 \alpha_s^i
 :={}&F_y^{\mathbf x,I_i}
       (s,\widetilde Y_s^i,\widetilde Z_s^i),
 \qquad
 \beta_s^i
 :=F_z^{\mathbf x,I_i}
       (s,\widetilde Y_s^i,\widetilde Z_s^i).
 \label{eq:local-derivative-coefficients}
\end{align}
On $I_i$, reconstruct
\begin{equation}\label{eq:U-reconstruction}
 U_t^\tau
 =u_i^\tau(t,\widetilde Y_t^i)
  +J_i(t,\widetilde Y_t^i)\widetilde U_t^{\tau,i}
\end{equation}
and
\begin{equation}\label{eq:V-reconstruction}
 \begin{split}
 V_t^\tau
 ={}&\Bigl[
    j_i^\tau(t,\widetilde Y_t^i)
    +K_i(t,\widetilde Y_t^i)\widetilde U_t^{\tau,i}
   \Bigr]\widetilde Z_t^i\\
 &+J_i(t,\widetilde Y_t^i)\widetilde V_t^{\tau,i}.
 \end{split}
\end{equation}
For $i\leq N-2$, the terminal value in
\eqref{eq:local-derivative-bsde} is
\begin{equation}\label{eq:terminal-derivative-handoff}
 \eta_{i+1}^\tau:=U_{t_{i+1}}^\tau,
\end{equation}
where the right-hand side has already been determined on $I_{i+1}$.

The map $A_{\mathbf x}:\mathbb T_{\mathbf x}^p\to\mathcal X$ is continuous
and linear.  It depends only on $(h,\kappa)$ and not on a representative of
that tangent.  Moreover, there is a deterministic $C_{\mathrm{patch}}$
such that
\begin{equation}\label{eq:patched-derivative-bound}
 \|A_{\mathbf x}\tau\|_{\mathcal X}
 \leq C_{\mathrm{patch}}
       \|\tau\|_{\mathbb T_{\mathbf x}^p}.
\end{equation}

For a family of local difference-quotient errors, let $e_i^\varepsilon$
be the sum of its $\mathbb S^\infty(I_i)$ and local BMO errors after
reconstruction, and let $r_i^\varepsilon$ collect the corresponding local
source and flow-jet remainders.  The constants can be chosen so that
\begin{equation}\label{eq:patching-error-recursion}
 e_i^\varepsilon
 \leq C_{\mathrm{loc}}
       (e_{i+1}^\varepsilon+r_i^\varepsilon),
 \qquad
 e_N^\varepsilon=0,
\end{equation}
with $C_{\mathrm{loc}}\geq1$ independent of the admissible perturbation.
Consequently,
\begin{equation}\label{eq:finite-patching-constant}
 \max_{0\leq i<N}e_i^\varepsilon
 \leq N_*C_{\mathrm{loc}}^{N_*}
       \max_{0\leq i<N}r_i^\varepsilon.
\end{equation}
\end{lemma}

\begin{proof}
The terminal condition is fixed, so its derivative on the last interval is
zero.

At the right endpoint of a reset interval,
\begin{equation}\label{eq:fixed-reset-endpoint-identities}
 \phi_i(t_{i+1},y)=y,\qquad
 J_i(t_{i+1},y)=1,\qquad
 K_i(t_{i+1},y)=0.
\end{equation}
Since the endpoint is fixed for every perturbed driver, its intrinsic
variation also satisfies
\begin{equation}\label{eq:fixed-reset-variation-identities}
 u_i^\tau(t_{i+1},y)=0,
 \qquad
 j_i^\tau(t_{i+1},y)=0.
\end{equation}
Equations~\eqref{eq:U-reconstruction} and
\eqref{eq:fixed-reset-variation-identities} therefore give
\[
 U_{t_{i+1}}^\tau
 =\widetilde U_{t_{i+1}}^{\tau,i}
 =\eta_{i+1}^\tau.
\]
Thus the endpoint value reconstructed from $I_{i+1}$ is exactly the
terminal value for $I_i$.  There is no derivative of a moving endpoint and
no residual flow term.  The same endpoint identities hold before taking a
difference quotient, so its terminal error on $I_i$ is precisely the
endpoint error inherited from $I_{i+1}$.

Let \(\mathbb Q_i\) be the measure generated by
\(\beta^i\mathbin{\cdot}W\).  The pointwise weighted estimate
\eqref{eq:dot-F-continuous-bound} from
Proposition~\ref{prop:weighted-full-tangent-generator}, evaluated along the
base transformed solution and combined with the conditional energy moment
estimate \eqref{eq:conditional-energy-moments} under $\mathbb Q_i$, gives
\begin{equation}\label{eq:linear-source-bound}
 \|\dot F_i[\tau]\|_{\mathfrak R_{p_A}(\mathbb Q_i;I_i)}
 \leq C_F\|\tau\|_{\mathbb T_{\mathbf x}^p},
\end{equation}
uniformly in $i$.  The base-pair specialization of
\eqref{eq:mean-coefficient-bmo} gives
\(\beta^i\mathbin{\cdot}W\in\mathrm{BMO}_2(\mathbb P;I_i)\), while
\eqref{eq:mean-alpha-absolute-growth} and the BMO energy of
\(\widetilde Z^i\) give
\(\int_{I_i}|\alpha_s^i|\,ds<\infty\) almost surely.  The base-pair case of
\eqref{eq:common-critical-budget} gives the required exponential moment of
\((\alpha^i)^+\), and \eqref{eq:linear-source-bound} verifies the source
norm.  Finally, \(\eta_N^\tau=0\) is in
\(L^\infty(\mathcal F_T)\), and backward induction gives
\(\eta_{i+1}^\tau\in L^\infty(\mathcal F_{t_{i+1}})\), since it is the
endpoint of the already constructed bounded solution on \(I_{i+1}\).
Proposition~\ref{prop:critical-linear-bsde} therefore applies on every
\(I_i\) and yields a unique solution of
\eqref{eq:local-derivative-bsde} in
\(\mathbb S^\infty(I_i)\times\mathbb H^2_{\mathrm{BMO}}(I_i)\).
The critical linear estimate and backward induction then
give a recurrence of the form
\[
 m_i\leq C_{\mathrm{loc}}
       (m_{i+1}+C_F\|\tau\|_{\mathbb T_{\mathbf x}^p}),
 \qquad m_N=0.
\]
Hence
\begin{equation}\label{eq:local-variation-geometric-bound}
 \max_i m_i
 \leq N_*C_{\mathrm{loc}}^{N_*}C_F
       \|\tau\|_{\mathbb T_{\mathbf x}^p}.
\end{equation}

The reconstruction of the martingale integrand is controlled without
changing its natural norm.  On $I_i$,
\begin{align}
 \|V^\tau\mathbin{\cdot}W\|_{\mathrm{BMO}(I_i)}
 \leq{}&
 \bigl(
  \|j_i^\tau\|_\infty
  +\|K_i\|_\infty\|\widetilde U^{\tau,i}\|_\infty
 \bigr)
 \|\widetilde Z^i\mathbin{\cdot}W\|_{\mathrm{BMO}(I_i)}
 \notag\\
 &+\|J_i\|_\infty
 \|\widetilde V^{\tau,i}\mathbin{\cdot}W\|_{\mathrm{BMO}(I_i)}.
 \label{eq:reconstructed-local-bmo}
\end{align}
The flow-variation bound and
\eqref{eq:local-variation-geometric-bound} control the right-hand side
linearly in $\|\tau\|_{\mathbb T_{\mathbf x}^p}$.  For an arbitrary
stopping time, splitting the remaining conditional energy over the
partition gives
\begin{equation}\label{eq:global-bmo-aggregation}
 \|V^\tau\mathbin{\cdot}W\|_{\mathrm{BMO}([0,T])}^2
 \leq N_*
 \max_i
 \|V^\tau\mathbin{\cdot}W\|_{\mathrm{BMO}(I_i)}^2.
\end{equation}
Together with the corresponding supremum estimate for $U^\tau$, this
proves \eqref{eq:patched-derivative-bound}.  For example, after enlarging
the local reconstruction constant $C_{\mathrm{rec}}$, one may take
\[
 C_{\mathrm{patch}}
 =C_{\mathrm{rec}}(1+\sqrt{N_*})
   N_*C_{\mathrm{loc}}^{N_*}(1+C_F).
\]

Each step in the construction is linear: the intrinsic flow equation, the
generator variation, the local BSDE, the endpoint handoff, and the two
reconstruction formulas.  The estimates above give continuity.  By
Proposition~\ref{prop:intrinsic-flow-variation}, changing the joint lift or
the cross-central decomposition of $(h,\kappa)$ does not change the
intrinsic flow variation.  It therefore leaves $\dot F_i[\tau]$, the
unique local BSDE solution, and the reconstructed pair unchanged.  This
proves representative independence without comparing radial realizations.

We make the reconstruction remainder explicit.  Suppress the tangent and
representative from the notation, let
\[
 P^{\varepsilon,i}
 :=\frac{\widetilde Y^{\varepsilon,i}-\widetilde Y^i}{\varepsilon},
 \qquad
 Q^{\varepsilon,i}
 :=\frac{\widetilde Z^{\varepsilon,i}-\widetilde Z^i}{\varepsilon},
\]
and choose a deterministic strip $|y|\leq M$ containing all transformed
states.  Define
\begin{equation}\label{eq:normalized-flow-jet-error}
 \begin{split}
 \mathfrak e_{\mathrm{jet},i}^\varepsilon
 :=\sup_{\substack{t\in I_i\\ |y|\leq M}}
 \Bigg(&
  \left|
   \frac{\phi_i^\varepsilon(t,y)-\phi_i(t,y)}{\varepsilon}
   -u_i^\tau(t,y)
  \right|\\
 &+\left|
   \frac{J_i^\varepsilon(t,y)-J_i(t,y)}{\varepsilon}
   -j_i^\tau(t,y)
  \right|
  +|K_i^\varepsilon(t,y)-K_i(t,y)|
 \Bigg).
 \end{split}
\end{equation}
The uniform full-jet expansion gives
$\mathfrak e_{\mathrm{jet},i}^\varepsilon\leq C|\varepsilon|$.
Taylor expansion on the strip and the uniform quotient bounds give
\begin{align}
 &\left\|
  \frac{Y^{\varepsilon}-Y}{\varepsilon}-U^\tau
 \right\|_{\mathbb S^\infty(I_i)}
 +\left\|
  \left(
   \frac{Z^{\varepsilon}-Z}{\varepsilon}-V^\tau
  \right)\mathbin{\cdot}W
 \right\|_{\mathrm{BMO}(I_i)}                                  \notag\\
 &\quad\leq C_{\mathrm{rec}}
 \Bigl(
  \|P^{\varepsilon,i}-\widetilde U^{\tau,i}\|_{\mathbb S^\infty(I_i)}
  +\|(Q^{\varepsilon,i}-\widetilde V^{\tau,i})
       \mathbin{\cdot}W\|_{\mathrm{BMO}(I_i)}
  +\mathfrak e_{\mathrm{jet},i}^\varepsilon
  +|\varepsilon|
 \Bigr).
 \label{eq:uniform-reconstruction-error}
\end{align}
For the $Z$-term, write
\[
 \frac{J_i^\varepsilon(\widetilde Y^{\varepsilon,i})
             \widetilde Z^{\varepsilon,i}
       -J_i(\widetilde Y^i)\widetilde Z^i}{\varepsilon}
 =\frac{J_i^\varepsilon(\widetilde Y^{\varepsilon,i})
             -J_i(\widetilde Y^i)}{\varepsilon}\widetilde Z^i
  +J_i^\varepsilon(\widetilde Y^{\varepsilon,i})
       Q^{\varepsilon,i}.
\]
The first term contains
$K_i(P^{\varepsilon,i}-\widetilde U^{\tau,i})\widetilde Z^i$
and a state Taylor remainder of order
$|\varepsilon|\,|P^{\varepsilon,i}|^2|\widetilde Z^i|$.
The second contains the coefficient difference
$[J_i^\varepsilon(\widetilde Y^{\varepsilon,i})
  -J_i(\widetilde Y^i)]Q^{\varepsilon,i}$.
The base BMO bound for $\widetilde Z^i$, the uniform BMO bound for
$Q^{\varepsilon,i}$, and the flow-jet bounds control all three terms in
the norm appearing in \eqref{eq:uniform-reconstruction-error}.

The local difference-quotient theorem, the exact endpoint handoff, and
\eqref{eq:uniform-reconstruction-error} now give
\eqref{eq:patching-error-recursion}.  Iteration yields
\[
 e_i^\varepsilon
 \leq\sum_{j=i}^{N-1}
 C_{\mathrm{loc}}^{j-i+1}r_j^\varepsilon,
\]
which implies \eqref{eq:finite-patching-constant} because $N\leq N_*$.
\end{proof}

\section{Intrinsic tangential Hadamard differentiability}
\label{sec:main-proof}

The preceding construction gives the candidate derivative.  The next
proposition records the uniform expansion needed to identify it from
rough-path perturbations.

\begin{proposition}[Uniform radial expansion of the solution map]
\label{prop:radial-expansion}
For every $R,B<\infty$,
\begin{equation}\label{eq:uniform-radial-balls}
 \lim_{\varepsilon\to0}
 \sup_{\substack{
       \tau\in\mathbb T_{\mathbf x}^p,
       \,\|\tau\|_{\mathbb T_{\mathbf x}^p}\leq R\\
       (\mathbf Z,a)\in\operatorname{Rep}_{\mathbf x}(\tau),
       \,\mathfrak b(\mathbf Z,a)\leq B}}
 \left\|
  \frac{
   \mathcal S(\mathbf x^{\varepsilon;\mathbf Z,a})
   -\mathcal S(\mathbf x)}{\varepsilon}
  -A_{\mathbf x}\tau
 \right\|_{\mathcal X}
 =0.
\end{equation}
\end{proposition}

\begin{proof}
The case $R=0$ is immediate.  For $R>0$, the unit-ball estimates extend to
the fixed ball of radius $R$ by the ray rescaling
$\tau=R\bar\tau$, $\mathbf Z=D_{1,R*}\bar{\mathbf Z}$, and
$a=R\bar a$, with the perturbation parameter replaced by
$R\varepsilon$.  This identity also preserves the self-area term in the
radial realization.  The constants may depend on $R$ and $B$.

Use the partition of Lemma~\ref{lem:common-reset-partition}.  On each
interval, Propositions~\ref{prop:uniform-full-jet-expansion} and
\ref{prop:weighted-full-tangent-generator} identify the parameter
derivative in \eqref{eq:local-derivative-bsde} and make the weighted source
remainder converge to zero uniformly on the sets in
\eqref{eq:uniform-radial-balls}.  On the last interval, the terminal
interface in Theorem~\ref{thm:local-full-tangent-differentiability} is
identically zero.  Apply that theorem, reconstruct at $t_{N-1}$, and use
the fixed reset identities to obtain the terminal interface on $I_{N-2}$.
Repeating this step backwards produces local residuals satisfying
\begin{equation}\label{eq:uniform-local-residuals}
 \max_{0\leq i<N}r_i^\varepsilon\longrightarrow0.
\end{equation}
The fixed endpoint identities ensure that the terminal remainder on one
interval is exactly the reconstructed endpoint remainder from the next
one.  Lemma~\ref{lem:backward-patching}, in particular
\eqref{eq:finite-patching-constant}, propagates
\eqref{eq:uniform-local-residuals} to the whole interval.  The full-jet
remainder in the inverse Doss--Sussmann reconstruction is uniform under
the same budget.  This proves \eqref{eq:uniform-radial-balls} in
$\mathbb S^\infty\times\mathbb H^2_{\mathrm{BMO}}$.
\end{proof}

\begin{proof}[Proof of Theorem~\ref{thm:intrinsic-hadamard}]
Lemma~\ref{lem:backward-patching} defines a continuous linear map
$A_{\mathbf x}:\mathbb T_{\mathbf x}^p\to\mathcal X$.  If a bounded
representative assignment is fixed on the unit tangent ball, choose $R=1$
and let $B$ be its budget bound in
Proposition~\ref{prop:radial-expansion}.  Equation
\eqref{eq:uniform-radial-balls} is then precisely
\eqref{eq:uniform-radial-expansion}.

The same estimate already allows the tangent to change.  Indeed, if
$\tau_n\to\tau$, $\varepsilon_n\to0$, and
$(\mathbf Z_n,a_n)$ represent $\tau_n$ with a common budget, then
\begin{align}
 &\left\|
  \frac{
   \mathcal S(\mathbf x^{\varepsilon_n;\mathbf Z_n,a_n})
       -\mathcal S(\mathbf x)}{\varepsilon_n}
  -A_{\mathbf x}\tau
 \right\|_{\mathcal X}                                      \notag\\
 &\quad\leq
 \left\|
  \frac{
   \mathcal S(\mathbf x^{\varepsilon_n;\mathbf Z_n,a_n})
       -\mathcal S(\mathbf x)}{\varepsilon_n}
  -A_{\mathbf x}\tau_n
 \right\|_{\mathcal X}
 +\|A_{\mathbf x}(\tau_n-\tau)\|_{\mathcal X}
 \longrightarrow0.
 \label{eq:changing-radial-tangents}
\end{align}

Now let $\mathbf x_n$ be an admissible secant with strong levelwise
variation contact.  Lemma~\ref{lem:secant-recoding} supplies tangents
$\tau_n^*\to\tau$ and representatives
$(\mathbf Z_n^*,a_n^*)$ with a common budget such that
\begin{equation}\label{eq:exact-secant-radial-recoding}
 \mathbf x_n
 =\mathbf x^{\varepsilon_n;\mathbf Z_n^*,a_n^*}
\end{equation}
holds exactly.  Applying \eqref{eq:changing-radial-tangents} to
\eqref{eq:exact-secant-radial-recoding} proves
\eqref{eq:hadamard-secant-limit}.

Representative independence was established in
Lemma~\ref{lem:backward-patching} from uniqueness of the intrinsic tangent
sewing problem and of the local linear BSDEs.  Thus the value of
$A_{\mathbf x}(h,\kappa)$ is unaffected by the joint lift, by a
cross-central decomposition of $\kappa$, or by the self-area appearing in
a radial realization.

It remains to check uniqueness.  Let $\bar A_{\mathbf x}$ be another
continuous linear map with the radial expansion in the theorem.  For a
fixed $\tau$, choose any finite-budget representative and follow its radial
realization.  The difference quotient converges both to
$A_{\mathbf x}\tau$ and to $\bar A_{\mathbf x}\tau$.  Hence the two maps
agree on every tangent.
\end{proof}

\section{Selected chart pullbacks}
\label{sec:chart-corollary}

\begin{proof}[Proof of Corollary~\ref{cor:chart-pullback}]
Let $u\neq0$, put $r=\|u\|_B$, and set $v=u/r$.  The ray-homogeneity of
the selected representations gives the exact identity
\begin{equation}\label{eq:chart-as-radial-realization}
 \chi(u)
 =\mathbf x^{r;\mathbf Z_{Jv},a_{Jv}}.
\end{equation}
The set $J\{v:\|v\|_B=1\}$ is bounded in
$\mathbb T_{\mathbf x}^p$, and the corresponding representative budgets
are bounded by assumption.  Proposition~\ref{prop:radial-expansion},
applied uniformly over this set, yields
\[
 \frac{
  \|\mathcal S(\chi(u))-\mathcal S(\mathbf x)
       -A_{\mathbf x}Ju\|_{\mathcal X}}
      {\|u\|_B}
 \longrightarrow0
 \qquad\text{as }u\to0.
\]
This is Fr\'echet differentiability at zero with derivative
$A_{\mathbf x}\circ J$.
\end{proof}

The abstract pullback statement contains useful canonical slices.  Let
\begin{equation}\label{eq:young-conjugate-threshold}
 p^*:=\frac{p}{p-1},
 \qquad q\leq r<p^*.
\end{equation}
For $h\in\mathcal V_0^r(E)$, define the Young cross term
\begin{equation}\label{eq:young-cross-tangent}
 \mathcal C^{\mathbf x,h}_{s,t}
 :=\int_s^t x_{s,u}\otimes dh_u
   +\int_s^t h_{s,u}\otimes dx_u
\end{equation}
and the Young self-area
\begin{equation}\label{eq:young-self-area}
 \mathbb h_{s,t}:=\int_s^t h_{s,u}\otimes dh_u.
\end{equation}
Both integrals are well defined because $1/p+1/r>1$ and $r<2$.

\begin{corollary}[Fixed Young--central regularity slices]
\label{cor:fixed-slices}
Fix $r$ as in \eqref{eq:young-conjugate-threshold} and set
\[
 \mathbb I_{r,p}
 :=\mathcal V_0^r(E)
   \oplus\mathcal V_0^q(\mathfrak{so}(E)).
\]
Equip this space with the Banach norm
\begin{equation}\label{eq:young-central-slice-norm}
 \|(h,a)\|_{\mathbb I_{r,p}}
 :=\|h\|_{r\text{-var}}+\|a\|_{q\text{-var}}.
\end{equation}
For $a\in\mathcal V_0^q(\mathfrak{so}(E))$, write
$a_{s,t}=a_t-a_s$ and define
\begin{equation}\label{eq:young-central-chart}
 \chi_{\mathbf x}^r(h,a)
 :=\left(
  1,
  x+h,
  \mathbb x+\mathcal C^{\mathbf x,h}+a+\mathbb h
 \right).
\end{equation}
Then $\mathcal S\circ\chi_{\mathbf x}^r$ is Fr\'echet differentiable at
$(0,0)$ as a map from $\mathbb I_{r,p}$ to $\mathcal X$, and
\begin{equation}\label{eq:fixed-slice-derivative}
 D(\mathcal S\circ\chi_{\mathbf x}^r)(0,0)[h,a]
 =A_{\mathbf x}\bigl(h,\mathcal C^{\mathbf x,h}+a\bigr).
\end{equation}
\end{corollary}

\begin{proof}
Young's inequality gives
\begin{equation}\label{eq:young-cross-bound}
 \|\mathcal C^{\mathbf x,h}\|_{q\text{-var};2}
 \leq C_{p,r}\|x\|_{p\text{-var}}\|h\|_{r\text{-var}}.
\end{equation}
The map
\[
 J_{\mathbf x}^r(h,a)
 :=(h,\mathcal C^{\mathbf x,h}+a)
\]
is therefore continuous and linear from $\mathbb I_{r,p}$ into
$\mathbb T_{\mathbf x}^p$.  The canonical Young joint lift of $(x,h)$,
together with the central residual $a$, is ray-homogeneous and has bounded
realization budget on bounded subsets of $\mathbb I_{r,p}$.  Moreover,
\eqref{eq:young-central-chart} is its selected radial chart.  The result
follows from Corollary~\ref{cor:chart-pullback}.
\end{proof}

The constants in Corollary~\ref{cor:fixed-slices} may depend on $r$.  The
statement does not provide a uniform chart estimate over an unnormed union
of spaces as $r\uparrow p^*$.

\begin{corollary}[Central-area directions]
\label{cor:central-area}
For $a\in\mathcal V_0^q(\mathfrak{so}(E))$, define
\begin{equation}\label{eq:central-area-chart}
 \chi_{\mathbf x}^{\mathrm{area}}(a)
 :=(1,x,\mathbb x+a).
\end{equation}
The pullback
$\mathcal S\circ\chi_{\mathbf x}^{\mathrm{area}}$ is Fr\'echet
differentiable at zero, and
\begin{equation}\label{eq:central-area-derivative}
 D(\mathcal S\circ\chi_{\mathbf x}^{\mathrm{area}})(0)[a]
 =A_{\mathbf x}(0,a).
\end{equation}
\end{corollary}

\begin{proof}
Take $h=0$ in Corollary~\ref{cor:fixed-slices}.  Then
$\mathcal C^{\mathbf x,0}=0$ and the Young self-area vanishes.
\end{proof}

\section{Structural examples and boundary cases}
\label{sec:examples}

\begin{example}[A nonzero pure L\'evy-area response]
\label{ex:pure-area-response}
Let $d=2$, let the base rough path be zero, and set
\begin{equation}\label{eq:pure-area-driver}
 (\mathbf X^\varepsilon)^1_{s,t}=0,
 \qquad
 (\mathbf X^\varepsilon)^2_{s,t}
 =\varepsilon\frac{t-s}{T}
  (e_1\otimes e_2-e_2\otimes e_1).
\end{equation}
Take
\begin{equation}\label{eq:pure-area-coefficients}
 H_1(y)=\cos y,
 \qquad H_2(y)=\sin y,
 \qquad f=0,
 \qquad \xi=c\in\mathbb R.
\end{equation}
With the convention
\[
 [H_1,H_2]
 :=DH_2\,H_1-DH_1\,H_2,
\]
one has $[H_1,H_2]=1$.  The backward flow and the rough BSDE are therefore
solved exactly by
\begin{equation}\label{eq:pure-area-exact-solution}
 Y_t^\varepsilon
 =c-\varepsilon\left(1-\frac{t}{T}\right),
 \qquad
 Z_t^\varepsilon=0.
\end{equation}
If
\[
 a_t=\frac{t}{T}
 (e_1\otimes e_2-e_2\otimes e_1),
\]
then
\begin{equation}\label{eq:pure-area-nonzero-derivative}
 A_{\mathbf 0}(0,a)
 =\left(-\left(1-\frac{\cdot}{T}\right),0\right).
\end{equation}
The first-order remainder is identically zero.  Thus the central component
of the tangent fibre can carry a nontrivial response even when the
first-level tangent vanishes.
\end{example}

\begin{remark}[Commuting vector fields]
Suppose $[H_i,H_j]=0$ for every $i,j$.  A central second-level perturbation
then contributes no source to the intrinsic reset-flow variation.  Under
the present assumptions that the terminal condition is fixed and that the
generator has no separate dependence on an absolute rough state, one has
\begin{equation}\label{eq:commuting-central-degeneration}
 A_{\mathbf x}(0,a)=0
 \qquad
 \text{for every }a\in\mathcal V_0^q(\mathfrak{so}(E)).
\end{equation}
The nonzero response in Example~\ref{ex:pure-area-response} is therefore a
bracket effect, not a consequence of the first level.
\end{remark}

\begin{remark}[The homogeneous rough-metric scale]
For a nonzero central path $a$, the chart
\eqref{eq:central-area-chart} satisfies, for any of the equivalent standard
homogeneous $p$-variation rough-path metrics and up to fixed
norm-equivalence constants,
\begin{equation}\label{eq:central-homogeneous-scale}
 d_{p\text{-var}}
 \bigl(\chi_{\mathbf x}^{\mathrm{area}}(\varepsilon a),\mathbf x\bigr)
 \asymp |\varepsilon|^{1/2}
          \|a\|_{q\text{-var}}^{1/2}.
\end{equation}
Example~\ref{ex:pure-area-response} has a response of order
$|\varepsilon|$, which is quadratic at the scale in
\eqref{eq:central-homogeneous-scale}.  Under the Carnot dilation, a central
direction is parameterized as $\varepsilon^2a$, and the quotient by
$\varepsilon$ in this example converges to zero.  The derivative in
Corollary~\ref{cor:central-area} is thus an intrinsic tensor-level area
derivative.  It should not be read as a first-order Fr\'echet derivative in
the homogeneous rough-path metric.  This observation does not rule out
other metric notions of differentiability.
\end{remark}

\section{Limitations and outlook}
\label{sec:outlook}

The theorem concerns a scalar BSDE with a deterministic step-two weakly
geometric driver, $2<p<3$, a fixed bounded terminal condition,
$H_i\in C_b^9$, and a globally Lipschitz original generator.  The quadratic
growth treated in the proof is generated by the Doss--Sussmann
transformation.  These restrictions separate two kinds of possible
extension.

Some refinements are technically natural within the present architecture.
The assumption $H_i\in C_b^9$ was chosen to keep the flow, inverse-flow,
and rational generator expansions under a single uniform derivative
budget.  A more economical allocation of derivatives among the direct
flow, its inverse, and the transformed generator may lower this
regularity.  Likewise, one could formulate higher-order deterministic jet
expansions, provided the corresponding derivatives of $H$ and $f$ are
available.  Neither refinement changes the tangent fibre or the backward
patching mechanism.  Smooth terminal conditions depending on a
finite-dimensional state may also fit this scheme after adding their
terminal variation to the last local linear BSDE.  These are plausible
extensions of the estimates, but they are not consequences of the theorem
as stated, and no reduced regularity threshold or higher-order expansion is
claimed here.

Other extensions require new theory rather than longer versions of the
same estimates.  A random rough driver would make the reset flows and the
partition random and would raise measurability, adaptedness, and
conditional rough-path questions.  If the rough driver is correlated with
$W$, one must also decide whether a joint stochastic lift is part of the
model and how it enters the solution concept.  The deterministic
representative-independence argument does not settle these issues.
Vector-valued quadratic BSDEs present a different obstruction: the scalar
comparison, exponential estimates, and one-dimensional Doss--Sussmann
order structure used here have no direct multidimensional substitute.
Allowing $p\geq3$ would require higher tensor levels, higher linearized
Chen relations, and a tangent-sewing argument that retains all bracket
directions at the relevant step.  A genuinely quadratic original generator
would also alter the BMO and reverse-H\"older thresholds before the rough
transformation is applied.  Each of these changes affects an essential
part of the proof.

The present result also stops short of a global original-coordinate linear
rough BSDE driven by an unspecified joint lift.  The derivative is instead
characterized by the intrinsic reset-flow tangent equation, local
transformed linear BSDEs, fixed-endpoint handoff, and inverse
Doss--Sussmann reconstruction.  This local characterization is sufficient
for uniqueness and representative independence.  A global rough linear
equation would require a separate solution concept and an explicit account
of the cross iterated integrals; it is not needed for the theorem proved
here.

Finite signature groups provide natural selected charts, but the analytic
result alone does not produce a statistical or financial application.  A
Lie derivative on a truncated signature group additionally requires a
differentiable map from the group state to the full future rough scenario.
A delta method requires a limit theorem for an estimator in a specified
logarithmic chart, at its actual normalization rate, followed by a verified
chain rule.  Bootstrap validity depends on whether the resulting derivative
is continuous linear or merely directional.  For dynamic risk
interpretations, the observed signature state must be adapted and
sufficient for the conditional law of future increments; time augmentation
does not by itself make a finite truncation sufficient.  Without those
assumptions, $Y_t^{\mathbf x}$ is a deterministic-scenario, or quenched,
response rather than an online risk functional based on an observed future
path.  Establishing these links is a separate program, not an implication
of intrinsic rough-path differentiability.

\begin{appendix}

\section{BMO, reverse-H\"older, and critical linear estimates}
\label{app:bmo-reverse-holder}

This appendix proves the probability estimates used in
Section~\ref{sec:local-transformed-bsde}.  The constants are tracked only to
the extent needed for uniformity over the rough perturbation family.

\subsection{Conditional energy moments}
\label{subsec:conditional-energy}

Let \(N\) be a continuous martingale on \(I=[u,v]\) with
\(\|N\|_{\mathrm{BMO}_2(\mathbb Q;I)}\leq b\).  For every integer
\(k\geq1\), the conditional energy inequality gives
\begin{equation}
 \sup_{\tau\in\mathcal T_I}
 \left\|
  \mathbb E_\tau^{\mathbb Q}
  \left[
   \bigl(\langle N\rangle_v-\langle N\rangle_\tau\bigr)^k
  \right]
 \right\|_\infty
 \leq k!\,b^{2k}.
 \label{eq:conditional-energy-moments}
\end{equation}
Indeed, the identity
\[
 \bigl(\langle N\rangle_v-\langle N\rangle_\tau\bigr)^k
 =k\int_\tau^v
   \bigl(\langle N\rangle_v-\langle N\rangle_s\bigr)^{k-1}
   \,d\langle N\rangle_s
\]
and conditional Fubini reduce the order-\(k\) estimate to the order-
\(k-1\) estimate and the BMO bound.  Induction yields
\eqref{eq:conditional-energy-moments}.  Consequently, for
\(0\leq c<b^{-2}\),
\begin{equation}
 \sup_{\tau\in\mathcal T_I}
 \left\|
  \mathbb E_\tau^{\mathbb Q}
  \exp\left(
   c(\langle N\rangle_v-\langle N\rangle_\tau)
  \right)
 \right\|_\infty
 \leq\frac{1}{1-cb^2}.
 \label{eq:conditional-energy-exponential}
\end{equation}

To derive \eqref{eq:critical-alpha-explicit-budget}, apply
\eqref{eq:conditional-energy-exponential} to
\(Z^0\mathbin{\cdot}W^{\mathbb Q}\) and
\(Z^1\mathbin{\cdot}W^{\mathbb Q}\) with \(c=2r_A\delta\), then use
conditional Cauchy--Schwarz.  This gives
\[
 \mathbb E_\tau^{\mathbb Q}
 \exp\left(
  r_A\delta\int_\tau^v
  (\lvert Z_s^0\rvert^2+\lvert Z_s^1\rvert^2)\,ds
 \right)
 \leq\frac{1}{1-2r_A\delta b_{\mathbb Q}^2}.
\]
Multiplication by \(e^{r_AA_I}\) proves
\eqref{eq:critical-alpha-explicit-budget}.

\subsection{Proof of the uniform BMO--Girsanov package}
\label{subsec:proof-bmo-girsanov}

Fix \(\theta\) and abbreviate \(N=N^\theta\),
\[
 L_t=\exp\left(
  N_t-N_u-\frac12(\langle N\rangle_t-\langle N\rangle_u)
 \right),\qquad u\leq t\leq v,
\]
and \(\mathbb Q=\mathbb Q^\theta\).  Thus \(L_u=1\) and
\(d\mathbb Q/d\mathbb P=L_v\) on \(\mathcal F_v\).
By \cite[Lemma~1.2(1)]{AnkirchnerImkellerDosReis2007}, \(L\) is a true
density.  Proposition~1 of \cite{Kazamaki1983} gives an exponent
\(r_+>1\) and a constant \(C_+\) for
\eqref{eq:forward-reverse-holder}.  The norm-dependent construction in
that proposition allows \(r_+\) and \(C_+\) to be selected from the common
budget \(n_*\), hence uniformly in \(\theta\).

Under \(\mathbb Q\), set
\[
 \widehat N=N-\langle N\rangle.
\]
This is a \(\mathbb Q\)-local martingale by Girsanov's theorem, and
\cite[Lemma~1.2(2)]{AnkirchnerImkellerDosReis2007} gives
\(\widehat N\in\mathrm{BMO}_2(\mathbb Q;I)\).  The required uniform norm
bound also follows directly from the forward reverse-H\"older estimate.
Let \(r_+'=r_+/(r_+-1)\) and choose an integer \(k_+\geq r_+'\).  Bayes'
formula, conditional H\"older, and
\eqref{eq:conditional-energy-moments} give
\begin{align}
 &\mathbb E_\tau^{\mathbb Q}
  [\langle\widehat N\rangle_v-
    \langle\widehat N\rangle_\tau] \notag\\
 &\quad=
 \mathbb E_\tau^{\mathbb P}
 \left[
  \frac{L_v}{L_\tau}
  (\langle N\rangle_v-\langle N\rangle_\tau)
 \right] \notag\\
 &\quad\leq
 C_+^{1/r_+}(k_+!)^{1/k_+}n_*^2.
 \label{eq:self-girsanov-uniform-bmo}
\end{align}
For \(u\leq\tau\leq t\leq v\), the inverse density ratio satisfies
\begin{equation}
 \frac{L_\tau}{L_t}
 =\exp\left(
  -(\widehat N_t-\widehat N_\tau)
  -\frac12(\langle\widehat N\rangle_t-
             \langle\widehat N\rangle_\tau)
 \right).
 \label{eq:inverse-density-exponential}
\end{equation}
Applying \cite[Proposition~1]{Kazamaki1983} under \(\mathbb Q\), with the
uniform bound in \eqref{eq:self-girsanov-uniform-bmo}, gives
\(r_->1\), \(C_-<\infty\), and
\eqref{eq:inverse-reverse-holder}.

For a continuous \(\mathbb P\)-local martingale \(X\), its
\(\mathbb Q\)-local martingale transform is
\(\widetilde X=X-\langle X,N\rangle\), and
\(\langle\widetilde X\rangle=\langle X\rangle\).  Repeating the preceding
Bayes and energy calculation gives
\begin{equation}
 \|\widetilde X\|_{
       \mathrm{BMO}_2(\mathbb Q;I)}^2
 \leq C_+^{1/r_+}(k_+!)^{1/k_+}
       \|X\|_{\mathrm{BMO}_2(\mathbb P;I)}^2.
 \label{eq:forward-bmo-equivalence}
\end{equation}
Apply the same argument under \(\mathbb Q\) to the inverse density and
choose an integer \(k_-\geq r_-/(r_--1)\).  Since the inverse Girsanov
transform of \(\widetilde X\) is \(X\),
\begin{equation}
 \|X\|_{\mathrm{BMO}_2(\mathbb P;I)}^2
 \leq C_-^{1/r_-}(k_-!)^{1/k_-}
       \|\widetilde X\|_{
       \mathrm{BMO}_2(\mathbb Q;I)}^2.
 \label{eq:inverse-bmo-equivalence}
\end{equation}
Equations \eqref{eq:forward-bmo-equivalence} and
\eqref{eq:inverse-bmo-equivalence} prove
\eqref{eq:bmo-norm-equivalence}.  They are consistent with the isomorphism
and quantitative norm estimates in
\cite[Theorems~1--2]{ChikvinidzeMania2014}.

This proof selects \(r_+\) from the \(\mathbb P\)-BMO budget and \(r_-\)
from the resulting \(\mathbb Q\)-BMO budget.  If a later estimate needs one
exponent, any number strictly below \(\min\{r_+,r_-\}\) can be used.  Without
an additional smallness assumption, neither exponent can be set equal to
two in advance.

\subsection{Proof of the critical local linear estimate}
\label{subsec:proof-critical-linear}

Let
\[
 W_t^{\mathbb Q}
 =W_t-\int_u^t\beta_s\,ds.
\]
Under \(\mathbb Q\), equation \eqref{eq:critical-linear-bsde} becomes
\begin{equation}
 R_t
 =\eta+\int_t^v(\alpha_sR_s+r_s)\,ds
 -\int_t^vS_s\,dW_s^{\mathbb Q}.
 \label{eq:critical-linear-under-q}
\end{equation}
For \(t\leq s\), set
\[
 \Lambda_{t,s}
 =\exp\left(\int_t^s\alpha_\ell\,d\ell\right).
\]
The variation-of-constants formula gives
\begin{equation}
 R_t
 =\mathbb E_t^{\mathbb Q}
 \left[
  \Lambda_{t,v}\eta
  +\int_t^v\Lambda_{t,s}r_s\,ds
 \right].
 \label{eq:critical-linear-representation}
\end{equation}
The conditional expectation is well defined because
\[
 \sup_{t\leq s\leq v}\Lambda_{t,s}
 \leq\exp\left(\int_t^v\alpha_s^+\,ds\right),
\]
and conditional H\"older with exponents \(r_A,p_A\) yields
\eqref{eq:critical-linear-sup-bound}.

Conversely, define \(R\) by
\eqref{eq:critical-linear-representation}.  The martingale representation
property under \(\mathbb Q\), applied after localization when necessary,
provides a predictable integrand \(S\) for which
\eqref{eq:critical-linear-under-q} holds.  The conditional energy estimate
below shows that this local representation belongs to BMO, and therefore
removes the localization.

Apply It\^o's formula to \(\lvert R\rvert^2\) between
\(\tau\in\mathcal T_I\) and \(v\).  Since \(R\) is bounded,
\begin{align*}
 \mathbb E_\tau^{\mathbb Q}\int_\tau^v\lvert S_s\rvert^2\,ds
 &\leq \|\eta\|_\infty^2
 +2\|R\|_\infty^2
   \mathbb E_\tau^{\mathbb Q}\int_\tau^v\alpha_s^+\,ds\\
 &\quad
 +2\|R\|_\infty
   \mathbb E_\tau^{\mathbb Q}\int_\tau^v\lvert r_s\rvert\,ds.
\end{align*}
Conditional Jensen applied to
\eqref{eq:critical-exponential-budget} gives
\[
 \mathbb E_\tau^{\mathbb Q}\int_\tau^v\alpha_s^+\,ds
 \leq\frac{1}{r_A}\log K_A,
\]
while \eqref{eq:conditional-source-norm} controls the last term.  This proves
\eqref{eq:critical-linear-bmo-bound}.  The BMO equivalence in
Proposition~\ref{prop:uniform-bmo-girsanov} transfers the estimate from
\(S\mathbin{\cdot}W^{\mathbb Q}\) to
\(S\mathbin{\cdot}W\).

If two solutions satisfy the stated conditions, their difference has
\(\eta=0\) and \(r=0\).  Equation
\eqref{eq:critical-linear-sup-bound} gives \(R=0\), and
\eqref{eq:critical-linear-bmo-bound} then gives \(S=0\).  This proves
uniqueness and completes the proof of
Proposition~\ref{prop:critical-linear-bsde}.
\end{appendix}

\bibliographystyle{imsart-number}
\bibliography{rough_bsde_intrinsic_hadamard}

\begin{thebibliography}{21}

\bibitem{AnkirchnerImkellerDosReis2007}
\begin{barticle}[author]
\bauthor{\bsnm{Ankirchner},~\bfnm{Stefan}\binits{S.}},
  \bauthor{\bsnm{Imkeller},~\bfnm{Peter}\binits{P.}} \AND
  \bauthor{\bparticle{dos} \bsnm{Reis},~\bfnm{Gon{\c c}alo}\binits{G.}}
(\byear{2007}).
\btitle{Classical and variational differentiability of {BSDEs} with quadratic
  growth}.
\bjournal{Electronic Journal of Probability}
\bvolume{12}
\bpages{1418--1453}.
\bdoi{10.1214/EJP.v12-462}
\end{barticle}
\endbibitem

\bibitem{Bailleul2015}
\begin{barticle}[author]
\bauthor{\bsnm{Bailleul},~\bfnm{Isma{\"e}l}\binits{I.}}
(\byear{2015}).
\btitle{Regularity of the {It\^o--Lyons} map}.
\bjournal{Confluentes Mathematici}
\bvolume{7}
\bpages{3--11}.
\bdoi{10.5802/cml.15}
\end{barticle}
\endbibitem

\bibitem{BechererSun2025}
\begin{bmisc}[author]
\bauthor{\bsnm{Becherer},~\bfnm{Dirk}\binits{D.}} \AND
  \bauthor{\bsnm{Sun},~\bfnm{Yuchen}\binits{Y.}}
(\byear{2025}).
\btitle{Rough backward {SDEs} with discontinuous {Young} drivers}.
\bnote{arXiv:2505.20437}.
\bdoi{10.48550/arXiv.2505.20437}
\end{bmisc}
\endbibitem

\bibitem{BriandConfortola2008}
\begin{barticle}[author]
\bauthor{\bsnm{Briand},~\bfnm{Philippe}\binits{P.}} \AND
  \bauthor{\bsnm{Confortola},~\bfnm{Fulvia}\binits{F.}}
(\byear{2008}).
\btitle{{BSDEs} with stochastic Lipschitz condition and quadratic {PDEs} in
  {Hilbert} spaces}.
\bjournal{Stochastic Processes and their Applications}
\bvolume{118}
\bpages{818--838}.
\bdoi{10.1016/j.spa.2007.06.006}
\end{barticle}
\endbibitem

\bibitem{ChikvinidzeMania2014}
\begin{barticle}[author]
\bauthor{\bsnm{Chikvinidze},~\bfnm{B.}\binits{B.}} \AND
  \bauthor{\bsnm{Mania},~\bfnm{M.}\binits{M.}}
(\byear{2014}).
\btitle{New proofs of some results on bounded mean oscillation martingales
  using backward stochastic differential equations}.
\bjournal{Journal of Theoretical Probability}
\bvolume{27}
\bpages{1213--1228}.
\bdoi{10.1007/s10959-013-0524-x}
\end{barticle}
\endbibitem

\bibitem{CoutinLejay2018}
\begin{barticle}[author]
\bauthor{\bsnm{Coutin},~\bfnm{Laure}\binits{L.}} \AND
  \bauthor{\bsnm{Lejay},~\bfnm{Antoine}\binits{A.}}
(\byear{2018}).
\btitle{Sensitivity of rough differential equations: An approach through the
  {Omega} lemma}.
\bjournal{Journal of Differential Equations}
\bvolume{264}
\bpages{3899--3917}.
\bdoi{10.1016/j.jde.2017.11.031}
\end{barticle}
\endbibitem

\bibitem{DiehlFriz2012}
\begin{barticle}[author]
\bauthor{\bsnm{Diehl},~\bfnm{Joscha}\binits{J.}} \AND
  \bauthor{\bsnm{Friz},~\bfnm{Peter~K.}\binits{P.~K.}}
(\byear{2012}).
\btitle{Backward stochastic differential equations with rough drivers}.
\bjournal{The Annals of Probability}
\bvolume{40}
\bpages{1715--1758}.
\bdoi{10.1214/11-AOP660}
\end{barticle}
\endbibitem

\bibitem{DiehlZhang2017}
\begin{barticle}[author]
\bauthor{\bsnm{Diehl},~\bfnm{Joscha}\binits{J.}} \AND
  \bauthor{\bsnm{Zhang},~\bfnm{Jianfeng}\binits{J.}}
(\byear{2017}).
\btitle{Backward stochastic differential equations with {Young} drift}.
\bjournal{Probability, Uncertainty and Quantitative Risk}
\bvolume{2}
\bpages{5}.
\bdoi{10.1186/s41546-017-0016-5}
\end{barticle}
\endbibitem

\bibitem{EddahbiSene2014}
\begin{bmisc}[author]
\bauthor{\bsnm{Eddahbi},~\bfnm{M'hamed}\binits{M.}} \AND
  \bauthor{\bsnm{S{\`e}ne},~\bfnm{Abou}\binits{A.}}
(\byear{2014}).
\btitle{Quadratic {BSDEs} with rough drivers and {$L^2$}--terminal condition}.
\bnote{arXiv:1403.2998}.
\bdoi{10.48550/arXiv.1403.2998}
\end{bmisc}
\endbibitem

\bibitem{FrizOberhauser2009}
\begin{barticle}[author]
\bauthor{\bsnm{Friz},~\bfnm{Peter}\binits{P.}} \AND
  \bauthor{\bsnm{Oberhauser},~\bfnm{Harald}\binits{H.}}
(\byear{2009}).
\btitle{Rough path limits of the {Wong--Zakai} type with a modified drift
  term}.
\bjournal{Journal of Functional Analysis}
\bvolume{256}
\bpages{3236--3256}.
\bdoi{10.1016/j.jfa.2009.02.010}
\end{barticle}
\endbibitem

\bibitem{FrizVictoir2010}
\begin{bbook}[author]
\bauthor{\bsnm{Friz},~\bfnm{Peter~K.}\binits{P.~K.}} \AND
  \bauthor{\bsnm{Victoir},~\bfnm{Nicolas~B.}\binits{N.~B.}}
(\byear{2010}).
\btitle{Multidimensional Stochastic Processes as Rough Paths: Theory and
  Applications}.
\bseries{Cambridge Studies in Advanced Mathematics}
\bvolume{120}.
\bpublisher{Cambridge University Press}, \baddress{Cambridge}.
\bdoi{10.1017/CBO9780511845079}
\end{bbook}
\endbibitem

\bibitem{GellerLyons2026}
\begin{bmisc}[author]
\bauthor{\bsnm{Geller},~\bfnm{Martin}\binits{M.}} \AND
  \bauthor{\bsnm{Lyons},~\bfnm{Terry}\binits{T.}}
(\byear{2026}).
\btitle{The Geometry of Rough Path Space}.
\bnote{arXiv:2601.15402}.
\bdoi{10.48550/arXiv.2601.15402}
\end{bmisc}
\endbibitem

\bibitem{Gubinelli2004}
\begin{barticle}[author]
\bauthor{\bsnm{Gubinelli},~\bfnm{Massimiliano}\binits{M.}}
(\byear{2004}).
\btitle{Controlling rough paths}.
\bjournal{Journal of Functional Analysis}
\bvolume{216}
\bpages{86--140}.
\bdoi{10.1016/j.jfa.2004.01.002}
\end{barticle}
\endbibitem

\bibitem{Kazamaki1983}
\begin{barticle}[author]
\bauthor{\bsnm{Kazamaki},~\bfnm{Norihiko}\binits{N.}}
(\byear{1983}).
\btitle{On the reverse {H}{\"o}lder inequalities for certain exponential
  processes}.
\bjournal{Tohoku Mathematical Journal}
\bvolume{35}
\bpages{309--311}.
\bdoi{10.2748/tmj/1178229057}
\end{barticle}
\endbibitem

\bibitem{Kobylanski2000}
\begin{barticle}[author]
\bauthor{\bsnm{Kobylanski},~\bfnm{Magdalena}\binits{M.}}
(\byear{2000}).
\btitle{Backward stochastic differential equations and partial differential
  equations with quadratic growth}.
\bjournal{The Annals of Probability}
\bvolume{28}
\bpages{558--602}.
\bdoi{10.1214/aop/1019160253}
\end{barticle}
\endbibitem

\bibitem{LiZhangZhang2026}
\begin{barticle}[author]
\bauthor{\bsnm{Li},~\bfnm{Hanwu}\binits{H.}},
  \bauthor{\bsnm{Zhang},~\bfnm{Huilin}\binits{H.}} \AND
  \bauthor{\bsnm{Zhang},~\bfnm{Kuan}\binits{K.}}
(\byear{2026}).
\btitle{Reflected backward stochastic differential equations with rough
  drivers}.
\bjournal{Stochastic Processes and their Applications}
\bvolume{195}
\bpages{104874}.
\bdoi{10.1016/j.spa.2026.104874}
\end{barticle}
\endbibitem

\bibitem{LiangTang2025}
\begin{barticle}[author]
\bauthor{\bsnm{Liang},~\bfnm{Jiahao}\binits{J.}} \AND
  \bauthor{\bsnm{Tang},~\bfnm{Shanjian}\binits{S.}}
(\byear{2025}).
\btitle{Multidimensional Backward Stochastic Differential Equations with Rough
  Drifts}.
\bjournal{Transactions of the American Mathematical Society}
\bvolume{378}
\bpages{201--257}.
\bdoi{10.1090/tran/9237}
\end{barticle}
\endbibitem

\bibitem{LyonsVictoir2007}
\begin{barticle}[author]
\bauthor{\bsnm{Lyons},~\bfnm{Terry}\binits{T.}} \AND
  \bauthor{\bsnm{Victoir},~\bfnm{Nicolas}\binits{N.}}
(\byear{2007}).
\btitle{An extension theorem to rough paths}.
\bjournal{Annales de l'Institut Henri Poincar{\'e} C, Analyse non lin{\'e}aire}
\bvolume{24}
\bpages{835--847}.
\bdoi{10.1016/j.anihpc.2006.07.004}
\end{barticle}
\endbibitem

\bibitem{QianTudor2011}
\begin{barticle}[author]
\bauthor{\bsnm{Qian},~\bfnm{Zhongmin}\binits{Z.}} \AND
  \bauthor{\bsnm{Tudor},~\bfnm{Jan}\binits{J.}}
(\byear{2011}).
\btitle{Differential structure and flow equations on rough path space}.
\bjournal{Bulletin des Sciences Math{\'e}matiques}
\bvolume{135}
\bpages{695--732}.
\bdoi{10.1016/j.bulsci.2011.07.011}
\end{barticle}
\endbibitem

\bibitem{SongZhangZhang2025I}
\begin{bmisc}[author]
\bauthor{\bsnm{Song},~\bfnm{Jian}\binits{J.}},
  \bauthor{\bsnm{Zhang},~\bfnm{Huilin}\binits{H.}} \AND
  \bauthor{\bsnm{Zhang},~\bfnm{Kuan}\binits{K.}}
(\byear{2025}).
\btitle{Backward stochastic differential equations with nonlinear {Young}
  driver}.
\bnote{arXiv:2504.18632}.
\bdoi{10.48550/arXiv.2504.18632}
\end{bmisc}
\endbibitem

\bibitem{SongZhangZhang2025II}
\begin{bmisc}[author]
\bauthor{\bsnm{Song},~\bfnm{Jian}\binits{J.}},
  \bauthor{\bsnm{Zhang},~\bfnm{Huilin}\binits{H.}} \AND
  \bauthor{\bsnm{Zhang},~\bfnm{Kuan}\binits{K.}}
(\byear{2025}).
\btitle{Backward stochastic differential equations with nonlinear {Young}
  drivers {II}}.
\bnote{arXiv:2509.05183}.
\bdoi{10.48550/arXiv.2509.05183}
\end{bmisc}
\endbibitem

\end{thebibliography}
\end{document}